\documentclass [11pt,oneside,a4paper,mathscr]{amsart}

\usepackage{amscd,amsmath,amssymb,euscript}
\usepackage[frame,cmtip,curve,arrow,matrix,line,graph]{xy}
\usepackage{tikz-cd}
\usepackage{bigints}
\usetikzlibrary{matrix}
\usepackage{hyperref}
\usepackage{stmaryrd}
\usepackage[shortlabels]{enumitem}
\usepackage{here}

\title[Categorical DT/PT correspondence and quasi-BPS categories]{The categorical DT/PT correspondence and quasi-BPS categories for local surfaces}
\author{Tudor P\u adurariu and Yukinobu Toda}

\newtheorem{thm}{Theorem}[section]
\newtheorem{cor}[thm]{Corollary}
\newtheorem{conj}[thm]{Conjecture}
\newtheorem{prop}[thm]{Proposition}
\newtheorem{lemma}[thm]{Lemma}

\theoremstyle{definition}
\newtheorem{defn}[thm]{Definition}

\newtheorem{thm*}[thm]{Theorem$^*$}
\newtheorem{remark}[thm]{Remark}

\newtheorem{step}{Step}
\newcommand{\comment}[1]{}

\renewcommand{\leq}{\leqslant}
\renewcommand{\geq}{\geqslant}

\newcommand{\OO}{\mathcal{O}}

\newcommand{\X}{\mathcal{X}}

\newcommand{\Coh}{\operatorname{Coh}}

\newcommand{\id}{\operatorname{id}}
\newcommand{\Ext}{\operatorname{Ext}}
\newcommand{\Hom}{\operatorname{Hom}}

\newcommand{\Spec}{\operatorname{Spec}}

\newcommand{\GL}{\operatorname{GL}}
\newcommand{\ch}{\operatorname{ch}}

\newcommand{\ssslash}{/\!\!/}
\newcommand{\Tr}{\mathop{\mathrm{Tr}}\nolimits}

\usetikzlibrary{positioning,fit,calc}
\tikzstyle{block}=[draw=black, width=1cm, minimum height=2cm, align=center] 
\tikzstyle{block2}=[draw=black, text width=2cm, minimum height=1cm, align=center] 
\tikzstyle{block3}=[draw=black, text width=2cm, minimum height=1cm, align=center]

\makeatletter

\@addtoreset{equation}{section}	
\makeatother

\begin{document}
\maketitle
\begin{abstract}
We construct
semiorthogonal decompositions of Donaldson-Thomas (DT) categories for reduced curve classes on
local surfaces 
into products of quasi-BPS categories and Pandharipande-Thomas (PT) categories, 
giving a categorical analogue of the numerical DT/PT correspondence for Calabi-Yau 3-folds. 
The main ingredient is a categorical 
wall-crossing formula for DT/PT quivers
(which appear as Ext-quivers in the DT/PT wall-crossing)
proved in our previous paper. We also study quasi-BPS categories of points on  local surfaces and propose conjectural computations of their K-theory analogous to formulas already known for the three dimensional affine space. 
\end{abstract}

\section{Introduction}

\subsection{The DT/PT correspondence for Calabi-Yau 
3-folds}
For a Calabi-Yau 3-fold $X$, the DT/PT 
correspondence is a formula which relates
rank one Donaldson-Thomas (DT) invariants 
counting 
one-dimensional closed subschemes in $X$ 
with Pandharipande-Thomas (PT) invariants 
counting stable pairs on $X$. 
The PT invariants provide a geometric description of the contribution of sheaf counting invariants in the 
Maulik--Nekrasov--Okounkov--Pandharipande conjecture \cite{MNOP}. 
More precisely, 
for a curve class $\beta \in H_2(X, \mathbb{Z})$ and $n \in \mathbb{Z}$, 
we denote by 
\begin{align}\label{IP}
    I_X(\beta, n), \ 
  (\mbox{resp.~} P_X(\beta, n))
    \end{align}
the moduli space of (compactly supported) closed
subschemes $C \subset X$ with $[C]=\beta$, 
$\chi(\mathcal{O}_C)=n$, 
(resp.~stable
pairs $(F, s)$ with $[F]=\beta$, $\chi(F)=n$). 
By taking the integration 
of the Behrend constructible functions on the spaces \eqref{IP}, 
we obtain the DT invariant
$\mathrm{DT}_{\beta, n} \in \mathbb{Z}$
and the PT invariant
$\mathrm{PT}_{\beta, n} \in \mathbb{Z}$. 
The DT/PT correspondence is the formula 
\begin{equation}\label{DTPT}
    \text{DT}_{\beta, n}=\sum_{k\geq 0}\text{DT}_{0, k}\cdot\text{PT}_{\beta, n-k}.
\end{equation}

The above formula was conjectured in~\cite{MR2545686}, 
its unweighted version 
was proved in~\cite{MR2669709, MR2869309}
using Joyce's motivic Hall algebra machinery~\cite{Joy1, Joy2, Joy3, Joy4}, and its version incorporating the Behrend function was proved in~\cite{Br} (also see~\cite{Thall})
using the work of 
Joyce-Song~\cite{JS}
(see also~\cite{K-S}).


The purpose of this paper is to prove a categorical version of \eqref{DTPT} for $X$ a local surface, see Theorem \ref{thm:main}. In \cite{T}, the second author defined dg-categories categorifying DT and PT invariants for local surfaces,
which are total spaces of canonical 
line bundles on surfaces. 
For reduced curve classes, we construct semiorthogonal decompositions of DT categories with summands which are products of quasi-BPS categories on points on $X$ and PT categories. By taking the Grothendieck groups of the categories appearing in these semiorthogonal decompositions, we obtain a K-theoretic version of \eqref{DTPT}.

\subsection{The categorical DT/PT correspondence}
Let $S$ be a smooth projective surface. The local surface is the 
following non-compact Calabi-Yau 3-fold
\begin{align*}
	X:= \mathrm{Tot}_S(K_S). 
	\end{align*}
Take $\beta \in H_2(S, \mathbb{Z})$
and $n \in \mathbb{Z}$. 
In \cite[Definition 4.2.1]{T}, the second author defined 
dg-categories 
\begin{equation}\label{def:dgcatDTPT}
\mathcal{DT}_X(\beta, n)\text{ and }  
		\mathcal{PT}_X(\beta, n).
		\end{equation}
These categories may be regarded as a gluing of 
categories of graded matrix factorizations of 
super-potentials whose critical loci are 
locally isomorphic to 
$I_X(\beta, n)$ and $P_X(\beta, n)$, 
respectively. 
The dg-categories (\ref{def:dgcatDTPT}) should 
recover DT/PT invariants by 
taking their periodic cyclic homologies, 
see~\cite[Theorem~1.1]{Eff}, \cite[Subsection 3.3]{T}, especially \cite[Conjecture 3.3.4]{T}.
The DT/PT categories \eqref{def:dgcatDTPT} 
are constructed as quotients of 
$D^b(\mathfrak{M}^{\dag})$, where $\mathfrak{M}^{\dag}$
is a certain derived open substack of 
the moduli stack of pairs $(F, s)$, where 
$F \in \Coh(S)$ is at most one-dimensional and 
$s \colon \mathcal{O}_S \to F$
is a section, by the subcategory of 
objects whose singular supports~\cite{AG} are 
contained in unstable loci. We will 
review their constructions in Subsections~\ref{subsec:ssupport}
and~\ref{sub:dtptdef}. 
	
	We say that $\beta$ is a reduced class 
	if any effective divisor on $S$ with class 
$\beta$ is a reduced divisor. 
In this case, the second author constructed 
a fully-faithful functor from the PT category to the
DT category, see~\cite[Theorem~1.4.5]{T}. 
	The main theorem of this paper is
	the following, which refines the 
	above fully-faithful functor to 
	a semiorthogonal decomposition: 
	\begin{thm}\label{thm:main}
	Suppose that $\beta$ is a reduced class. 
	Then there is a semiorthogonal decomposition
			\begin{align*}
			\mathcal{DT}_X(\beta, n)
			=
			\left\langle  \bigotimes_{i=1}^k\mathbb{T}_X(d_i)_{v_i}\otimes \mathcal{PT}_{X}(\beta, n') \,\Big|
			-1<\frac{v_1}{d_1}<\cdots<\frac{v_k}{d_k} \leq 0  \right\rangle,
		\end{align*}
where the right hand side is after all partitions $d_1+\cdots+d_k+n'=n$. 
		\end{thm}
		There is also a version for arbitrary classes $\beta$ and for the DT/PT categories of reduced supports 
		over $S$, see Remark \ref{remark:reducedcurve}.
In the above, $\mathbb{T}_X(d)_w$ is the quasi-BPS category of $d$ points and weight $w$ of $X$, see Subsection \ref{sub:quasiBPS} for more details. 

The main ingredient in the proof of Theorem \ref{thm:main} is the \textit{local categorical DT/PT correspondence} for DT/PT quivers proved in \cite[Theorem 1.1]{PT3}. 
The wall crossing between ideal sheaves and stable pairs (on an arbitrary Calabi-Yau $3$-fold) can be described using Ext quivers with super-potential \cite{Todstack}, and all these Ext quivers can be constructed from 
DT/PT quivers, see Subsection \ref{sub:DTPTquivers}. We also make use of the Koszul equivalence \cite{I} between coherent sheaves on a Koszul (quasi-smooth) stack and a category of matrix factorizations, the window categories of Segal \cite{MR2795327}, Halpern-Leistner \cite{halp}, and Ballard--Favero--Katzarkov \cite{MR3895631}, the noncommutative resolutions of \v{S}penko--Van den Bergh \cite{SVdB}, and categorical Hall algebras of surfaces \cite{PoSa} and of quivers with potential \cite{P0}.		

We expect that the semiorthogonal decomposition in 
Theorem~\ref{thm:main} 
also holds for an arbitrary class 
$\beta$. 
The reason we restrict to a reduced class is that, 
in this case, 
 one can take $\mathfrak{M}^{\dag}$
to be the derived moduli stack
$\mathfrak{T}_S(\beta, n)$
of pairs $(F, s)$
such that $s$ has at most zero-dimensional 
cokernel. Then its classical 
truncation $\mathcal{T}_S(\beta, n)$
admits a good moduli space 
\[\mathcal{T}_S(\beta, n)\to T_S(\beta, n),\]
and we can reduce the statement to be proved to the analogous formal 
local statements on $T_S(\beta, n)$. 
For an arbitrary class $\beta$, 
we cannot find such a $\mathfrak{M}^{\dag}$
admitting a good moduli space in general. 
Alternatively, 
we expect an 
argument reducing to a formal local statement on 
$T_X(\beta, n)$, but the local structure of 
the categories (\ref{def:dgcatDTPT}) over $T_X(\beta, n)$
is hard to investigate and
we can prove Theorem~\ref{thm:main} only for reduced 
classes at this moment. 

For $\beta=0$, we obtain the following semiorthogonal decomposition of the DT category of points on $X$ in quasi-BPS categories:

\begin{cor}\label{MacMahonsurface}
Let $n\in\mathbb{N}$. There is a semiorthogonal decomposition
			\begin{align}\label{SODMacMahonsurface}
			\mathcal{DT}_X(0, n)=
			\left\langle  \bigotimes_{i=1}^k\mathbb{T}_X(d_i)_{v_i} \,\Big|
			-1<\frac{v_1}{d_1}<\cdots<\frac{v_k}{d_k} \leq 0\right\rangle, 
		\end{align}
where the right hand side contains all partitions $d_1+\cdots+d_k=n$. 
\end{cor}

We proved Corollary \ref{MacMahonsurface} for $\mathbb{C}^3$ in \cite[Theorem 1.1]{PT0}. We explain in Subsection \ref{sub:quasiBPS} that Corollary \ref{MacMahonsurface} is a categorical analogue of the computation of DT invariants of points on a $3$-fold of Behrend--Fantechi \cite{BF}. 
By taking the Grothendieck groups of the categories in Theorem \ref{thm:main} and Corollary \ref{MacMahonsurface}, we obtain an analogue of \eqref{DTPT} in K-theory: 

\begin{cor}\label{cor410}
There is an isomorphism
\begin{equation}\label{eq:cor410}
K(\mathcal{DT}_X(\beta, n))\cong \bigoplus_{n'\leq n}K\left(\mathcal{DT}_{X}(0, n-n')\otimes \mathcal{PT}_{X}(\beta, n')\right).
\end{equation}
\end{cor}

\subsection{Quasi-BPS categories}\label{sub:quasiBPS}
Let $S$ be a smooth (not necessarily projective) surface.
For the local surface $X=\mathrm{Tot}_S(K_S)$, the first author defined quasi-BPS categories $\mathbb{T}_X(d)_w$ of $d$ points and  weight $w$ in \cite{P2}. We briefly explain the relation of these categories to (numerical, cohomological) BPS (also called Gopakumar-Vafa) invariants.

There are semiorthogonal decompositions of the Hall algebra of points on a surface of Porta--Sala \cite{PoSa} with summands which are products of quasi-BPS categories \cite[Theorem~1.1]{P2}. This is a categorical analogue of the Davison--Meinhardt PBW theorem for the cohomological Hall algebra (CoHA) of a quiver with potential \cite{DM} which says that the CoHA is generated by the (cohomological) BPS spaces of the quiver with potential. 

The numerical BPS invariants are expected to be fundamental enumerative invariants which determine DT/PT invariants and Gromow-Witten invariants \cite[Section 2 and a half]{MR3221298}.
The BPS invariants of $d$ points on $X$ are given by 
\begin{equation}\label{bpsX}
    \Omega_{X, d}=-\chi_c(S),
\end{equation}
where $\chi_c(S)$ is the Euler characteristic of the compactly supported cohomology of $S$.
It is an interesting problem to compute the K-theory of quasi-BPS categories.
In \cite{PT0}, we discussed two K-theoretic computations for quasi-BPS categories of $\mathbb{C}^3$: of localized equivariant K-theory (with respect to the two dimensional Calabi-Yau torus of $\mathbb{C}^3$) and of topological K-theory. 

Assume $S$ is a toric surface. Let $T$ be the two dimensional torus acting on $S$. Let $\mathbb{K}:=K_0(BT)$, $\mathbb{F}:=\mathrm{Frac}\,\mathbb{K}$, and $K_T\left(\mathbb{T}_X(d)_w\right)_\mathbb{F}:=K_T\left(\mathbb{T}_X(d)_w\right)\otimes_\mathbb{K}\mathbb{F}$. The following is the analogue of \cite[Theorem 4.12]{PT0}:

\begin{thm}\label{thmtoric}
Let $S$ be a toric surface, let $X=\mathrm{Tot}_S(K_S)$, let $(d, w)\in\mathbb{N}\times\mathbb{Z}$, and let $n=\gcd(d, w)$. Then
\[\dim_\mathbb{F} K_T\left(\mathbb{T}_X(d)_w\right)_\mathbb{F}=\chi_c(\mathrm{Hilb}(S, n)).\]
\end{thm}

We next discuss topological K-theoretic computations.
For a dg-category $\mathcal{D}$, consider its topological K-theory $K_i^{\rm{top}}(\mathcal{D})$
defined by Blanc \cite{Blanc}, which
is two-periodic by its construction~\cite[Definition 1.6]{Blanc}.
Let \[\chi_K(\mathbb{T}_X(d)_w):=\dim_{\mathbb{Q}} K_0^{\mathrm{top}}(\mathbb{T}_X(d)_w)_{\mathbb{Q}}-\dim_{\mathbb{Q}} K_1^{\mathrm{top}}(\mathbb{T}_X(d)_w)_{\mathbb{Q}}.\]
We propose the following K-theoretic computations, see \cite[Corollary 5.13]{PT0} for the case of $\mathbb{C}^3$.

\begin{conj}\label{mainconj}
Let $S$ be a smooth surface, let $X=\mathrm{Tot}_S(K_S)$,
let $(d, w)\in \mathbb{N}\times\mathbb{Z}$, and let $n=\gcd(d, w)$. Then 
\[\chi_K(\mathbb{T}_X(d)_w)=\chi_c(\mathrm{Hilb}(S, n)).\]
\end{conj}

When $\gcd(d, w)=1$, we regard $\mathbb{T}_X(d)_w$ as categorifications of the BPS invariants $\Omega_{X,d}$, for example Conjecture \ref{mainconj} implies that $\chi_K(\mathbb{T}_X(d)_w)=|\Omega_{X,d}|$. 
Further, when $\gcd(d, w)=1$, we show (using \cite[Theorem 1.1]{PT1}) that objects in $\mathbb{T}_X(d)_w$ are supported over the small diagonal $X\hookrightarrow \mathrm{Sym}^d(X)$, which is a categorical analogue of Davison's support lemma \cite[Lemma 4.2, Theorem 5.1]{Dav}. 

For $\gcd(d,w)>1$, the Euler characteristic of the topological K-theory of $\mathbb{T}_X(d)_w$ may be different from $|\Omega_{X,d}|$.
In \cite{PT1}, we constructed a coproduct on the K-theory of quasi-BPS categories of $\mathbb{C}^3$ of fixed slope and we used it to define a K-theoretic version of BPS invariants for all $(d,w)\in\mathbb{N}\times\mathbb{Z}$. We do not know if this is possible for a general local surface $X$.

Let $(d, w)\in\mathbb{N}\times\mathbb{Z}$ with $\gcd(d, w)=n$.
Similarly to our expectations for $S=\mathbb{C}^2$ from \cite{PT0}, \cite{PT1}, we believe the category $\mathbb{T}_X(d)_w$ is related to the category $\mathrm{MF}^{\mathrm{gr}}(X^{\times d}/\mathfrak{S}_d, 0)$ of zero matrix factorizations on the orbifold $X^{\times d}/\mathfrak{S}_d$. 

\subsection{Computations of K-theoretic DT invariants via quasi-BPS categories}\label{sub:KDT}

We revisit the computation of DT invariants from BPS invariants. 
By \cite{BF}, 
there is an equality
\begin{align}\label{wallcrossingDT}
	\sum_{d \geq 0} \mathrm{DT}_{X,d}\cdot q^d=
	\prod_{d\geq 1}(1-(-q)^d)^{d \Omega_{X, d}}=
	\prod_{d\geq 1}(1-(-q)^d)^{-d\chi_c(S)}.
	\end{align}
	
In the spirit of \cite{PT0},  \cite[Subsection 5.5]{PT1}, we regard Corollary \ref{MacMahonsurface} as a (partial) categorical analogue of the above formula.
To obtain a K-theoretic version of \eqref{wallcrossingDT} from Theorem \ref{MacMahonsurface}, one needs a K-theoretic computation of quasi-BPS categories as in Conjecture \ref{mainconj}. More precisely, assuming Conjecture \ref{mainconj} and the following Künneth formula \begin{equation}\label{kunnethchi}
    \chi_K\left(\bigotimes_{i=1}^k \mathbb{T}_X(d_i)_{v_i}\right)=\prod_{i=1}^k \chi_K(\mathbb{T}_X(d_i)_{v_i}),\end{equation} 
    we obtain from Corollary~\ref{MacMahonsurface} the following computation of the Euler characteristic of the topological K-theory of DT categories of points on a local surface:
    \begin{equation}\label{wallcrossingDT2}
    \sum_{d \geq 0} \chi_K(\mathcal{DT}_X(0, d))\cdot q^d=
	\prod_{d\geq 1}(1-q^d)^{-d\chi_c(S)},
\end{equation}
see for example \cite[Corollary 4.13]{PT0}, \cite[Subsection 5.5]{PT1}. We discuss an analogous statement for the localized equivariant K-theory of toric surfaces in Remark \ref{sub:catMM}.

\subsection{Acknowledgements}
	Y.~T.~is supported by World Premier International Research Center
	Initiative (WPI initiative), MEXT, Japan, and Grant-in Aid for Scientific
	Research grant (No.~19H01779) from MEXT, Japan.
	
\section{Preliminaries}

\subsection{Notations}\label{notation}

All the spaces $\mathcal{X}$ considered are quasi-smooth (derived) stacks over $\mathbb{C}$, see Subsection \ref{sec:quasismooth}. 
Its classical truncation is denoted by $t_0(\mathcal{X})$. 
We denote by $\mathbb{L}_\mathcal{X}$ the cotangent complex of 
$\mathcal{X}$. 
For a quasi-projective scheme $X$ over $\mathbb{C}$, we denote by $\chi_c(X)$
the Euler characteristic of its compactly supported cohomology. 

In this paper, any dg-category considered is a $\mathbb{C}$-linear 
pre-triangulated dg-category, in particular its homotopy category is a 
triangulated category. 
For a dg-category $\mathscr{D}$, we denote by $K(\mathscr{D})$ the Grothendieck group of the homotopy category of $\mathscr{D}$. Let $\mathcal{X}$ be a stack with an action of an abelian group $T$. We denote by $G_T(\X)$ the Grothendieck group of the derived category of bounded complexes of $T$-equivariant coherent sheaves $D^b_T(\X)$. We denote by $K_T(\X)$ the Grothendieck group of the category of perfect complexes $\mathrm{Perf}_T(\X)\subset D^b_T(\X)$. If $T$ is trivial, we drop it from the notation. 
We let $\mathbb{K}:=K_0(BT)$ and $\mathbb{F}:=\mathrm{Frac}\,\mathbb{K}$. If $V$ is a $\mathbb{K}$-module, we let $V_\mathbb{F}:=V\otimes_\mathbb{K}\mathbb{F}$. 

For $G$ a reductive group and $X$ a dg-scheme with an action of $G$, denote by $X/G$ the corresponding
quotient stack. 
When $X$ is affine, we denote by $X\ssslash G$ the quotient dg-scheme with dg-ring of regular functions $\mathcal{O}_X^G$. 
For a morphism $f \colon X\to Y$ and for a closed point $y \in Y$, 
we denote by $\widehat{X}_y$ the formal fiber of $f$ at $y$, i.e. 
\begin{align*}
\widehat{X}_y := X \times_{Y} \Spec \left(\widehat{\mathcal{O}}_{Y, y}\right).
\end{align*}
When $X$ is a $G$-representation, $f\colon X\to Y:=X\ssslash G$, and $y=0$, we omit the subscript $y$ from the above notation. 





Let $R$ be a set. Consider a set $O\subset R\times R$ such that for any $i, j\in R$ we have $(i,j)\in O$, or $(j,i)\in O$, or both $(i,j)\in O$ and $(j,i)\in O$. 
Let $\mathbb{T}$ be a pre-triangulated dg-category. We will construct semiorthogonal decompositions
\[\mathbb{T}=\langle \mathbb{A}_i \mid i \in R \rangle\] with summands 
pre-triangulated subcategories $\mathbb{A}_i$ indexed by $i\in R$
such that for any $i,j\in R$ with $(i, j)\in O$ and for any objects $\mathcal{A}_i\in\mathbb{A}_i$, $\mathcal{A}_j\in\mathbb{A}_j$, we have 
$\Hom_{\mathbb{T}}(\mathcal{A}_i,\mathcal{A}_j)=0$.

\subsection{Matrix factorizations}

Reference for this subsection are \cite[Section 2.2]{T3}, \cite[Subsection~2.2]{T4}, \cite[Section~2.3]{MR3895631}, \cite[Section~1]{PoVa3}.

\subsubsection{The definition of categories of matrix factorizations}\label{sub:ungradedMF} Let $G$ be a reductive group and let $Y$ be a smooth affine scheme with an action of $G$
and a trivial $\mathbb{Z}/2$-action.
Let 
$\mathcal{Y}=Y/G$ be the corresponding quotient stack
and let $f$ be
a regular function 
\[f \colon \mathcal{Y}\to \mathbb{C}.\]
We define the dg-category of matrix factorizations $\mathrm{MF}(\mathcal{Y}, f)$. 
Its objects are $(\mathbb{Z}/2\mathbb{Z})\times G$-equivariant factorizations $(P, d_P)$, where $P$ is a $G$-equivariant coherent sheaf on $Y$, 
$\langle 1\rangle$ is the twist corresponding to a non-trivial $\mathbb{Z}/2\mathbb{Z}$-character on $Y$, and \[d_P \colon  P\to P\langle 1\rangle\] is a morphism satisfying $d_P\circ d_P=f$. Alternatively, the objects of $\mathrm{MF}(\mathcal{Y}, f)$ are tuples
\begin{align}\label{MF:obj}
\left(\alpha \colon E\rightleftarrows F \colon  \beta\right),
\end{align}
where $E$ and $F$ are $ G$-equivariant coherent sheaves
on $Y$, and $\alpha$ and $\beta$ are $G$-equivariant morphisms
such that $\alpha\circ\beta$ and $\beta\circ\alpha$ are multiplication by $f$.
See~\cite[Section~2.3]{MR3895631}, \cite[Subsection~2.2]{T4}, \cite[Definition 1.2]{PoVa3} for the definition of the morphism spaces.

For a dg-subcategory $\mathscr{A}$ of $D^b(\mathcal{Y})$, define $\mathrm{MF}(\mathscr{A}, f)$ as the full subcategory of $\mathrm{MF}(\mathcal{Y}, f)$ with objects pairs $(P, d_P)$ with $P$ in $\mathscr{A}$. 

\subsubsection{Graded matrix factorizations}\label{gradedMFdef}
Assume there exists an extra action of $\mathbb{C}^*$ on $Y$ which commutes with the action of $G$ on $Y$, 
and trivial on $\mathbb{Z}/2 \subset \mathbb{C}^{\ast}$.
Assume that $f$ is weight two with respect to 
the above $\mathbb{C}^{\ast}$-action. 
Denote by $(1)$ the twist by the character \[\mathrm{pr}_2 \colon G\times\mathbb{C}^*\to\mathbb{C}^*.\]
Consider the category of graded matrix factorizations $\mathrm{MF}^{\mathrm{gr}}(\mathcal{Y}, f)$. Its objects are pairs $(P, d_P)$ with $P$ an equivariant $G\times\mathbb{C}^*$-sheaf on $Y$ and $d_P \colon P\to P(1)$ a $G\times\mathbb{C}^*$-equivariant morphism. 
Note that
as the $\mathbb{C}^{\ast}$-action is trivial on $\mathbb{Z}/2$, 
we have the induced action of
$\mathbb{C}^{\star}=\mathbb{C}^{\ast}/(\mathbb{Z}/2)$ on $Y$
and $f$ is weight one with respect to the above $\mathbb{C}^{\star}$-action. 
The objects of $\mathrm{MF}^{\mathrm{gr}}(\mathcal{Y}, f)$ can be alternatively described as tuples 
\begin{align}\label{tuplet:graded}
(E, F, \alpha \colon 
E\to F(1)', \beta \colon F\to E),
\end{align}
where $E$ and $F$ are $G\times\mathbb{C}^{\star}$-equivariant coherent sheaves
on $Y$, $(1)'$ is the twist by the character 
$G \times \mathbb{C}^{\star} \to \mathbb{C}^{\star}$, 
and $\alpha$ and $\beta$ are $\mathbb{C}^{\star}$-equivariant morphisms
such that $\alpha\circ\beta$ and $\beta\circ\alpha$ are multiplication by $f$.

Functoriality of categories of (graded or ungraded) matrix factorizations for pullback and proper pushfoward is discussed in \cite{PoVa3}.

\subsection{The Koszul equivalence}

Let $Y$ be a smooth affine scheme with an action of a reductive group $G$, let $\mathcal{Y}=Y/G$, and let $V$ be a $G$-equivariant vector bundle on $Y$. Let $\mathbb{C}^*$ act on the fibers of $V$ with weight $2$ and consider $s\in \Gamma(Y, V)$ a section of $V$ of $\mathbb{C}^*$-weight $2$.
It induces a map $\partial \colon V^{\vee} \to \mathcal{O}_Y$. 
Let $s^{-1}(0)$ be the derived zero locus of $s$ with ring of regular functions
\begin{align}\label{Kscheme}
\mathcal{O}_{s^{-1}(0)}:=\mathcal{O}_Y\left[V^{\vee}[1];\partial\right].
\end{align}
Consider the quotient 
\begin{equation}\label{koszuldef}
\mathscr{P}:=s^{-1}(0)/G.
\end{equation}
We call $\mathscr{P}$ the \textit{Koszul stack} associated with 
$(Y, V, s, G)$. 
We denote by $D^b(\mathscr{P})$ the 
derived category of $G$-equivariant 
dg-modules over $\mathcal{O}_{s^{-1}(0)}$
with bounded coherent cohomologies. 

The section $s$ also induces the regular function \begin{equation}\label{defreg}
f\colon \mathscr{V}^{\vee}:=\text{Tot}_Y\left(V^{\vee}\right)/G\to\mathbb{C}
\end{equation}
defined by
$f(y,v)=\langle s(y), v \rangle$ for $y\in Y(\mathbb{C})$ and $v\in V^{\vee}|_y$.
Consider the category of graded matrix factorizations $\text{MF}^{\text{gr}}\left(\mathscr{V}^{\vee}, f\right)$ with respect to the group $\mathbb{C}^*$ mentioned above. The Koszul duality equivalence, also called dimensional reduction in the literature, says the following:

\begin{thm}\emph{(\cite{I, Hirano, T})}\label{thm:Kduality}
There is an equivalence 
\begin{align}\label{equiv:Phi}
	\Phi \colon D^b(\mathscr{P}) \stackrel{\sim}{\to}
	\mathrm{MF}^{\mathrm{gr}}(\mathscr{V}^{\vee}, f). 
	\end{align}
	The equivalence $\Phi$ is given by $\Phi(-)=\mathcal{K}\otimes_{\mathcal{O}_{\mathscr{P}}}(-)$
	for the Koszul factorization $\mathcal{K}$, see~\cite[Theorem~2.3.3]{T}. 
	\end{thm}

\subsection{Window categories}\label{subsection:window}
\subsubsection{Attracting stacks}\label{attractingloci}
Let $Y$ be an affine variety with an action of a reductive group $G$. Let $\lambda$ be a cocharacter of $G$. 
Let $G^\lambda$ and $G^{\lambda\geq 0}$ be the Levi and parabolic groups associated to $\lambda$. Let $Y^\lambda\subset Y$ be the closed subvariety of $\lambda$-fixed points.
Consider the attracting variety \[Y^{\lambda\geq 0}:=\{y\in Y|\,\lim_{t\to 0}\lambda(t)\cdot y\in Y^\lambda\}\subset Y.\]
Consider the attracting and fixed stacks
\begin{equation}\label{attracting}
\mathscr{Z}:=Y^\lambda/G^\lambda \xleftarrow{q}\mathscr{S}:=Y^{\lambda\geq 0}/G^{\lambda\geq 0}\xrightarrow{p}\mathcal{Y}.
\end{equation}
The map $p$ is proper. 
Kempf-Ness strata are connected components of certain attracting stacks $\mathscr{S}$, and the map $p$ restricted to a Kempf-Ness stratum is a closed immersion, see~\cite[Section~2.1]{halp}. 
The attracting stacks also appear in the definition of Hall algebras \cite{P0} (for $Y$ an affine space), where the Hall product is induced by the functor 
\begin{equation}\label{Hallproductquiver}
    \ast:=p_*q^*\colon D^b(\mathscr{Z})\to D^b(\mathcal{Y}).
\end{equation}
In this case, the map $p$ may not be a closed immersion.

Let $T \subset G$ be a maximal torus and 
$\lambda$ is a cocharacter $\lambda \colon \mathbb{C}^{\ast} \to T$. 
For a $G$-representation $Y$, 
the attracting variety
$Y^{\lambda \geq 0} \subset Y$
coincides with the sub $T$-representation
generated by weights which pair non-negatively with $\lambda$. 
We denote by 
$\langle \lambda, Y^{\lambda \geq 0} \rangle:= \langle \lambda, \det Y^{\lambda \geq 0} \rangle$, 
where $\det Y^{\lambda \geq 0}$ is the sum of $T$-weights of $Y^{\lambda \geq 0}$.

\subsubsection{The definition of window categories}
Let $Y$ be an affine variety with an action of a reductive group $G$ and a linearization $\mathscr{L}$. Consider the stacks 
\[j\colon\mathcal{Y}^{\mathscr{L}\text{-ss}}:=Y^{\mathscr{L}\text{-ss}}/G\hookrightarrow\mathcal{Y}:=Y/G.\]
We review the construction of window categories of $D^b(\mathcal{Y})$ which are equivalent to $D^b(\mathcal{Y}^{\mathscr{L}\text{-ss}})$ via the restriction map, due to
Segal \cite{MR2795327}, Halpern-Leistner \cite{halp}, and Ballard--Favero--Katzarkov \cite{MR3895631}. We follow the presentation from \cite{halp}.

By also fixing a Weyl-invariant norm on the 
cocharacter lattice, 
the unstable locus $\mathcal{Y}\setminus \mathcal{Y}^{\mathscr{L}\text{-ss}}$ has a stratification in Kempf-Ness strata $\mathscr{S}_i$ for $i\in I$ a finite ordered set:
\[\mathcal{Y}\setminus \mathcal{Y}^{\mathscr{L}\text{-ss}}=\bigsqcup_{i\in I}\mathscr{S}_i.\]
A Kempf-Ness stratum $\mathscr{S}_i$ is the attracting stack
in $\mathcal{Y} \setminus \sqcup_{j<i}\mathcal{S}_j$
for
a cocharacter $\lambda_i$, with the fixed stack $\mathscr{Z}_i:=\mathscr{S}_i^{\lambda_i}$. 
Let $N_{\mathscr{S}_i/\mathcal{Y}}$ be the normal bundle of $\mathscr{S}_i$ in $\mathcal{Y}$.
 Define the width of the window categories \[\eta_i:=\left\langle \lambda_i, N^{\vee}_{\mathscr{S}_i/\mathcal{Y}}|_{\mathscr{Z}_i}
 \right\rangle.\]
  For a choice of real numbers $w=(w_i)_{i\in I}\in \mathbb{R}^I$, define the category
\begin{equation}\label{def:window}
\mathbb{G}_w:=\{\mathcal{F}\in D^b(\mathcal{Y})\text{ such that } \mathrm{wt}_{\lambda_i}(\mathcal{F}|_{\mathscr{Z}_i})\subset [
w_i, w_i+\eta_i) \text{ for all }i\in I\}.
\end{equation}
In the above, $\mathrm{wt}_{\lambda_i}(\mathcal{F}|_{\mathscr{Z}_i})$
is the set of $\lambda_i$-weights on $\mathcal{F}|_{\mathscr{Z}_i}$.
Then \cite[Theorem 2.10]{halp}
says that the restriction functor 
$j^*$ induces an equivalence of categories: 
\begin{equation}\label{jequiv}
j^*\colon \mathbb{G}_w\xrightarrow{\sim} D^b\big(\mathcal{Y}^{\mathscr{L}\text{-ss}}\big)
\end{equation}
for any choice of real numbers $w=(w_i)_{i\in I}\in \mathbb{R}^I$.

We discuss a slight extension of the above theorem for matrix factorizations. Let $f\colon \mathcal{Y}\to \mathbb{C}$ be a regular function and recall the definition of $\mathrm{MF}(\mathbb{G}_w, f)$ from Subsection \ref{sub:ungradedMF}. Then there is an equivalence 
\begin{equation}\label{jequiv2}
j^*\colon \mathrm{MF}(\mathbb{G}_w, f)\xrightarrow{\sim} \mathrm{MF}\big(\mathcal{Y}^{\mathscr{L}\text{-ss}}, f\big).
\end{equation}
The analogous statement holds for categories of graded matrix factorizations.

\subsection{Quasi-smooth stacks and singular support}\label{sec:quasismooth}

References for this section are \cite[Subsections 3.1 and 3.2.1]{T}, \cite{AG}.

\subsubsection{Quasi-smooth stacks}\label{subsub:qsmooth} Let $\mathfrak{M}$ be a derived stack over $\mathbb{C}$ and let $\mathcal{M}$ be its classical truncation. Let $\mathbb{L}_{\mathfrak{M}}$ be the cotangent complex 
of $\mathfrak{M}$. The stack $\mathfrak{M}$ is called \textit{quasi-smooth} if for all closed points $x\to \mathcal{M}$, the restriction $\mathbb{L}_\mathfrak{M}|_x$ has cohomological amplitude in $[-1, 1]$.
Examples of quasi-smooth stacks are the Koszul stacks $\mathscr{P}$ from \eqref{koszuldef} or the moduli stacks of Gieseker semistable compactly supported sheaves on a smooth surface $S$. 

By \cite[Theorem 2.8]{BBBJ}, a stack $\mathfrak{M}$ is quasi-smooth if and only if it is a $1$-stack and any point of $\mathfrak{M}$ lies in the image of a $0$-representable smooth morphism \begin{equation}\label{alpha}
    \alpha \colon \mathscr{U}\to\mathfrak{M}
\end{equation} for a Koszul scheme $\mathscr{U}$ as in \eqref{Kscheme}.
The dg-category $D^b(\mathfrak{M})$ is defined to be the limit
\begin{align*}
    D^b(\mathfrak{M})=\lim_{\mathscr{U} \to \mathfrak{M}} D^b(\mathscr{U}). 
\end{align*}

Suppose that $\mathcal{M}$ admits a good moduli space
$\mathcal{M} \to M$, 
see~\cite{MR3237451} for the notion of good moduli space.
For each point in $M$, there is an \'{e}tale
  neighborhood $U \to M$
  and Cartesian squares 
  \begin{align}\label{dia:fnbd}
	\xymatrix{
		\mathfrak{M}_U \ar[d] \ar@{}[rd]|\square
		& \ar@<-0.3ex>@{_{(}->}[l] \mathcal{M}_U
		\ar[r] \ar[d] \ar@{}[rd]|\square & U \ar[d] \\
		\mathfrak{M}  & \ar@<-0.3ex>@{_{(}->}[l] \mathcal{M} \ar[r] & M
	}
\end{align}
where each vertical arrow is \'{e}tale
and $\mathfrak{M}_U$ is equivalent to a Koszul stack 
$\mathscr{P}=s^{-1}(0)/G$ as in (\ref{koszuldef}), see~\cite[Subsection~3.1.4]{T}, \cite[Theorem 4.2.3]{HalpK32}, \cite{AHR}.

\subsubsection{$(-1)$-shifted cotangent stacks}
Let $\mathfrak{M}$ be a quasi-smooth stack. Let $\mathbb{T}_\mathfrak{M}$ be the tangent complex of $\mathfrak{M}$, which is the dual complex to the cotangent complex $\mathbb{L}_\mathfrak{M}$.
We denote by $\Omega_{\mathfrak{M}}[-1]$ the \textit{$(-1)$-shifted cotangent stack} of $\mathfrak{M}$:
\begin{align*}
\Omega_\mathfrak{M}[-1]:=\mathrm{Spec}_\mathfrak{M}\left(\mathrm{Sym}(\mathbb{T}_\mathfrak{M}[1])\right).
\end{align*}
Consider the projection map 
\begin{equation}\label{p0}
p_0\colon \mathcal{N}:=t_0\left(\Omega_\mathfrak{M}[-1]\right)\to \mathfrak{M}.
\end{equation}
For a Koszul stack $\mathscr{P}$ as in \eqref{koszuldef}, recall the function $f$ from \eqref{defreg} and consider the critical locus $\mathrm{Crit}(f)\subset \mathscr{V}^{\vee}$. 
In this case, the map $p_0$ is the natural projection
\begin{equation}\label{p1}
p_0\colon \mathrm{Crit}(f)=t_0\left(\Omega_\mathscr{P}[-1]\right)\to \mathscr{P}.
\end{equation}

\subsubsection{Singular support}\label{singularsupport}
We continue with the notations from the above subsection. 
Arinkin--Gaitsgory \cite{AG} defined the notion of singular support of an object $\mathcal{F}\in D^b(\mathfrak{M})$, 
denoted by 
\begin{align*}
    \mathrm{Supp}^{\rm{sg}}(\mathcal{F}) \subset \mathcal{N}. 
\end{align*}
The definition is compatible with maps $\alpha$ as in \eqref{alpha}, see \cite[Definition 3.2.1]{T}. 
Consider the group $\mathbb{C}^*$ scaling the fibers of the map $p_0$. 
A closed substack $\mathcal{Z}$ of $\mathcal{N}$ is called \textit{conical} if it is closed under the action of $\mathbb{C}^*$. 
The singular support $\mathrm{Supp}^{\rm{sg}}(\mathcal{F})$ of $\mathcal{F}$ is a 
conical subset $\mathcal{Z}$ of $\mathcal{N}$. 

Consider a Koszul stack $\mathscr{P}$ as in \eqref{koszuldef} and recall the Koszul equivalence $\Phi$ from \eqref{equiv:Phi}. Under $\Phi$, the singular support of $\mathcal{F}$ corresponds to the support $\mathcal{Z}$ of the matrix factorization $\Phi(\mathcal{F})$, namely the maximal closed substack $\mathcal{Z}\subset \mathrm{Crit}(w)$ such that $\mathcal{F}|_{\mathscr{V}^{\vee}\setminus \mathcal{Z}}=0$ in $\mathrm{MF}^{\mathrm{gr}}(\mathscr{V}^{\vee}\setminus \mathcal{Z}, w)$, see \cite[Subsection 2.3.9]{T}.

\subsubsection{DT categories via singular support quotients}\label{subsec:ssupport}
We review the definition of DT categories for quasi-smooth stacks \cite[Section 3.2]{T}. 
Let $\mathfrak{M}$ be a quasi-smooth stack with 
classical truncation $\mathcal{M}$ and 
let $\mathcal{N}:=t_0(\Omega_{\mathfrak{M}}[-1])$. 
Let $\mathscr{L}$ be a line bundle on $\mathcal{M}$.
We regard $\mathscr{L}$ as a line bundle on $\mathcal{N}$ by 
pulling it back via $p_0\colon \mathcal{N} \to \mathcal{M}$. 
We
let $\mathcal{N}^{\mathscr{L}\text{-ss}} \subset \mathcal{N}$ be the open substack of $\mathscr{L}$-semistable 
points, see~\cite[Subsections 2.1 and 2.2, Example 4.5.2]{Halpinstab}
for the definition of semistable points 
with respect to maps from 
the $\Theta$-stack 
$\Theta:=[\mathbb{A}^1/\mathbb{G}_m]$. 
Its complement \[\mathcal{Z}:=\mathcal{N}\setminus \mathcal{N}^{\mathscr{L}\text{-ss}}\subset \mathcal{N}\] is a conical 
closed substack. 
Let $\mathcal{C}_{\mathcal{Z}} \subset D^b(\mathfrak{M})$
be the subcategory consisting of objects with singular supports
contained in $\mathcal{Z}$. The quotient category 
\begin{align}\label{def:catDT}
    D^b(\mathfrak{M})/\mathcal{C}_{\mathcal{Z}}
\end{align}
is a model of the 
DT category for the semistable locus on the $(-1)$-shifted 
cotangent stack, see~\cite[Definition 3.2.2]{T}. 
\begin{remark}
If $\mathfrak{M}=\mathscr{P}$ for a Koszul stack $\mathscr{P}$ as in (\ref{koszuldef}), 
then the quotient category (\ref{def:catDT}) is 
equivalent to $\mathrm{MF}^{\rm{gr}}((\mathscr{V}^{\vee})^{\mathscr{L}\text{-ss}}, f)$, 
see~\cite[Proposition~2.3.9]{T}. 
As a grading is involved on the matrix factorization side, 
the category (\ref{def:catDT}) is called the $\mathbb{C}^{\ast}$-equivariant DT category in \cite{T}. 
The (ungraded) $\mathbb{Z}/2$-periodic version is introduced in~\cite{T4} and it is
given by the quotient 
\begin{align}\label{def:catDT2}
    D^b(\mathfrak{M}_{\varepsilon})/\mathcal{C}_{\mathcal{Z}_{\varepsilon}}. 
\end{align}
Here $\mathfrak{M}_{\varepsilon}=\mathfrak{M}\times \Spec \mathbb{C}[\varepsilon]$
with $\deg \varepsilon=-1$, and 
\begin{align*}
    \mathcal{Z}_{\varepsilon}=\mathbb{C}^{\ast}(\mathcal{Z} \times \{1\})
    \sqcup (\mathcal{N} \times \{0\}) \subset \mathcal{N} \times \mathbb{C}
    =t_0(\Omega_{\mathfrak{M}_{\varepsilon}}[-1]).
\end{align*}
In the case of a Koszul stack $\mathfrak{M}=\mathscr{P}$, the category (\ref{def:catDT2})
is equivalent to the category of (ungraded) matrix factorizations $\mathrm{MF}((\mathscr{V}^{\vee})^{\mathscr{L}\text{-ss}}, f)$. 
The argument used to prove Theorem~\ref{thm:main} can be also used to 
obtain its $\mathbb{Z}/2$-periodic version. 
\end{remark}

\subsection{The window theorem for DT categories}\label{windowfirst}
\subsubsection{Window categories}
We review
the theory of window categories for quasi-smooth stacks \cite[Chapter 5]{T}.
Let $\mathfrak{M}$ be quasi-smooth and assume
throughout this subsection that $\mathcal{M}=t_0(\mathfrak{M})$ admits a good moduli space 
$\mathcal{M} \to M.$
Let $\mathscr{L}$ be a line bundle on $\mathcal{M}$
and take 
a positive definite class $b \in H^4(\mathcal{M}, \mathbb{Q})$, see~\cite[Definition~3.7.6]{Halpinstab}. 
We also use the same symbols $(\mathcal{L}, b)$ for 
$p_0^{\ast}\mathcal{L} \in \mathrm{Pic}(\mathcal{N})$ and 
$p_0^{\ast}b \in H^4(\mathcal{N}, \mathbb{Q})$. 
Then there is a $\Theta$-stratification 
with respect to $(\mathcal{L}, b)$
\begin{align}\label{theta:N}
    \mathcal{N}=\mathscr{S}_1 \sqcup \cdots \sqcup \mathscr{S}_N\sqcup \mathcal{N}^{\mathscr{L}\text{-ss}}
\end{align}
with centers $\mathscr{Z}_i\subset\mathscr{S}_i$, see~\cite[Theorem 5.2.3, Proposition 5.3.3]{Halpinstab}.  
When $\mathcal{M}$ is a (global) quotient stack, 
$\Theta$-stratifications are the same as Kempf-Ness stratifications \cite[Example 0.0.5]{Halpinstab}. 
The class $b$ is then constructed as the pullback of the class in $H^4(BG_y, \mathbb{Q})$ corresponding to chosen positive definite form \cite[Example 5.3.4]{Halpinstab}.
In the above situation, an analogue of the window theorem is 
proved in~\cite[Theorem 1.1]{Totheta}, \cite[Theorem 5.3.13]{T}:
\begin{thm}\emph{(\cite{T, Totheta})}\label{thm:window:M}
In addition to the above, 
suppose that $\mathcal{M} \to M$ satisfies the formal neighborhood theorem. 
Then for each $w=(w_i)_{i=1}^N\in \mathbb{R}^N$, 
 there is a subcategory $\mathbb{W}_w \subset D^b(\mathfrak{M})$
 such that the composition 
 \begin{align*}
     \mathbb{W}_{w} \subset D^b(\mathfrak{M}) \twoheadrightarrow
     D^b(\mathfrak{M})/\mathcal{C}_{\mathcal{Z}}
 \end{align*}
 is an equivalence. 
\end{thm}
We explain the meaning of ``the formal neighborhood theorem" in the statement of the above theorem, see \cite[Definition 5.2.3]{T}. 
For a closed point $y \in M$, denote also by $y \in \mathcal{M}$
the unique closed point in the fiber of $\mathcal{M} \to M$
at $y$.
Set $G_y:=\mathrm{Aut}(y)$,
which is a reductive algebraic group. 
Let $\widehat{\mathcal{M}}_y$ be the formal fiber
along with $\mathcal{M} \to M$ at $y$. 
Then the formal 
neighborhood theorem says that 
there is a $G_y$-equivariant morphism 
\begin{align*}
\kappa_y \colon \widehat{\mathcal{H}}^0(\mathbb{T}_{\mathcal{M}}|_{y}) 
\to \mathcal{H}^1(\mathbb{T}_{\mathcal{M}}|_{y}) 
\end{align*}
such that, by setting $\mathcal{U}_y$
to be the classical zero locus of $\kappa_y$,
there is an isomorphism  
$\widehat{\mathcal{M}}_y \cong \mathcal{U}_y/G_y$. 
In this case, 
there is a (unique up to equivalence) derived stack $\widehat{\mathfrak{M}}_y$ and Cartesian squares
\begin{align*}
   	\xymatrix{
		\widehat{\mathfrak{M}}_y \ar[d] \ar@{}[rd]|\square
		& \ar@<-0.3ex>@{_{(}->}[l] \widehat{\mathcal{M}}_y
		\ar[r] \ar[d] \ar@{}[rd]|\square & 
		\Spec \widehat{\mathcal{O}}_{M, y}\ar[d] \\
		\mathfrak{M}  & \ar@<-0.3ex>@{_{(}->}[l] \mathcal{M} \ar[r] & M.
    }
\end{align*}
We call $\widehat{\mathfrak{M}}_y$ \textit{the formal fiber} of $\mathfrak{M}$ at $y$.
Let $\mathfrak{U}_y$ be the derived zero locus of 
$\kappa_y$. 
Then, by replacing $\kappa_y$ if necessary, 
$\widehat{\mathfrak{M}}_y$ is equivalent to 
$\mathfrak{U}_y/G$, see~\cite[Lemma~5.2.5]{T}. 

\subsubsection{Local description of window categories}
Below we give a formal local 
description of $\mathbb{W}_w$. 
Consider the pair of a smooth
stack and a regular function $(\mathcal{X}_y, f_y)$:
\begin{align*}
    \mathcal{X}_y:=\left(\widehat{\mathcal{H}}^0(\mathbb{T}_{\mathfrak{M}}|_{y})
    \oplus \mathcal{H}^1(\mathbb{T}_{\mathfrak{M}}|_{y})^{\vee}\right)/G_y
    \stackrel{f_y}{\to} \mathbb{C},
\end{align*}
where $f_y(u, v)=\langle \kappa_y(u), v \rangle$. 
From (\ref{p1}), the critical locus of $f_y$
is isomorphic to the classical 
truncation of the $(-1)$-shifted cotangent stack over $\widehat{\mathfrak{M}}_y$, 
so it is isomorphic to the formal fiber 
$\widehat{\mathcal{N}}_y$
of $\mathcal{N} \to \mathcal{M} \to M$
at $y$. 
The pull-back of the $\Theta$-stratification 
(\ref{theta:N}) to $\widehat{\mathcal{N}}_y$
gives a Kempf-Ness stratification 
\begin{align*}
    \widehat{\mathcal{N}}_y=\widehat{\mathscr{S}}_{1, y} \sqcup 
    \cdots \sqcup \widehat{\mathscr{S}}_{N, y} \sqcup \widehat{\mathcal{N}}_y^{\mathscr{L}\text{-ss}}
\end{align*}
with centers $\widehat{\mathscr{Z}}_{i, y}\subset \widehat{\mathscr{S}}_{i, y}$
and one parameter subgroups $\lambda_i \colon \mathbb{C}^{\ast} \to G_y$. 
By Koszul duality, 
see Theorem~\ref{thm:Kduality}, 
there is an equivalence:
\begin{align*}
    \Phi_y \colon D^b(\widehat{\mathfrak{M}}_y) \stackrel{\sim}{\to}
    \mathrm{MF}^{\text{gr}}(\mathcal{X}_y, f_y). 
\end{align*}
Then the subcategory $\mathbb{W}_{w}$ in Theorem~\ref{thm:window:M} is characterized 
as follows: 
an object $\mathcal{E} \in D^b(\mathfrak{M})$ is an object of 
$\mathbb{W}_w$
if and only if, for any closed point $y \in M$, we have 
\begin{align}\label{PhiEy}
    \Phi_{y}(\mathcal{E}|_{\widehat{\mathfrak{M}}_y}) \in \mathrm{MF}^{\text{gr}}(\mathbb{G}_{w'}, f_y), \ 
  w_i'=w_i-\langle \lambda_i, \mathcal{H}^1(\mathbb{T}_{\mathfrak{M}}|_{y})^{\lambda_i>0} \rangle.  
\end{align}
The category $\mathbb{G}_w$ is the window category \eqref{def:window} for the weights $w'_i$ and the 
line bundle $\mathscr{L}$.
The difference of $w_i$ and $w_i'$ is due to the discrepancy of categorical Hall products on 
$\mathfrak{M}_y$ and $\mathcal{X}_y$, see \cite[Proposition 3.1]{P2}. 

\subsubsection{Adjoints of window categories}
Let $\mathcal{N}^{\mathscr{L}\text{-st}} \subset \mathcal{N}^{\mathscr{L}\text{-ss}}$ be the substack of $\mathscr{L}$-stable 
points.
We will use the following lemma in the proof of Theorem \ref{thm:windowdec}: 
\begin{lemma}\label{lemma:adjoint}
Assume that $\mathcal{N}^{\mathscr{L}\text{-st}} =\mathcal{N}^{\mathscr{L}\text{-ss}}$. 
Then the inclusion $\mathbb{W}_w \subset D^b(\mathfrak{M})$ has a right adjoint. 
\end{lemma}
\begin{proof}
    First assume that $\mathfrak{M}$ is a Koszul stack \eqref{koszuldef}, so  $\mathfrak{M}=s^{-1}(0)/G$
    for a section $s$ of a vector bundle $V$ on a smooth affine scheme $Y$
    and a reductive algebraic group $G$. 
    By Koszul duality and replacing $Y$ with $\mathrm{Tot}_Y(V^{\vee})$, we can reduce the problem
    to the existence of a right adjoint of a window category
    $\mathbb{G}_w \subset D^b(\mathcal{Y})$ as in Subsection \ref{subsection:window},
    where $\mathcal{Y}=Y/G$
    and 
    $\mathcal{Y}^{\mathscr{L}\text{-st}}=\mathcal{Y}^{\mathscr{L}\text{-ss}}$
    so that 
    $\mathcal{Y}^{\mathscr{L}\text{-ss}}$ is a 
    projective scheme over $Y\ssslash G$. The scheme $Y$ is smooth and affine, so 
    the good moduli space morphism $\mathcal{Y} \to Y\ssslash G$
    is cohomologically proper. 
    The composition 
    \begin{align}\label{compose:G}
  D^b(\mathcal{Y}^{\mathscr{L}\text{-ss}}) \stackrel{(j^{\ast})^{-1}}{\longrightarrow}
  \mathbb{G}_w \subset D^b(\mathcal{Y})
  \end{align}
  is given by a Fourier-Mukai functor with kernel object
  $\mathcal{P} \in D^b(\mathcal{Y}^{\mathscr{L} \text{-ss}} \times \mathcal{Y})$
  supported on $\mathcal{Y}^{\mathscr{L}\text{-ss}} \times_{Y\ssslash G}\mathcal{Y}$, 
  see~\cite[Lemma~6.7]{T3}. The support of $\mathcal{P}$ is proper over $\mathcal{Y}$
  and cohomologically proper over $\mathcal{Y}^{\mathscr{L}\text{-ss}}$. Thus 
  the Fourier-Mukai functor 
  $D^b(\mathcal{Y}) \to D^b(\mathcal{Y}^{\mathscr{L}\text{-ss}})$
  with kernel $\mathcal{P}^{\vee} \boxtimes \omega_{\mathcal{Y}^{\mathscr{L}\text{-ss}}}[\dim \mathcal{Y}]$
  is well-defined and
  gives a right adjoint of (\ref{compose:G}). 
  
  In general, for each point in $M$, there is an \'{e}tale
  neighborhood $U \to M$
  and Cartesian squares (\ref{dia:fnbd})
 such that $\mathfrak{M}_U$ is of the form 
$s^{-1}(0)/G$ as above. 
Let $\mathbb{W}_{U, w} \subset D^b(\mathfrak{M}_U)$
be the window subcategory in Theorem~\ref{thm:window:M} for $\mathfrak{M}_U$, which 
admits a right adjoint $R_U$ by the above argument. 
Since we have 
\begin{align*}
    \mathbb{W}_{w}=\lim_{U \to M}\mathbb{W}_{U, w}, \ 
    D^b(\mathfrak{M})=\lim_{U\to M}D^b(\mathfrak{M}_U),
\end{align*}
and $R_U$ is compatible with base change, 
a right adjoint of $\mathbb{W}_w \subset D^b(\mathfrak{M})$
is obtained as $\lim_{U\to M}R_U$. 
See Subsection \cite[Subsection 3.1.4]{T} for the definition of the category of smooth morphisms in which we are taking the above limits, 
and the proof of~\cite[Theorem~5.3.13]{T}
for the above descriptions of $\mathbb{W}_w$
and $D^b(\mathfrak{M})$ as limits. 
\end{proof}

\subsection{Intrinsic window subcategory}
Let $\mathfrak{M}$ be a quasi-smooth derived stack 
such that $\mathcal{M}=t_0(\mathfrak{M})$ admits a good moduli 
space $\mathcal{M} \to M$. 
We say that $\mathfrak{M}$ is \textit{symmetric} if for any closed point 
$x \in \mathfrak{M}$, the $G_x:=\mathrm{Aut}(x)$-representation 
\begin{align*}
	\mathcal{H}^0(\mathbb{T}_{\mathfrak{M}}|_{x}) \oplus 
	\mathcal{H}^1(\mathbb{T}_{\mathfrak{M}}|_{x})^{\vee}
	\end{align*}
is a self dual $G_x$-representation. 
In this subsection, we assume that $\mathfrak{M}$ is symmetric. 
Let $\delta \in \mathrm{Pic}(\mathfrak{M})_{\mathbb{R}}$. 
We define a different kind of window categories from the ones in Subsection \ref{windowfirst}, called \textit{intrinsic window subcategories}
$\mathbb{W}_{\delta}^{\rm{int}} \subset D^b(\mathfrak{M})$, 
see~\cite[Definition~5.2.12, 5.3.12]{T}. These categories are the quasi-smooth version of ``magic window categories" from \cite{SVdB, hls}.

First, assume that 
there is a smooth affine scheme $Y$ with an action of 
a reductive algebraic group $G$, 
$V \to Y$ a $G$-equivariant vector bundle, and 
$s$ is a $G$-equivariant section of $V \to Y$ such that 
\begin{align}\label{present:M}
\mathfrak{M}=s^{-1}(0)/G.
\end{align}
Let $\mathcal{Y}=Y/G$, 
$i \colon \mathfrak{M} \hookrightarrow \mathcal{Y}$
the closed immersion, and $\mathscr{V} \to \mathcal{Y}$
the total space of $V/G \to Y/G$. 
In this case, we define 
$\mathbb{W}_{\delta}^{\rm{int}} \subset D^b(\mathfrak{M})$ to be 
consisting of 
$\mathcal{E} \in D^b(\mathfrak{M})$ such that 
for any map $\lambda \colon B\mathbb{C}^{\ast} \to \mathfrak{M}$ we have 
\begin{align*}
	\mathrm{wt}(\lambda^{\ast}i_{\ast}\mathcal{E}) 
	\subset \left[\frac{1}{2}\mathrm{wt}\left(\det \lambda^{\ast}(\mathbb{L}_{\mathscr{V}}|_{\mathcal{Y}})^{\lambda<0}
	\right), 
	\frac{1}{2}\mathrm{wt}\left(\det \lambda^{\ast}(\mathbb{L}_{\mathscr{V}}|_{\mathcal{Y}})^{\lambda>0}\right)
	  \right]
	  +\mathrm{wt}(\lambda^{\ast}\delta). 
	\end{align*}
	Here, by abuse of notation, we have also denoted the composition 
	$B\mathbb{C}^{\ast}\stackrel{\lambda}{\to} \mathfrak{M} \stackrel{i}{\hookrightarrow} \mathcal{Y}$
	by $\lambda$. 
The above subcategory $\mathbb{W}_{\delta}^{\rm{int}}$
is intrinsic to $\mathfrak{M}$, and independent of a 
choice of a presentation $\mathfrak{M}$ as (\ref{present:M})
for $(Y, V, s, G)$,  
see~\cite[Lemma~5.3.14]{T}. 

In general, 
the intrinsic window subcategory is defined as follows:

\begin{defn}(\cite[Definition~5.3.12]{T})\label{def:intwind}
	We define the subcategory 
	\begin{align*}
		\mathbb{W}_{\delta}^{\rm{int}} \subset D^b(\mathfrak{M})
		\end{align*}
	to be consisting of objects $\mathcal{E}$ satisfying the following: 
	for any étale morphism 
$U \to M$ such that $\mathfrak{M}_U$ is of the form $s^{-1}(0)/G$
as above with étale morphism $\iota_U \colon \mathfrak{M}_U \to \mathfrak{M}$, see Subsection~\ref{subsub:qsmooth}, 
we have 
$\iota_U^{\ast}\mathcal{E} \in \mathbb{W}_{\iota_U^{\ast}\delta}^{\rm{int}}
\subset D^b(\mathfrak{M}_U)$. 
\end{defn}

\section{Categorical wall-crossing of DT/PT quivers: review}
In this section, we review 
the categorical wall-crossing formula of DT/PT quivers 
obtained in~\cite{PT3}.

\subsection{DT/PT quivers}\label{sub:DTPTquivers}

A \textit{quiver} $Q=(I, E)$ is a direct graph with a set of vertices $I$ and a set of oriented edges $E$.
A \textit{super-potential} is a linear combination of cycles in $Q$. 

We discuss some examples of quivers which model the DT/PT wall-crossing for any curve class on a Calabi-Yau $3$-fold. These quivers have been studied in \cite{PT3}.
Let $Q=(I, E)$ be the \textit{triple loop quiver}, that is, the quiver with vertex set $I=\{1\}$ and edge set $E=\{x, y, z\}$:
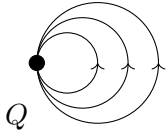
\begin{align*}
Q 
\begin{tikzpicture}
\draw[->] (0, 0) arc (-180:0:0.4) ;
\draw (0.8, 0) arc (0:180:0.4);
\draw[->] (0, 0) arc (-180:0:0.6) ;
\draw (1.2, 0) arc (0:180:0.6);
\draw[->] (0, 0) arc (-180:0:0.8);
\draw (1.6, 0) arc (0:180:0.8);
\draw[fill=black] (0, 0) circle (0.1);
\end{tikzpicture}
\end{align*}

For $a\in \mathbb{N}$, 
let $Q^{af}=(J, E^{af})$ be the quiver with 
vertex set $J=\{0, 1\}$, and edge set $E^{af}$ containing 
three loops $E=\{x, y, z\}$ at $1$, $(a+1)$-edges from $0$ to $1$, 
and $a$-edges from $1$ to $0$. 
A \textit{DT/PT quiver} is a quiver of the form $Q^{af}$ for some $a\in \mathbb{N}$, 
see the following picture for $a=2$: 

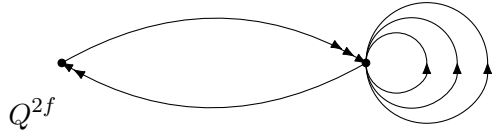
\begin{align*}
	Q^{2f}
	\begin{tikzpicture}			
			\draw[->_>_>, >={Latex[round]}] 	
			(-4, 0) to [bend left=30] (0, 0);
					\draw[->_>, >={Latex[round]}] 	
				(0, 0) to [bend left=30] (-4, 0);
				\draw[->, >={Latex[round]}] (0, 0) arc (-180:0:0.4) ;
		\draw (0.8, 0) arc (0:180:0.4);
		\draw[->, >={Latex[round]}] (0, 0) arc (-180:0:0.6) ;
		\draw (1.2, 0) arc (0:180:0.6);
		\draw[->, >={Latex[round]}] (0, 0) arc (-180:0:0.8);
		\draw (1.6, 0) arc (0:180:0.8);
		\draw[fill=black] (0, 0) circle (0.05);
		\draw[fill=black] (-4, 0) circle (0.05);
	\end{tikzpicture}
	\end{align*}

For $d\in \mathbb{N}$, let $V$ be a $\mathbb{C}$-vector space of dimension $d$. 
We often write $GL(d)$ as $GL(V)$. 
Its Lie algebra is 
denoted by $\mathfrak{gl}(d)=\mathfrak{gl}(V):=\text{End}(V)$. 
When the dimension is clear from the context, we drop $d$ from its notation
and write it as $\mathfrak{g}$. 
Consider the $GL(V)\cong GL(d)$ representations:
\begin{align*}
R(d)&:=\mathfrak{gl}(V)^{\oplus 3},\ 
    R^{af}(1,d):=V^{\oplus \left(a+1\right)}\oplus \left(V^{\vee}\right)^{\oplus a}\oplus \mathfrak{gl}(V)^{\oplus 3}.
\end{align*}
Define the following stacks:
\begin{align*}
\X(d):=R(d)/GL(d),\
\mathcal{X}^{af}(1, d):=R^{af}(1, d)/GL(d).
\end{align*}

\subsection{Quasi-BPS categories for \texorpdfstring{$\mathbb{C}^3$}{C3}}\label{sec:quasiBPS}

\subsubsection{Notations involving weights}\label{subsub:const}

Recall the triple loop quiver $Q$ with moduli stack of representation \[\X(d)=R(d)/GL(d)=\mathfrak{gl}(d)^{\oplus 3}/GL(d)\] of dimension $d$. Let $T(d)$ be the maximal torus of $GL(d)$, let $M(d)$ be its weight space, and let $M(d)_{\mathbb{R}}:=M(d)\otimes_\mathbb{Z}\mathbb{R}$. Denote by $\beta_1,\ldots, \beta_d$ the simple roots of $T(d)$. We choose the dominant chamber of $M(d)$ such that a weight $\chi=\sum_{i=1}^d c_i\beta_i$ is dominant if $c_1\leq\cdots\leq c_d.$
We also define the following elements in $M(d)_{\mathbb{R}}$:
\begin{align*}
\rho:=\frac{1}{2}\sum_{i>j}(\beta_i-\beta_j), \ 
    \tau_d :=\frac{1}{d}\sum_{i=1}^d \beta_i. 
\end{align*}
Note that $\rho$ is half the sum of positive roots of $GL(d)$. 

For later use, we also generalize the above notation and construction. 
Let $d=d^{(1)}+\cdots+d^{(m)}$ and set
$G=\prod_{j=1}^m GL(d^{(j)})$. 
Its maximal torus is $T(d)=\prod_{j=1}^m T(d^{(j)})$, and we
denote the simple roots of $T(d^{(j)})$ by 
$\beta_1^{(j)}, \ldots, \beta_{d^{(j)}}^{(j)}$. 
We write $\chi \in M(d)_{\mathbb{R}}$ as
\begin{align}\label{decom:chij}
\chi=\sum_{j=1}^m \chi^{(j)} \in \bigoplus_{j=1}^m M(d^{(j)})_{\mathbb{R}}.
\end{align}
We say that $\chi$ is $G$-dominant if each $\chi^{(j)}$ is dominant. 
Let $\rho^{(j)}$ be half the sum of positive roots of $GL(d^{(j)})$. 
We set $\rho_G=\sum_{j=1}^m \rho^{(j)}$, which is half the sum of 
positive roots of $G$. 

\subsubsection{The \v{S}penko--Van den Bergh construction}
We explain a construction of (twisted) noncommutative 
resolutions of singularities $\mathbb{M}(d)_w$ of the coarse quotient space $R(d)\ssslash GL(d)=\mathfrak{gl}(d)^{\oplus 3}\ssslash GL(d)$ due to \v{S}penko--Van den Bergh \cite{SVdB}.
Define the polytope $\textbf{W}(d)$ as
\begin{equation}\label{W}
    \textbf{W}(d):=\frac{3}{2}\text{sum}[0, \beta_i-\beta_j]+\mathbb{R}\tau_d\subset M(d)_{\mathbb{R}},
    \end{equation}
where the Minkowski sum is after all $1\leq i, j\leq d$. 
For $w\in\mathbb{Z}$,
consider the hyperplane: 
\begin{equation}\label{W0}
    \textbf{W}(d)_w:=\frac{3}{2}\text{sum}[0, \beta_i-\beta_j]+w\tau_d\subset \textbf{W}(d).
    \end{equation}
For $w \in \mathbb{Z}$, denote by $D^b(\mathcal{X}(d))_w$
the subcategory of $D^b(\mathcal{X}(d))$
consisting of objects of
weight $w$ with respect to the diagonal 
cocharacter $1_d$ of $T(d)$.  
There is a direct sum decomposition 
\begin{align*}
    D^b(\mathcal{X}(d))=\bigoplus_{w\in \mathbb{Z}}
    D^b(\mathcal{X}(d))_w. 
\end{align*}
We define the dg-subcategories 
\begin{align*}
    \mathbb{M}(d) \subset D^b(\X(d)), \ 
    (\mbox{resp. }
    \mathbb{M}(d)_w \subset D^b(\X(d))_w)
\end{align*}
to be generated 
by the vector bundles $\OO_{\X(d)}\otimes \Gamma_{GL(d)}(\chi)$, where $\chi$ is a dominant weight of $T(d)$ such that
\begin{equation}\label{M}
    \chi+\rho\in \textbf{W}(d), \ 
    (\mbox{resp. } \chi+\rho \in \textbf{W}(d)_w). 
    \end{equation}
    Note that $\mathbb{M}(d)$
    decomposes into the direct sum of $\mathbb{M}(d)_w$
    for $w \in \mathbb{Z}$. 
Alternatively, the category $\mathbb{M}(d)_w$ 
is generated by the vector bundles $\OO_{\X(d)} \otimes \Gamma$
for $\Gamma$ a $GL(d)$-representation such that 
the $T(d)$-weights of $\Gamma$ are contained in the set
$\nabla_w$ defined by:
\begin{align}\label{def:nabla}
 \nabla_{w} &:=
    \left\{\chi \in M_{\mathbb{R}} \,\Big| -\frac{1}{2}
    \eta_{\lambda} \leq \langle \lambda, \chi \rangle 
    \leq \frac{1}{2}\eta_{\lambda} \mbox{ for all } \lambda \colon \mathbb{C}^* \to T(d) \right\} +w\tau_d, 
    \end{align}
   where $\eta_{\lambda}:=\langle \lambda, \mathbb{L}_{\mathcal{X}(d)}^{\lambda>0}\rangle=2\langle \lambda, \mathfrak{gl}(d)^{\lambda>0}\rangle$,
    see~\cite[Lemma~2.9]{hls}.

\subsubsection{Quasi-BPS categories via matrix factorizations}
    Consider the regular function
    \begin{align}\label{TrWd}\mathrm{Tr}\,W_d\colon \X(d)\to \mathbb{C},\, (X, Y, Z)\mapsto \mathrm{Tr}\left(X[Y, Z]\right)
    \end{align}
    induced by the potential $W=x[y, z]=xyz-xzy$ of the triple loop quiver $Q$.
Let
\[\mathbb{S}(d):=\text{MF}(\mathbb{M}(d), \Tr W_d)
\subset \mathrm{MF}(\mathcal{X}(d), \Tr W_d)
\]  be the subcategory of matrix factorizations
$\left(\alpha \colon F\rightleftarrows G\colon \beta\right)$ with $F$ and $G$ in $\mathbb{M}(d)$.
It decomposes into the direct sum of 
$\mathbb{S}(d)_w$ for $w \in \mathbb{Z}$, 
where $\mathbb{S}(d)_w$
is defined similarly to $\mathbb{S}(d)$
using $\mathbb{M}(d)_w$. 
We also consider analogous categories of graded and/or equivariant matrix factorizations
\[\mathbb{S}^{\bullet}_{\ast}(d):=
\text{MF}^{\bullet}_{\ast}(\mathbb{M}(d), \Tr W_d)
\subset \text{MF}^{\bullet}_{\ast}(\X(d), \Tr W_d)
\] 
for $\ast \in \{\emptyset, T\}$, 
$\bullet \in \{\emptyset, \text{gr}\}$, 
where $T=(\mathbb{C}^{\ast})^2$ acts on $\mathbb{C}^3$
by $(t_1, t_2)(x, y, z)=(t_1 x, t_2 y, t_1^{-1}t_2^{-1}z)$
and the grading is given by the weight two $\mathbb{C}^{\ast}$-action on $Z$. 
    
\subsubsection{Quasi-BPS categories via the Koszul equivalence}\label{subsub:qBPSK}    
    \label{subsec:qBPS}
Let $V$ be a $d$-dimensional complex vector space 
and denote by $\mathfrak{g}=\Hom(V, V)$ the Lie algebra of $GL(V)$. 
We set 
\begin{align*}
	\mathcal{Y}(d):=\mathfrak{g}^{\oplus 2}/GL(V),
	\end{align*}
where $GL(V)$ acts on $\mathfrak{g}$ by conjugation. 
The stack $\mathcal{Y}(d)$ is the moduli stack of 
representations of dimension $d$ of the quiver with one vertex and two loops. 
Let $\mu$ be the morphism 
\begin{align}\label{mor:s}
\mu \colon \mathcal{Y}(d) \to \mathfrak{g}, \ (X, Y) \mapsto [X, Y].
\end{align}
The morphism $\mu$ induces a morphism of vector bundles $\partial: \mathfrak{g}^{\vee}\otimes\mathcal{O}_{\mathfrak{g}^{\oplus 2}}\to \mathcal{O}_{\mathfrak{g}^{\oplus 2}}$.
Let $\mu^{-1}(0)$ be the derived scheme with the dg-ring of regular functions
\begin{align}\label{diff:ds2}
\mathcal{O}_{\mu^{-1}(0)}:=\mathcal{O}_{\mathfrak{g}^{\oplus 2}}\left[\mathfrak{g}^{\vee}\otimes\mathcal{O}_{\mathfrak{g}^{\oplus 2}}[1]; d_{\mu}\right],
\end{align}
where the differential $d_\mu$ is induced by the morphism $\partial$. Consider the (derived) stack 
\begin{equation}\label{def:cd}
i \colon \mathscr{C}(d):=\mu^{-1}(0)/GL(V) \hookrightarrow \mathcal{Y}(d).
\end{equation}
For $v\in\mathbb{Z}$,
define the full dg-subcategory 
$\widetilde{\mathbb{T}}(d)_v\subset D^b(\mathcal{Y}(d))$
generated by the vector bundles $\mathcal{O}_{\mathcal{Y}(d)}\otimes \Gamma_{GL(d)}(\chi)$
for a dominant weight $\chi$ satisfying
$\chi+\rho \in \textbf{W}(d)_{v}$. 
Define the full dg-subcategory 
\begin{align}\label{def:N}
	\mathbb{T}(d)_v \subset D^b(\mathscr{C}(d))
	\end{align}
with objects $\mathcal{E}$ such that 
 $i_{\ast}\mathcal{E}$ is in $\widetilde{\mathbb{T}}(d)_v$.
 
    Consider the grading induced by the action of $\mathbb{C}^*$ on $\X(d)$ scaling the linear map corresponding to $Z$ with weight $2$.
By Koszul duality in Theorem~\ref{thm:Kduality}, there is an equivalence
\begin{align}\label{equiv:Phi2}
	\Phi \colon D^b(\mathscr{C}(d)) \stackrel{\sim}{\to}
	\mathrm{MF}^{\mathrm{gr}}(\mathcal{X}(d), \Tr W_d). 
	\end{align}
	The above equivalence restricts to the equivalence, 
	see~\cite[Equation (4.22)]{PT0}:
	\begin{align}\label{equiv:ST}
	    \Phi\colon \mathbb{T}(d)_v\stackrel{\sim}{\to} \mathbb{S}^{\mathrm{gr}}(d)_v. 
	\end{align}

    \subsection{The local categorical DT/PT correspondence}
Let $Q^{af, N}$ be the quiver obtained by adding $N$-loops at the vertex $0$
of the DT/PT quiver $Q^{af}$. 
Then the moduli stack of representations of $Q^{af, N}$ 
of dimension $(1, d)$ is 
\begin{align*}
    \mathcal{X}^{af, N}(1, d)=\mathbb{C}^N \times \mathcal{X}^{af}(1, d)
    =(\mathbb{C}^N \times R^{af}(1, d))/GL(d) 
\end{align*}
where $GL(d)$ acts on $\mathbb{C}^N$ trivially. 
Let $\widetilde{W}$ be a super-potential on $Q^{af, N}$ satisfying \[\widetilde{W}|_{Q}=X[Y, Z],\] where 
$Q \subset Q^{af, N}$ is the full subquiver consisting 
of the vertex $\{1\}$, i.e. it is the triple loop quiver. 
Then we have the function 
\begin{align*}
    \Tr \widetilde{W}_d \colon \mathcal{X}^{af, N}(1, d) \to \mathbb{C}. 
\end{align*}

Let $\lambda \colon \mathbb{C}^{\ast} \to T(d)$ be the cocharacter given by 
\begin{align*}
    \lambda(t)=(\overbrace{t^k, \ldots, t^k}^{d_1}, \ldots, \overbrace{t, \ldots, t}^{d_k}, \overbrace{1, \ldots, 1}^{d'}). 
\end{align*}
We have the diagram of attracting and fixed stacks 
\begin{align}\label{dia:XfN}
    \mathcal{X}^{af, N}(1, d)^{\lambda} \stackrel{q_{\lambda}}{\leftarrow} 
    \mathcal{X}^{af, N}(1, d)^{\lambda \geq 0} \stackrel{p_{\lambda}}{\to}
    \mathcal{X}^{af, N}(1, d)
\end{align}
where the fixed stack is 
\begin{align*}
\mathcal{X}^{af, N}(1, d)^{\lambda}=\prod_{i=1}^k \mathcal{X}(d_i) \times 
\mathcal{X}^{af, N}(1, d'). 
\end{align*}
The pull-back/push-forward with respect to the diagram (\ref{dia:XfN})
gives the categorical Hall product for DT/PT quivers with super-potentials, see~\cite[Section~3]{P0}: 
\begin{align}\label{cat:Hall:M}
    \ast=q_{\lambda \ast}p_{\lambda}^{\ast} \colon 
    \boxtimes_{i=1}^k \mathrm{MF}(\mathcal{X}(d_i), \Tr W_{d_i}) \boxtimes &
    \,\mathrm{MF}(\mathcal{X}^{af, N}(1, d'), \Tr \widetilde{W}_{d'}) \\
    &\notag \to 
     \mathrm{MF}(\mathcal{X}^{af, N}(1, d), \Tr \widetilde{W}_d). 
\end{align}

Let $\delta=\mu \chi_0=\mu\sum_{i=1}^d \beta_i$ for $\mu \in \mathbb{R}$. We  define the 
subcategories 
\begin{align*}
    &\mathbb{E}^a(1, d; \delta) \subset \mathrm{MF}(\mathcal{X}^{af, N}(1, d), \Tr \widetilde{W}_d), \\  
    &\mathbb{F}^a(1, d; \delta) \subset \mathrm{MF}(\mathcal{X}^{af, N}(1, d), \Tr \widetilde{W}_d)
    \end{align*}
    of matrix factorizations whose factors are generated by  
    $\Gamma_{GL(d)}(\chi) \otimes \mathcal{O}_{\mathcal{X}^{af, N}(1, d)}$
    for $\chi$ a dominant weight of $T(d)$ satisfying 
    \begin{align*}
        &\chi+\rho+\delta \in \frac{3}{2}\mathrm{sum}[0, \beta_i-\beta_j]+
        \frac{a}{2}\mathrm{sum}[-\beta_k, \beta_k]+\mathrm{sum}[-\beta_k, 0], \\
        &\chi+\rho+\delta \in \frac{3}{2}\mathrm{sum}[0, \beta_i-\beta_j]+
        \frac{a}{2}\mathrm{sum}(-\beta_k, \beta_k]
    \end{align*}
    respectively, where $1\leq i,j,k\leq d$. 
    
   	Let $R$ be the following set \[R=\{(d_i, v_i)_{i=1}^k\mid d_i\in \mathbb{N}, v_i \in \mathbb{Z}\}.\] 
	We define a subset $O \subset R \times R$ which will be used to compare summands in the semiorthogonal decomposition of Theorem~\ref{thm:locDTPT} and 
	Theorem \ref{thm:main}, see Subsection \ref{notation}. 
	This order is a particular example of an order from \cite[Subsection 3.2.3]{PT3}, which depends on a choice of $\mu\in\mathbb{R}$. 
	
	\begin{defn}\label{def:order}
	A pair of elements $S=(d_i, v_i)_{i=1}^k$
	and $S'=(d_i', v_i')_{i=1}^{k'}$
	in $R$
	is an element of $O \subset R \times R$ if either 
	\begin{itemize}
	    \item $\sum_{i=1}^k v_i >\sum_{i=1}^{k'} v_i'$, or 
	    \item $\sum_{i=1}^k v_i =\sum_{i=1}^{k'} v_i'$
	    and $\sum_{i=1}^k d_i<\sum_{i=1}^{k'} d_i'$, or 
	    \item $\sum_{i=1}^k v_i =\sum_{i=1}^{k'} v_i'$,
	    and $\sum_{i=1}^k d_i=\sum_{i=1}^{k'} d_i'$, and $(S, S')$
	    is in the set $O$ from~\cite[Subsection~3.4]{PT0}. 
	\end{itemize} 
	\end{defn}
        The following is the local categorical DT/PT correspondence in terms of the 
    above subcategories, see~\cite[Proposition~3.12, Corollary~3.14]{PT3}: 
    \begin{thm}\label{thm:locDTPT}
        Suppose that $2 \mu l \notin \mathbb{Z}$ for $1\leq l \leq d$. 
        There is a semiorthogonal decomposition 
        \begin{align*}
            \mathbb{E}^a(1, d; \delta)=\left\langle \boxtimes_{i=1}^k 
            \mathbb{S}(d_i)_{w_i} \boxtimes \mathbb{F}(1, d'; \delta') \right\rangle. 
        \end{align*}
              The right hand side is after all $d'\leq d$, partitions $(d_i)_{i=1}^k$ of $d-d'$, and integers $(w_i)_{i=1}^k$ such that for \begin{equation}\label{vwtransform}
                  v_i:=w_i+d_i\left(d'+\sum_{j>i} d_j-\sum_{j<i}d_j\right), \end{equation}
              we have
    \begin{equation}\label{equationboundmu}
    -1-\mu-\frac{a}{2}< \frac{v_1}{d_1}<\cdots<\frac{v_k}{d_k}< -\mu-\frac{a}{2}.
\end{equation}
In the above, we let $\delta':=\left(\mu-d+d'\right)\sum_{i>d-d'}^d\beta_i$.
Moreover, each fully-faithful functor 
\[\boxtimes_{i=1}^k 
            \mathbb{S}(d_i)_{w_i} \boxtimes \mathbb{F}(1, d'; \delta')
            \to \mathbb{E}^a(1, d; \delta)\]
            is given by the restriction of the categorical Hall product~\eqref{cat:Hall:M}. 
            
    The order of the semiorthogonal summands is 
       discussed in \cite[Subsection 3.2.3]{PT3} and depends on $\mu\in\mathbb{R}$. 
    For $0<\varepsilon\ll 1$, let $\mu=-a/2-\varepsilon$. Then the order is
     that of $(d_i, v_i)_{i=1}^k \in R$
    in Definition~\ref{def:order}, see also Subsection \ref{notation}. 
    \end{thm}

The analogous conclusion holds if we replace $\mathbb{C}^N$ by an open subset in all the constructions and statements above, and also for graded categories of matrix factorizations. 

\begin{remark}
One also obtains similar semiorthogonal decompositions for different super-potentials $\widetilde{W}$. Assume $N=0$ and $\widetilde{W}=0$.
Let $\chi_0 \colon GL(V) \to \mathbb{C}^{\ast}$ be the determinant character
$g \mapsto \det g$, 
and define the following DT and PT spaces for 
the quivers $Q^{af}$
defined by the GIT quotient stacks (which are smooth quasi-projective varieties):
\begin{align*}
    I^a(d):=\mathcal{X}^{af}(1, d)^{\chi_0\text{-ss}},\
    P^a(d):=\mathcal{X}^{af}(1, d)^{\chi^{-1}_0\text{-ss}}.
\end{align*}
Using the window theorem, we obtain the local categorical DT/PT correspondence \cite[Theorem 1.1]{PT3}, which says that there is a semiorthogonal decomposition
\[D^b(I^a(d))=\Big\langle \boxtimes_{i=1}^k \mathbb{M}(d_i)_{w_i}\boxtimes D^b(P^a(d'))\Big\rangle,\]
where the right hand side is indexed as in Theorem \ref{thm:locDTPT}. 
\end{remark}

\section{Quasi-BPS categories for points on a surface}
In this section, we introduce quasi-BPS categories for points on surfaces
and study their properties. 

\subsection{The definition of quasi-BPS category}\label{subquasi-BPS}
Let $S$ be a smooth surface. We denote by 
 $\mathfrak{M}_S(d)$ the derived moduli stack of 
 zero-dimensional sheaves on $S$ of length $d$. 
 The derived stack $\mathfrak{M}_S(d)$ is quasi-smooth, 
 and its classical truncation $\mathcal{M}_S(d)=t_0(\mathfrak{M}_S(d))$
 admits a good moduli space 
 \begin{align*}
 	\mathcal{M}_S(d) \to \mathrm{Sym}^d(S)
 	\end{align*}
 by sending a zero-dimensional sheaf to its support. 
 Moreover, the stack $\mathfrak{M}_S(d)$ is
 symmetric, see~\cite[Lemma~5.4.1]{T}. 
 
 Let $\mathcal{Q}\in D^b(S \times \mathfrak{M}_S(d))$ be the universal zero-dimensional sheaf. 
 We define the following line bundle on $\mathfrak{M}_S(d)$:
 \begin{align*}
 	\mathscr{L}:= \mathrm{det}(p_{\mathfrak{M}\ast}\mathcal{Q}) 
 	\in \mathrm{Pic}(\mathfrak{M}_S(d)),
 	\end{align*}
 	where $p_{\mathfrak{M}} \colon S \times \mathfrak{M}_S(d) \to \mathfrak{M}_S(d)$ is
  the natural projection. 
 	Recall the definition of an intrinsic window subcategory for 
 	a quasi-smooth and symmetric derived stack, see Definition~\ref{def:intwind}. 
 We define the quasi-BPS category as follows: 
 \begin{defn}
 	For $(d, w) \in \mathbb{N} \times \mathbb{Z}$, 
 	we define the quasi-BPS category 
 	$\mathbb{T}_X(d)_w \subset D^b(\mathfrak{M}_S(d))$ by 
 	\begin{align*}
 		\mathbb{T}_X(d)_w:=\mathbb{W}_{\delta=w \mathscr{L}/d}^{\rm{int}}
 		\subset D^b(\mathfrak{M}_S(d)). 
 		\end{align*}
 	\end{defn}
 We see that the above definition coincides with the 
 one introduced in~\cite{P2}. 
 Let us take a point $p=\sum_{j=1}^m d^{(j)}x^{(j)} \in \mathrm{Sym}^d(S)$ for distinct 
 points $x^{(1)}, \ldots, x^{(m)} \in S$, and denote by 
 $\widehat{\mathfrak{M}}_S(d)_p$ the formal fiber of $\mathfrak{M}_S(d)$ at $p$.
 Let $y^{(1)}, \ldots, y^{(m)} \in \mathbb{C}^2$ be distinct points 
 and set $q=\sum_{j=1}^m d^{(j)}y^{(j)} \in \mathrm{Sym}^d(\mathbb{C}^2)$. 
 Then there is an equivalence of derived stacks 
 \begin{align*}
 	\iota_p \colon \widehat{\mathscr{C}}(d)_q \stackrel{\sim}{\to}
 	\widehat{\mathfrak{M}}_S(d)_p.
 	\end{align*}
 Let $\mathbb{T}(d)_w \subset D^b(\mathscr{C}(d))$ be the quasi-BPS category 
 from Subsection~\ref{subsub:qBPSK} and define 
 $\widehat{\mathbb{T}}(d)_{w, q} \subset D^b(\widehat{\mathscr{C}}(d)_q)$
 to be the subcategory split generated by 
 $\mathbb{T}(d)_{w}|_{\widehat{\mathscr{C}}(d)_q}$. 
 The following lemma
 shows that 
 $\mathbb{T}_X(d)_w$ coincides with the 
 one defined in~\cite[Section~4.1.6]{P2}:
 	\begin{lemma}\label{lem:compare}
 		An object $\mathcal{E} \in D^b(\mathfrak{M}_S(d))$ 
 		is an object in $\mathbb{T}_X(d)_w$ if and only if 
 		for any $p \in \mathrm{Sym}^d(S)$, we have 
 		$\iota_p^{\ast}\mathcal{E}|_{\widehat{\mathfrak{M}}_S(d)_p} \in \widehat{\mathbb{T}}(d)_{w, q}$. 
 		 		\end{lemma}
 	 		\begin{proof}
 	 			The defining property of the intrinsic window 
 	 			subcategory is local on the good moduli space, so 
 	 			$\mathcal{E} \in D^b(\mathfrak{M}_S(d))$ is an object 
 	 			in $\mathbb{W}_{\delta}^{\rm{int}}$ if and only if, 
 	 			for any $p\in \mathrm{Sym}^d(S)$, we have 
 	 			$\mathcal{E}|_{\widehat{\mathfrak{M}}_S(d)_p} \in 
 	 			\mathbb{W}_{\delta_p}^{\rm{int}}$
 	 			in $D^b(\widehat{\mathfrak{M}}_S(d)_p)$, 
 	 			where $\delta_p=\delta|_{\widehat{\mathfrak{M}}_S(d)_p}$. 
 	 			 	 				By the presentation invariance of the intrinsic window subcategory, 
 	 			see~\cite[Lemma~5.3.14]{T}, 
 	 			this is also equivalent to
 	 		$\mathcal{F}_q:=\iota_p^{\ast}
 	 		\left(\mathcal{E}\big|_{\widehat{\mathfrak{M}}_S(d)_p}\right)$
 	 		being an object of $\mathbb{W}_{w \chi_0/d}^{\rm{int}}$ in $D^b(\widehat{\mathscr{C}}(d)_q)$. 
 	 		Here, note that the line bundle $\mathscr{L}$ restricted to 
 	 			$\widehat{\mathfrak{M}}_S(d)_p$ corresponds to the 
 	 			determinant character 
 	 			$\chi_0=\det \colon GL(d) \to \mathbb{C}^{\ast}$
 	 			on $\widehat{\mathscr{C}}(d)_q$. 	
 	 			
 	 			 We have the closed immersion
 	 			\begin{align*}
 	 				j_q \colon 
 	 				\widehat{\mathscr{C}}(d)_q=
 	 				\mu_q^{-1}(0)/GL(d) \hookrightarrow \widehat{\mathcal{Y}}(d)_q,
 	 			\end{align*}
 	 			where $\mu_q$ is the commutator map $\widehat{\mathcal{Y}}(d)_q \to 
 	 			\mathfrak{gl}(d)/GL(d)$. 
 	 			Let $\mathscr{V} \to \widehat{\mathcal{Y}}(d)_q$
 	 			be the total space of the vector bundle induced by the $GL(d)$-representation 
 	 			$\mathfrak{gl}(d)$. 
 	 			Then we have 
 	 			\begin{align*}
 	 			    \mathbb{L}_{\mathscr{V}}|_{\widehat{\mathcal{Y}}(d)_q}
 	 			    =\left(\mathfrak{gl}(d) \otimes \mathcal{O}_{\widehat{\mathcal{Y}}(d)_q} 
 	 			    \to \mathfrak{gl}(d)^{\oplus 3} \otimes \mathcal{O}_{\widehat{\mathcal{Y}}(d)_q} \right). 
 	 			\end{align*}
 	 			Therefore, the condition that $\mathcal{F}_q \in \mathbb{W}_{w\chi_0/d}^{\rm{int}}$
 	 			means that, for any 
 	 		$\lambda \colon B\mathbb{C}^{\ast} \to \widehat{\mathscr{C}}(d)_q$, 
 	 		we have 
 	 		\begin{align}\label{cond:lambda}
 	 			\mathrm{wt}(\lambda^{\ast}j_{q\ast}\mathcal{F}_q)
 	 			\subset \left[-\langle \lambda, \mathfrak{gl}(d)^{\lambda>0} \rangle, \langle \lambda, \mathfrak{gl}(d)^{\lambda>0}\rangle  \right]+w\tau_d,
 	 			\end{align} where $\mathrm{wt}(\lambda^{\ast}j_{q\ast}\mathcal{F}_q)$ denotes the set of $\mathbb{C}^*$-weights of $\lambda^{\ast}j_{q\ast}\mathcal{F}_q$.
 	 			Here, in the right hand side, we have denoted the cocharacter 
 	 			$\mathbb{C}^{\ast}\to GL(d)$ associated with $\lambda \colon B\mathbb{C}^{\ast} \to \widehat{\mathscr{C}}(d)_q$ by the same symbol $\lambda$. 
  			The condition (\ref{cond:lambda}) for all $\lambda$ is equivalent to 
  			$j_{q\ast}\mathcal{F}_q$ being generated by the vector bundles
  			$\mathcal{O}_{\widehat{\mathcal{Y}}(d)_q}\otimes\Gamma$
  			for $GL(d)$-representations $\Gamma$ whose $T(d)$-weights are 
  			contained in $\nabla_w$, see (\ref{def:nabla}). 
  			The last condition is also equivalent to 
  			$\mathcal{F}_q \in \widehat{\mathbb{T}}(d)_{w, q}$, 
  			see Subsection~\ref{subsub:const}. 
 	 			\end{proof}

We also give another characterization of quasi-BPS categories 
in terms of Ext-quivers. 
Let $p \in \mathrm{Sym}^d(S)$ be as above. 
The unique closed point in the fiber 
of $\mathcal{M}_S(d) \to \mathrm{Sym}^d(S)$
at $p$ corresponds to the 
semisimple sheaf 
\begin{align*}
	F=\bigoplus_{j=1}^m V^{(j)} \otimes \mathcal{O}_{x^{(j)}}. 
\end{align*}
We set 
\begin{align*}
	G_p :=\mathrm{Aut}(F)=\prod_{j=1}^m \GL(V^{(j)}), \ 
	\mathcal{Y}'(d)_p &:= \left[\widehat{\Ext}_S^1(F, F)/G_p\right]. 
\end{align*}
Let $\kappa_p$ be the following $G_p$-equivariant morphism 
\begin{align*}
	\kappa_p \colon \widehat{\Ext}_S^1(F, F)
	 \to \Ext_S^2(F, F), \ u \mapsto [u, u]. 
\end{align*}
Then the formal fiber 
$\widehat{\mathfrak{M}}_{S}(d)_p$
of $\mathfrak{M}_S(d)$ at $p$ 
is also equivalent to 
the derived zero locus of $\kappa_p$:
\begin{align}\label{equiv:kappa}
\kappa_p^{-1}(0)/G_p \stackrel{\sim}{\to} \widehat{\mathfrak{M}}_S(d)_p. 
    \end{align}

We define 
$\widetilde{\mathbb{T}}_{X}(d)_{w, p} \subset D^b(\mathcal{Y}'(d)_p)$
to be the subcategory generated by 
$\Gamma_{G_p}(\chi) \otimes \mathcal{O}_{\mathcal{Y}'(d)_p}$, where $\chi$ is a 
$G_p$-dominant weight of $T(d)$  
satisfying 
\begin{align}\label{chirho}
	\chi+\rho_{G_p} \in
	\frac{3}{2}\mathrm{sum}[0, \beta_{i}^{(j)}-\beta_{i'}^{(j)}]
	+w\tau_d,
\end{align} 
where the sum is after all $1\leq j\leq m$ and $1\leq i, i'\leq d^{(j)}$, see also Subsection~\ref{subsub:const} for the definition of a $G_p$-dominant weight.  
Define the subcategory
\begin{align*}
	\mathbb{T}_{X}(d)_{w, p} \subset D^b(\widehat{\mathfrak{M}}_S(d)_p)
\end{align*}
consisting of objects $\mathcal{E}$ such that $i_{p\ast}\mathcal{E} \in \widetilde{\mathbb{T}}_{X}(d)_{w, p}$, 
where $i_p \colon \widehat{\mathfrak{M}}_S(d)_p \hookrightarrow \mathcal{Y}'(d)_p$
is the closed immersion through the equivalence (\ref{equiv:kappa}). 

\begin{lemma}\label{lem:Mdw}
	An object $\mathcal{E} \in D^b(\mathfrak{M}_{S}(d))_w$ lies in 
	$\mathbb{T}_X(d)_w$ if and only if, for any closed point $p\in \mathrm{Sym}^d(S)$, 
	we have
	$\mathcal{E}|_{\widehat{\mathfrak{M}}_{S}(d)_p}
		\in \mathbb{T}_X(d)_{w, p}$. 
\end{lemma}
\begin{proof}
	The proof is the same as the one of Lemma~\ref{lem:compare}, 
	using the presentation independence of the intrinsic 
	window subcategory and the presentation 
	$\widehat{\mathfrak{M}}_S(d)_p=\kappa_p^{-1}(0)/G_p$
	instead of $\mu_q^{-1}(0)/GL(d)$. 
	\end{proof}

As before, we can also characterize $\mathbb{T}_X(d)_{w, p}$ via Koszul duality. 
Let $(\mathcal{X}(d)_p, f_p)$ be defined by 
\begin{align*}
	\mathcal{X}(d)_p=\left(\widehat{\Ext}_S^1(F, F) \oplus \Ext_S^2(F, F)^{\vee} \right)/G_p
	\stackrel{f_p}{\to} \mathbb{C}
\end{align*}
where $f_p(u, v)=\langle \kappa_p(u), v \rangle$. 
Note that we have 
\begin{align}\label{prod:V}
	(\Ext_S^1(F, F) \oplus \Ext_S^2(F, F)^{\vee})/G_p
	=\prod_{j=1}^m \mathrm{End}(V^{(j)})^{\oplus 3}/\GL(V^{(j)}), 
\end{align}
such that $f_p=\sum_{j=1}^m \Tr W_{d^{(j)}}$,
where $\Tr W_d$ is given in (\ref{TrWd}). 
We denote by 
\begin{align*}
	\mathbb{S}^{\rm{gr}}(d)_{w, p} \subset \text{MF}^{\text{gr}}(\mathcal{X}(d)_p, f_p)
\end{align*}
the dg-subcategory generated by 
matrix factorizations whose factors are generated by  
$\Gamma_{G_p}(\chi) \otimes \mathcal{O}_{\mathcal{X}(d)_p}$, where $\chi$ is a
$G_p$-dominant weight
satisfying (\ref{chirho}). 
The Koszul duality gives an equivalence (see Theorem~\ref{thm:Kduality}): 
\begin{align}\label{Kos:p}
	\Phi_p \colon D^b(\widehat{\mathfrak{M}}_{S}(d)_p)
	\stackrel{\sim}{\to} \text{MF}^{\text{gr}}(\mathcal{X}(d)_p, f_p). 
\end{align}
Then similarly to (\ref{equiv:ST}), it restricts to the equivalence 
\begin{align*}
	\Phi_p \colon \mathbb{T}_X(d)_{w, p} \stackrel{\sim}{\to}\mathbb{S}^{\rm{gr}}(d)_{w, p}.
\end{align*}

\begin{remark}\label{rmk:product}
We set $p^{(j)}=d^{(j)}x^{(j)} \in \mathrm{Sym}^{d^{(j)}}(S)$ for $1\leq j\leq m$. We then have 
\begin{align}\label{prod:Mp}
	\widehat{\mathfrak{M}}_{S}(d)_p=\prod_{j=1}^m   \widehat{\mathfrak{M}}_{S}(d^{(j)})_{p^{(j)}}. 
\end{align}
Note that under the product (\ref{prod:Mp}),
we have 
\begin{align*}
	\mathbb{T}_{X}(d)_{w, p}=\boxtimes_{j=1}^m
	\mathbb{T}_{X}(d^{(j)})_{w^{(j)}, p^{(j)}}
\end{align*}
where $w^{(j)}/d^{(j)}=w/d$. 
Similarly under the product (\ref{prod:V}),
we have 
\begin{align}\label{Mtimes}
	\mathbb{S}^{\rm{gr}}(d)_{w, p}=
	\boxtimes_{j=1}^m \mathbb{S}^{\rm{gr}}(d^{(j)})_{w^{(j)}, p^{(j)}}. 
\end{align}
\end{remark}

\subsection{The semiorthogonal decomposition of \texorpdfstring{$D^b(\mathfrak{M}_S(d))$}{MS} into
quasi-BPS categories}

The main result of \cite{P2} is that quasi-BPS categories provide a semiorthogonal decomposition of $D^b(\mathfrak{M}_{S}(d))$, alternatively of the Hall algebra of points on a surface. 
Let 
\begin{align}\label{dfilt}
\mathfrak{M}_S(d_1, \ldots, d_k)
\end{align}
be the derived moduli stack 
of filtrations of coherent sheaves on $S$:
\begin{align}\label{filt:Q}
0=Q_0 \subset	Q_1 \subset Q_2 \subset \cdots \subset Q_k
	\end{align}
such that each subquotient $Q_i/Q_{i-1}$ is a zero-dimensional 
sheaf on $S$ with length $d_i$. 
There exist evaluation morphisms 
\begin{align*}
	\mathfrak{M}_S(d_1) \times \cdots \times \mathfrak{M}_S(d_k) \stackrel{q}{\leftarrow} 
	\mathfrak{M}_S(d_1, \ldots, d_k) \stackrel{p}{\to} \mathfrak{M}_S(d),
	\end{align*}
where $p$ is proper, $q$ is quasi-smooth,
and $d=d_1+\cdots+d_k$. 
The above diagram for $k=2$
defines the categorical Hall product
due to Porta--Sala \cite{PoSa}:
\begin{align}\label{PS}
\ast=p_{\ast}q^{\ast} \colon 
D^b(\mathfrak{M}_S(d_1)) \boxtimes D^b(\mathfrak{M}_S(d_2)) \to 
D^b(\mathfrak{M}_S(d)).
\end{align}

\begin{thm}\emph{(\cite[Theorem~1.1]{P2})}\label{SODsurface}
	There is a semiorthogonal decomposition 
	\begin{align}\label{eqSODsurface}
		D^b(\mathfrak{M}_{S}(d))=\left\langle \bigotimes_{i=1}^k 
		\mathbb{T}_X(d_i)_{v_i} \,\Big|\, 
		\frac{v_1}{d_1}<\cdots<\frac{v_k}{d_k}, d_1+\cdots+d_k=d\right\rangle. 
		\end{align}
	Each fully-faithful functor 
	$\bigotimes_{i=1}^k \mathbb{T}_X(d_i)_{v_i} \hookrightarrow D^b(\mathfrak{M}_{S}(d))$
	is given by the categorical Hall product. 
	The order is given by the subset $O \subset R \times R$
	from~\cite[Subsection~3.4]{PT0}. 
	\end{thm}

\subsection{Objects in quasi-BPS categories for \texorpdfstring{$\mathbb{C}^3$}{C3}} 
If $S=\mathbb{C}^2$, the stack $\mathfrak{M}_S(d)$ is identified with 
the derived commuting stack
defined in Subsection~\ref{subsec:qBPS},
i.e. 
\begin{align*}
\mathfrak{M}_{\mathbb{C}^2}(d)=\mathscr{C}(d).
\end{align*}
Let $(d, v)\in \mathbb{N}\times\mathbb{Z}$. 
We revisit the construction of the objects $\mathcal{E}_{d,v}\in D^b(\mathscr{C}(d))$
from \cite[Definition 4.2]{PT0}, which are objects of quasi-BPS categories. 
Let $\mathfrak{b}$ be the Lie algebra of the Borel subgroup 
$B(d) \subset GL(d)$ consisting of upper triangular matrices, let
$\mathfrak{n}$ be the nilpotent subalgebra of $\mathfrak{b}$, and let $\overline{\mathfrak{n}}:=[\mathfrak{n}, \mathfrak{n}]$. Consider the commutator map \[\nu\colon \mathfrak{n}^{\oplus 2}\to \overline{\mathfrak{n}}.\]  Let $B(d)$ act  naturally on the derived zero locus 
$\nu^{-1}(0)$ and let $T(d)\subset B(d)$ act trivially on $\mathbb{C}^2$. 
Consider the stack $\mathcal{Z}$ and the maps
\[\mathbb{C}^2/T(d)\xleftarrow{q} \mathcal{Z}:=\left(\nu^{-1}(0)\times \mathbb{C}^2\right)/B(d)\xrightarrow{p}\mathscr{C}(d),\]
where $p$ 
is a proper map induced by the inclusion 
$\mathfrak{n}^{\oplus 2} \subset \mathfrak{g}^{\oplus 2}$
and the diagonal matrices $\mathbb{C} \subset \mathfrak{g}$, 
and $q$ is induced by the natural projections $\nu^{-1}(0) \times \mathbb{C}^2 \to \mathbb{C}^2$, 
$B(d) \to T(d)$, see
\cite[Subsection 4.3]{PT0}. 
The following lemma is implicitly proved in~\cite[Section~5]{PT0}. 
\begin{lemma}
The derived stack $\mathcal{Z}$ is classical, i.e. we have an equivalence
\begin{equation}\label{ZZcl}
    \mathcal{Z}^{\mathrm{cl}}\stackrel{\sim}{\to}\mathcal{Z}.
\end{equation}
\end{lemma}
\begin{proof}
It suffices to show that $\nu^{-1}(0)$ is a classical scheme. The structure complex $\mathcal{O}_{\nu^{-1}(0)}$ is the Koszul complex of a section in $\Gamma(\mathfrak{n}^{\oplus 2}, \mathcal{O}_{\mathfrak{n}^{\oplus 2}}\otimes\overline{\mathfrak{n}})$, so it suffices to check that \begin{equation}\label{nunucl}
\dim \nu^{-1}(0)=\dim \nu^{-1}(0)^{\mathrm{cl}},\end{equation} see for example \cite[Subsection 1.4]{MR2573635}, where $\dim \nu^{-1}(0)=2\dim_{\mathbb{C}}\mathfrak{n}-\dim_{\mathbb{C}}\overline{\mathfrak{n}}$ is the dimension of $\nu^{-1}(0)$ as a quasi-smooth scheme.
The equality \eqref{nunucl} follows from \cite[Proposition 5.17]{PT0}.
\end{proof}
We set
\begin{align}\label{def:mi}
	m_i :=\left\lceil \frac{vi}{d} \right\rceil -\left\lceil \frac{v(i-1)}{d} \right\rceil
	+\delta_i^d -\delta_i^1 \in \mathbb{Z},
	\end{align}
	where $\delta^j_i$ is the Kronecker delta function: $\delta^j_i=1$ if $i=j$ and $\delta^j_i=0$ otherwise.
	Let $\lambda$ be the cocharacter 
\begin{align}\label{lambda:cochar}
	\lambda \colon \mathbb{C}^{\ast} \to GL(d), \ 
	t \mapsto (t^d, t^{d-1}, \ldots, t). 
	\end{align}
		 For a weight
	 $\chi=\sum_{i=1}^d n_i \beta_i$ with $n_i \in \mathbb{Z}$, 
		    we denote by $\mathbb{C}(\chi)$
		    the one dimensional $GL(d)^{\lambda \geq 0}$-representation given by 
		    \begin{align*}
		       GL(d)^{\lambda \geq 0}= B(d) \to GL(d)^{\lambda}=T(d) \stackrel{\chi}{\to} \mathbb{C}^{\ast},
		    \end{align*}
		    where the first morphism is the natural projection.
We define the object $\mathcal{E}_{d, v}$ by 
\begin{align}\label{def:Edw}
\mathcal{E}_{d, v}:=p_{\ast}\left(\mathcal{O}_{\mathcal{Z}} \otimes 
\mathbb{C}(m_1, \ldots, m_d)\right)
\in D^b(\mathscr{C}(d)). 
\end{align}
By \cite[Lemma 4.3]{PT0}, we have  
\begin{equation}\label{edvmagic}
    \mathcal{E}_{d,v}\in \mathbb{T}_{\mathbb{C}^3}(d)_v.
\end{equation}


\subsection{Objects in quasi-BPS categories for a local surface} 

Let $S$ be a smooth surface. 
The classical truncation of (\ref{dfilt})
for all $d_i=1$:
\[\mathcal{M}_S(1, \ldots, 1)=t_0(\mathfrak{M}_S(1, \ldots, 1))\] 
is the classical moduli stack of filtrations
(\ref{filt:Q}) 
such that
$Q_i/Q_{i-1}=\mathcal{O}_{x_i}$ for some 
$x_i \in S$. It admits a natural morphism to $S^{\times d}$
by sending the above filtration to $(x_1, \ldots, x_d)$. 
Define \[\mathcal{Z}_S:=\mathcal{M}_S(1, \ldots, 1) \times_{S^{\times d}} \Delta,\]
where $\Delta\cong S \subset S^{\times d}$ is the small diagonal and the fiber product is in the 
classical sense. 
Observe that \eqref{ZZcl} implies that \[\mathcal{Z}_{\mathbb{C}^2}\cong \mathcal{Z}.\] 
For $d\in \mathbb{N}$, let $T(d)$ act trivially on $S$. 
We have the commutative diagram 
\begin{align*}
    \xymatrix{
    \mathcal{Z}_S  \ar@/^18pt/[rrr]^-{p_S}
    \ar@<-0.3ex>@{^{(}->}[r]\ar[d]_-{q_S} \ar@{}[rd]|\square & \mathcal{M}_S(1, \ldots, 1) \ar[r] \ar[d] & 
    \mathcal{M}_S(d) \ar@<-0.3ex>@{^{(}->}[r] & \mathfrak{M}_S(d) \\
    S/T(d) \ar@<-0.3ex>@{^{(}->}[r]_-{\Delta} & S^{\times d}/T(d). & & 
    }
\end{align*}
Let $x\in S(\mathbb{C})$. Let $\widehat{\mathbb{C}}^2$ be the formal completion of $0$ in $\mathbb{C}^2$.
For a choice of an isomorphism 
$\widehat{\mathcal{O}}_{S, x} \cong \widehat{\mathcal{O}}_{\mathbb{C}^2, 0}$, there is a commutative diagram
\begin{align}\label{diagiso2}
    \xymatrix{
    (\nu^{-1}(0)/B(d)) \times \widehat{\mathbb{C}}^2 \ar[r]^-{\cong} \ar[d] & \widehat{\mathcal{Z}}_{S, x} \ar[d]
    \ar[r] \ar@{}[rd]|\square
    & \mathcal{Z}_S \ar[d] \\
    \widehat{\mathfrak{M}}_{\mathbb{C}^2}(d)_0 \ar[r]^-{\cong} & \widehat{\mathfrak{M}}_{S}(d)_x
    \ar[r] & \mathfrak{M}_S(d). 
    }
\end{align}
Let $(d, v)\in \mathbb{N}\times\mathbb{Z}$. 
Recall the definition of the integers $m_i$ from \eqref{def:mi}. Define the functor 
\begin{align*}
    \Phi_{d,v}\colon D^b(S)&\to D^b(\mathfrak{M}_{S}(d))_v,\\
    \mathcal{F}&\mapsto p_{S*}\left(q_S^*(\mathcal{F}\otimes \mathbb{C}(m_1, \ldots, m_d))\right).
\end{align*}

\begin{prop}
The image of the functor $\Phi_{d,v}$ is in the category $\mathbb{T}_X(d)_v$.
\end{prop}

\begin{proof}
It suffices to check the claim formally locally on $\mathrm{Sym}^d(S)$. The objects in the image of $\Phi_{d,v}$ are supported over the small diagonal $\Delta\cong S\hookrightarrow \mathrm{Sym}^d(S)$ by construction. Let $x\in S$ and consider the formal completions of spaces along the fiber over $d[x]\in\mathrm{Sym}^d(S)$ or $d[0]\in \mathrm{Sym}^d(\mathbb{C}^2)$. 
The conclusion then follows from 
the diagrams (\ref{diagiso2}) and \eqref{edvmagic}.
\end{proof}

\subsection{Dimensions of quasi-BPS categories}\label{sub:dimBPS} 
In this section, we prove Theorem \ref{thmtoric}. Let $S$ be a smooth toric surface. 
Then the two-dimensional torus $T=(\mathbb{C}^{\ast})^2$ acts on $S$, 
and the categorical Hall product (\ref{PS}) can be defined $T$-equivariantly. 
Recall that $\mathbb{K}:=K_0(BT)$, $\mathbb{F}:=\mathrm{Frac}\,\mathbb{K}$, and that for $V$ a $\mathbb{K}$-module we let $V_\mathbb{F}:=V\otimes_\mathbb{K}\mathbb{F}$. 
There is an associative algebra structure on 
\begin{align*}
    \mathrm{KHA}_{S, T}:=\bigoplus_{d \geq 0} G_T(\mathfrak{M}_S(d))_{\mathbb{F}}
    \stackrel{\cong}{\leftarrow}\bigoplus_{d \geq 0} G_T(\mathcal{M}_S(d))_{\mathbb{F}}
    . 
\end{align*}

We begin with a few preliminaries. 
Let $S^T$ be the (finite) set of $T$-fixed points on $S$. 
For $p \in \mathrm{Sym}^d(S)$, denote by $\mathcal{M}_S(d)_p$ the fiber over $p$ along the map $\pi_d\colon \mathcal{M}_{S}(d)\to \mathrm{Sym}^d(S)$. 
We have the following lemma: 
\begin{lemma}\label{lem:KHA}
    There is an isomorphism 
    \begin{align}\label{isomultiple}
        \mathrm{KHA}_{S, T} \cong \bigotimes_{x\in S^T} \left(\bigoplus_{d\geq 0}G_T\big(\mathcal{M}_{S}(d)_{d[x]}\big)_\mathbb{F}\right). 
    \end{align}
\end{lemma}
\begin{proof}
The $T$-fixed loci of $\mathcal{M}_{S}(d)$ are included in
fibers over $T$-fixed loci of $\mathrm{Sym}^d(S)$. 
The $T$-fixed locus
$\mathrm{Sym}^d(S)^T$
consists of points $p=\sum_{j=1}^m d^{(j)} x^{(j)}$
such that $x^{(j)} \in S^T$, $x^{(1)}, \ldots, x^{(m)}$
are distinct, 
and $d^{(1)}+\cdots +d^{(m)}=d$. 
By the localization theorem in K-theory \cite{MR1297442}, there is an isomorphism
\begin{align*} 
G_T\left(\mathcal{M}_{S}(d)\right)_\mathbb{F}&\cong G_T\left(\bigsqcup_{p\in \mathrm{Sym}^d(S)^T}\mathcal{M}_{S}(d)_p\right)_\mathbb{F}.
\end{align*}
Moreover, we have 
\begin{align*}
    \mathcal{M}_S(d)_p=\prod_{j=1}^m \mathcal{M}_S(d^{(j)})_{d^{(j)}[x^{(j)}]}. 
\end{align*}
We thus obtain the desired isomorphism. 
\end{proof}
For $x\in S^T$,
define \[\Phi(d)_{v,x}:=\Phi_{d,v}(\mathcal{O}_x)\in \mathbb{T}_X(d)_v.\]
We abuse notation and also denote by $\Phi(d)_{v,x}$ its class in K-theory.

\begin{prop}\label{prop38}
(a) Let $x$ and $x'$ be points in $S^T$, not necessarily distinct, and consider two pairs $(d, v)$ and $(d', v')$ such that $v/d=v'/d'$. Then 
\begin{equation}\label{relation}
\Phi(d)_{v,x}\Phi(d')_{v',x'}=\Phi(d')_{v',x'}\Phi(d)_{v,x}.
\end{equation}(b) The $\mathbb{F}$-vector space $G_T\big(\mathfrak{M}_{S}(d)\big)_{\mathbb{F}}$ has an $\mathbb{F}$-basis given by equivalences classes of monomials \begin{equation}\label{def:monomial}
    \Phi(d_1)_{v_1,x_1}\cdots \Phi(d_k)_{v_k,x_k},\end{equation} where $x_i\in S^T$ for $1\leq i\leq k$, the inequality $v_1/d_1 \leq\cdots\leq v_k/d_k$ holds, and $\sum_{i=1}^k d_i=d$, where two monomials \eqref{def:monomial} are equivalent if they differ by relations \eqref{relation}. 
\end{prop}

\begin{proof}

Let $x\in S^T$. 
We argue that the Hall product \eqref{PS} induces an algebra structure on 
\begin{equation}\label{KHAp}
\mathrm{KHA}_{S, T, x}:=\bigoplus_{d\geq 0}G_T(\mathcal{M}_S(d)_{d[x]})_\mathbb{F}.
\end{equation}
First, by the localization theorem in K-theory \cite{MR1297442}, there are injective maps $G_T(\mathcal{M}_S(d)_{d[x]})_\mathbb{F}\hookrightarrow G_T(\mathcal{M}_S(d))_\mathbb{F}$, and thus 
\[G_T(\mathcal{M}_S(d)_{d[x]})_\mathbb{F}=\mathrm{Ker}\left(G_T(\mathcal{M}_S(d))_\mathbb{F}\xrightarrow{\mathrm{res}} G_T(\mathcal{M}_S(d)\setminus \mathcal{M}_S(d)_{d[x]})_\mathbb{F} \right).\]
The following composition of the natural pushforward, of the Hall product, and of the natural restriction is 
a zero map: 
\begin{align*}
    G_T(\mathcal{M}_S(d)_{d[x]})_\mathbb{F}\otimes G_T(\mathcal{M}_S(e)_{e[x]})_\mathbb{F} &\to G_T(\mathcal{M}_S(d))_\mathbb{F}\otimes G_T(\mathcal{M}_S(e))_\mathbb{F}\\ 
    &\xrightarrow{\ast}G_T(\mathcal{M}_S(d+e))_\mathbb{F} \\
    &\to G_T(\mathcal{M}_S(d+e)\setminus \mathcal{M}_S(d+e)_{(d+e)[x]})_\mathbb{F}.
\end{align*}
There is thus a Hall product 
\[\ast\colon G_T(\mathcal{M}_S(d)_{d[x]})_\mathbb{F}\otimes G_T(\mathcal{M}_S(e)_{e[x]})_\mathbb{F}\to G_T(\mathcal{M}_S(d+e)_{(d+e)[x]})_\mathbb{F}\] and \eqref{KHAp} is indeed an algebra.
The algebras $\mathrm{KHA}_{S, T, x}$ and $\mathrm{KHA}_{S, T, x'}$ commute for $x\neq x'\in S^T$
inside $\mathrm{KHA}_{S, T}$. Further, by the isomorphism \eqref{diagiso2} and the localization theorem in K-theory \cite{MR1297442}, we have that \[\mathrm{KHA}_{S, T, x}\cong \mathrm{KHA}_{\mathbb{C}^2, T, 0}\cong \mathrm{KHA}_{\mathbb{C}^2, T}.\] 
Part (a) then follows by \cite[Proposition 5.7]{PT1}.
By \cite[Equations (4.31) and (4.36)]{PT0}, the $\mathbb{F}$-algebra $\mathrm{KHA}_{\mathbb{C}^2, T, 0}\cong \mathrm{KHA}_{\mathbb{C}^2, T}$ has an $\mathbb{F}$-basis with unordered monomials \[\Phi(d_1)_{v_1,0}\ldots\Phi(d_k)_{v_k,0}\] for decompositions $\sum_{i=1}^k d_i=d$ and integers $v_i$ for $1\leq i\leq k$ satisfying the inequality $v_1/d_1 \leq\cdots\leq v_k/d_k$. 
By Lemma~\ref{lem:KHA}, the conclusion of (b) follows.
\end{proof}

\begin{cor}\label{cor39}
The $\mathbb{F}$-vector space $K_T\big(\mathbb{T}_X(d)_v\big)_{\mathbb{F}}$ has an $\mathbb{F}$-basis given by equivalence classes of monomials $\Phi(d_1)_{v_1,x_1}\cdots \Phi(d_k)_{v_k,x_k}$ for partition $\sum_{i=1}^k (d_i, v_i)=(d, v)$, points $x_i\in S^T$ for $1\leq i\leq k$, and $v_1/d_1=\cdots=v_k/d_k$, where two monomials are equivalent if they differ by relations \eqref{relation}. 
\end{cor}

\begin{proof}
The same argument as the one in \cite[Proof of Theorem 4.12]{PT0} (which treats the case $S=\mathbb{C}^2$) applies here and is based on the isomorphism \eqref{isomultiple}, Proposition \ref{prop38},
   and on the $T$-equivariant version of Theorem \ref{SODsurface}.
\end{proof}

\begin{proof}[Proof of Theorem \ref{thmtoric}]
Let $(d', v')\in \mathbb{N}\times\mathbb{Z}$ be coprime integers such that $n(d', v')=(d, v)$. Using \cite[Lemma 4.8]{PT0}, the Hall product induces an algebra structure on 
\begin{align}\label{alg:KT}
\bigoplus_{k\geq 0} K_T(\mathbb{T}_X(kd')_{kv'})_\mathbb{F}.
\end{align}
By Corollary \ref{cor39}, the above algebra 
has an $\mathbb{F}$-basis with equivalences classes of monomials $\Phi(k_1d')_{k_1v', x_1}\ldots \Phi(k_rd')_{k_rv', x_r}$ for all partitions $k_1+\cdots+k_r=k$ and choice of points $x_i\in S^T$ for $1\leq i\leq r$. 
Let $a_k:=\dim_\mathbb{F} K_T(\mathbb{T}_X(kd')_{kv'})_\mathbb{F}$ and let $b_k$ be the number of partitions of $k$. Then 
\begin{align*}
    \sum_{k\geq 0}a_kq^k=\left(\sum_{k\geq 0}b_kq^k\right)^{|S^T|}
    =\left(\prod_{b\geq 1}\frac{1}{1-q^b}\right)^{\chi_c(S)}
    =\sum_{k\geq 0}\chi_c(\mathrm{Hilb}(S, k))q^k.
\end{align*}
The first equality holds from the description of (\ref{alg:KT}), 
the second equality 
holds because $|S^T|=\chi_c(S)$
and the known product formula for the number of partitions, and the 
last equality
follows from \cite{MR1032930}.  
\end{proof}

\begin{remark}
Recall Conjecture \ref{mainconj}. We may also conjecture the analogous statement 
\[\chi_{\mathrm{HP}}(\mathbb{T}_X(d)_w)=\chi_c(\mathrm{Hilb}(S, n)),\]
where the left hand side is the Euler characteristic of the periodic cyclic homology of $\mathbb{T}_X(d)_w$, which is two-periodic: 
\[\chi_{\mathrm{HP}}(\mathbb{T}_X(d)_w):=\dim_{\mathbb{C}(\!(u)\!)} \mathrm{HP}_0(\mathbb{T}_X(d)_w)-\dim_{\mathbb{C}(\!(u)\!)} \mathrm{HP}_1(\mathbb{T}_X(d)_w).\]
\end{remark}

\subsection{Singular support of objects in quasi-BPS categories}

Let $S$ be a smooth surface, let $X=\mathrm{Tot}_S(K_S)$, and let $d\in \mathbb{N}$. 
Let $\mathcal{M}_X(d)$ be the classical 
moduli stack of zero-dimensional sheaves on $X$ with 
length $d$. By~\cite[Lemma~3.4.1]{T}, we have 
\begin{align*}
    \mathcal{M}_X(d)=
    t_0(\Omega_{\mathfrak{M}_S(d)}[-1]).
    \end{align*}
In particular, 
an object $F\in D^b(\mathfrak{M}_{S}(d))$ has a singular support $\mathrm{Supp}^{\rm{sg}}(F)\subset \mathcal{M}_{X}(d)$, see Subsection \ref{singularsupport}. Consider the diagram
\begin{equation*}
    \begin{tikzcd}
    \pi_d^{-1}(X)\arrow[d, "\pi_d"]\arrow[r, hook]& \mathcal{M}_{X}(d)\arrow[d, "\pi_d"]\\
    X\arrow[r, hook, "\Delta"]& \mathrm{Sym}^d(X),
    \end{tikzcd}
\end{equation*}
where $\Delta\colon X\hookrightarrow \mathrm{Sym}^d(X)$ is the small diagonal map $x\mapsto (x,\ldots, x)$.
Davison~\cite[Theorem~5.1]{Dav} showed that the BPS sheaf for the moduli stack of degree $d$ sheaves on $\mathbb{C}^3$ is 
\[\mathcal{B}PS(d)=\Delta_*\mathrm{IC}_{\mathbb{C}^3}.\]
In \cite[Theorem 1.1]{PT0}, we proved a categorical version of the above result, namely that the singular support of an object
in $\mathbb{T}_{\mathbb{C}^3}(d)_w$ is contained in $\pi_d^{-1}(\mathbb{C}^3)$.
Using the formal local description of $\mathbb{T}_X(d)_w$ and \cite[Theorem 3.1]{PT0}, we obtain the analogous result for $X$:

\begin{prop}
Let $w\in \mathbb{Z}$ be such that $\gcd(d, w)=1$ and
    let $F\in \mathbb{T}_X(d)_w$. Then $\mathrm{Supp}^{\rm{sg}}(F)\subset \pi_d^{-1}(X)$.
\end{prop}

\begin{proof}
 Let $\tau\colon X\to S$, $\tau\colon \mathbb{C}^3=\mathrm{Tot}_{\mathbb{C}^2}(K_{\mathbb{C}^2})\to\mathbb{C}^2$ be the natural projection maps. Consider the maps
 \begin{equation}\label{map:sym}
     \begin{tikzcd}
     \mathcal{M}_{X}(d)\arrow[d, "\tau"]\arrow[r, "\pi^3_d"]& \mathrm{Sym}^d(X)\arrow[d, "\tau"]\\ \mathcal{M}_{S}(d)\arrow[r, "\pi_d^2"]& \mathrm{Sym}^d(S).
     \end{tikzcd}
 \end{equation}
 We also use the notations $\pi^2_d$ and $\pi^3_d$ for maps of moduli spaces constructed from $\mathbb{C}^3$.
 Let $p\in \mathrm{Sym}^d(S)$ be such that $\mathrm{Supp}^{\rm{sg}}(F)\cap (\pi^3_d\tau)^{-1}(p)\neq \emptyset$. 
 Write $p=\sum_{i=1}^m d^{(i)}x^{(i)}$ for $m\geq 1$, $x^{(i)}\neq x^{(j)}$ points in $S$ for $1\leq i, j\leq m$, and $d^{(i)}\geq 1$. 
 There is a point $q=\sum_{i=1}^m d^{(i)}y^{(i)}\in \mathrm{Sym}^d(\mathbb{C}^2)$ for points $y^{(i)}\neq y^{(j)}$ in $\mathbb{C}^2$ for $1\leq i\neq j\leq m$, together with the following commutative diagram 
 
\begin{align}\label{diagXc2}
    \xymatrix{
      &    &   \widehat{\mathcal{M}}_{\mathbb{C}^3}(d)_q \ar[r]^-{\pi_d^3} \ar[d]^(.3){\tau}|\hole &  \widehat{\mathrm{Sym}}^d(\mathbb{C}^3)_q\ar[d]_-{\tau} \\
      \widehat{\mathcal{M}}_{X}(d)_p \ar[rru]^-{\cong}_-{\alpha} \ar[r]_-{\pi_d^3} \ar[d]_-{\tau} 
      &  \widehat{\mathrm{Sym}}^d(X)_p 
      \ar[d]_(.3){\tau} \ar[rru]^(.3){\cong} & \widehat{\mathcal{M}}_{\mathbb{C}^2}(d)_q \ar[r]^-{\pi_d^2} & \widehat{\mathrm{Sym}}^d(\mathbb{C}^2)_q \\
       \widehat{\mathcal{M}}_{S}(d)_p \ar[rru]^(.3){\cong}|\hole \ar[r]_-{\pi_d^2} & \widehat{\mathrm{Sym}}^d(S)_p  
       \ar[rru]_-{\cong} & & &
      }
\end{align}

  In the above, $\widehat{\mathrm{Sym}}^d(S)_p$ is the completion of $\mathrm{Sym}^d(S)$ at $p$ and all the spaces
  in the front square is a base-change 
  of the diagram (\ref{map:sym})
 by $\widehat{\mathrm{Sym}}^d(S)_p \to \mathrm{Sym}^d(S)$, 
 and similarly for the back square. 
 The commutativity of
 the bottom square is classical (see the constructions in Subsection \ref{subquasi-BPS}). 
 The left square is commutative since the 
 isomorphism $\widehat{\mathcal{M}}_S(d)_p \stackrel{\cong}{\to} \widehat{\mathcal{M}}_{\mathbb{C}^2}(d)_q$
  extends to an equivalence
 \begin{align}\label{equiv:MSp}
 \widehat{\mathfrak{M}}_S(d)_p
 \stackrel{\sim}{\to}\widehat{\mathfrak{M}}_{\mathbb{C}^2}(d)_q
 \end{align}
 and we have the  
 induced isomorphisms $\alpha$ on $(-1)$-shifted cotangent stacks. 
 
 The support of sheaves is respected under the isomorphisms in \eqref{diagXc2}.
 Further, (\ref{equiv:MSp}) induces an equivalence
 \[\mathbb{T}_X(d)_{w, p}\stackrel{\sim}{\to} \mathbb{T}_{\mathbb{C}^3}(d)_{w, q}.\] 
 Let $E\in \mathbb{T}_{\mathbb{C}^3}(d)_{w, q}$ be the image of $F|_{\widehat{\mathfrak{M}}_S(d)_p}$ under the above equivalence. Then \[\alpha\left(\mathrm{Supp}^{\rm{sg}}(F|_{\widehat{\mathfrak{M}}_S(d)_p})\right)=
 \alpha\left(\mathrm{Supp}^{\rm{sg}}(F) \times_{\mathcal{M}_S(d)} \widehat{\mathcal{M}}_S(d)_p  \right)=
 \mathrm{Supp}^{\rm{sg}}(E).\] In the proof of \cite[Theorem 3.1]{PT1}, we showed that $\mathrm{Supp}^{\rm{sg}}(E)\subset (\pi^3_d)^{-1}(\Delta)$. The equivalence (\ref{equiv:MSp}) respect the singular support, and we thus obtain the desired conclusion.
\end{proof}

\section{The categorical DT/PT correspondence for local surfaces}
In this section, we prove Theorem~\ref{thm:main}. 

\subsection{DT/PT categories for local surfaces}\label{sub:dtptdef}
Let $S$ be a smooth projective surface. Its corresponding local surface is the non-compact Calabi-Yau 3-fold
\begin{align*}
	X:= \mathrm{Tot}_S(K_S) \stackrel{\tau}{\to} S. 
	\end{align*}
	Here, $\tau$ is the natural projection. 
Let $\beta \in H_2(S, \mathbb{Z})$ and $n \in \mathbb{Z}$. Recall the 
DT moduli space
\begin{align*}
    I_X(\beta, n)
\end{align*}
which parametrizes ideal sheaves 
$I_C$ for $C \subset X$ compactly supported and with $\dim C \leq 1$, 
$\tau_{\ast}[C]=\beta$, and $\chi(\mathcal{O}_C)=n$.
Recall also the PT moduli space \begin{align*}
    P_X(\beta, n)
\end{align*}
of PT stable pairs $(\mathcal{O}_X \stackrel{s}{\to} F)$,
where $F$ is a compactly supported pure one-dimensional sheaf on $X$, 
and $s$ has at most zero-dimensional cokernel
satisfying $\tau_{\ast}[F]=\beta$ and $\chi(F)=n$, where $[F]$ is the 
fundamental one-cycle associated with $F$.
	
We denote by $\mathfrak{M}_S^{\dag}(\beta, n)$ the derived moduli stack of 
pairs $(F, s)$, where $F \in \Coh_{\leq 1}(S)$ has one-dimensional support 
$\beta$, $\chi(F)=n$, and $s \colon \mathcal{O}_S \to F$ is a morphism.   
We refer to~\cite[Section~4.1.1]{T} for the 
derived structure on $\mathfrak{M}_S^{\dag}(\beta, n)$. 
The derived stack $\mathfrak{M}_S^{\dag}(\beta, n)$ is quasi-smooth, 
and its $(-1)$-shifted cotangent stack
\begin{align*}
	\mathcal{M}_X^{\dag}(\beta, n) :=t_0(\Omega_{\mathfrak{M}_S^{\dag}(\beta, n)}[-1])
	\stackrel{\tau_{\ast}}{\to} \mathfrak{M}_S^{\dag}(\beta, n)
	\end{align*}
	is isomorphic to the moduli stack of 
	objects in the extension closure 
	\begin{align*}
\mathcal{A}_X=\langle \mathcal{O}_X, \Coh_{\leq 1}(X)[-1] \rangle_{\rm{ex}},
\end{align*}
see \cite[Theorem 4.1.3.(ii)]{T}.
The category $\mathcal{A}_X$ is called \textit{the category of
	D0-D2-D6 bound states} in~\cite[Subsection~4.1.2]{T}. 
	In particular, there are open immersions
	\begin{align*}
	    I_X(\beta, n) \subset \mathcal{M}_X^{\dag}(\beta, n) \supset P_X(\beta, n). 
	\end{align*}
		In general, the stack $\mathfrak{M}_S^{\dag}(\beta, n)$ is too big, for example it is not of finite type. 
Let $\mathfrak{M}_S^{\dag}(\beta, n)_{\rm{fin}} \subset \mathfrak{M}_S^{\dag}(\beta, n)$
be an open substack of finite type which contains 
both of $\tau_{\ast}(I_X(\beta, n))$ and $\tau_{\ast}(P_X(\beta, n))$. 
Then there are open immersions
\begin{align*}
    I_X(\beta, n) \subset t_0(\Omega_{\mathfrak{M}_S^{\dag}(\beta, n)_{\rm{fin}}}[-1])
    \supset P_X(\beta, n). 
\end{align*}
Let $\mathcal{Z}_I, \mathcal{Z}_P$ be the complements of the above open immersions. 
Following the construction in Subsection~\ref{subsec:ssupport}, the 
DT/PT categories are defined as follows: 
\begin{defn}(\cite[Definition~4.2.1]{T})\label{def:DTPTcat}
	The DT/PT categories are defined to be the following singular 
	support quotients: 
	\begin{align*}
		&\mathcal{DT}_X(\beta, n):=
		D^b(\mathfrak{M}_S^{\dag}(\beta, n)_{\rm{fin}})/\mathcal{C}_{\mathcal{Z}_I}, \\ 
		&\mathcal{PT}_X(\beta, n):=
		D^b(\mathfrak{M}_S^{\dag}(\beta, n)_{\rm{fin}})/\mathcal{C}_{\mathcal{Z}_P}. 
		\end{align*}
	\end{defn}
	\begin{remark}
	The DT/PT categories in Definition~\ref{def:DTPTcat}
are independent of a choice of $\mathfrak{M}_S^{\dag}(\beta, n)_{\rm{fin}}$
up to equivalence, see~\cite[Lemma~3.2.9]{T}.
\end{remark}
	\subsection{The moduli of pairs for reduced curve classes}
Suppose that $\beta$ is a reduced class, i.e. any effective divisor on $S$ with 
class $\beta$ is a reduced divisor. 
In this case, we can take $\mathfrak{M}_S^{\dag}(\beta, n)_{\rm{fin}}$
which admits a good moduli space as we explain below. 
Let 
\begin{align*}
\mathfrak{T}_S(\beta, n) \subset \mathfrak{M}_S^{\dag}(\beta, n)
\end{align*}
be the open substack consisting of $(F, s)$ such that $s$ has at most zero-dimensional 
cokernel. 
If $\beta$ is reduced, then 
both of $\tau_{\ast}(I_n(X, \beta))$ and $\tau_{\ast}(P_n(X, \beta))$ are contained in 
$\mathfrak{T}_S(\beta, n)$. 
Indeed, when $\beta$ is a reduced curve class, by~\cite[Lemma~5.5.4]{T} we have that
\begin{align*}
\mathcal{T}_X(\beta, n):=t_0(\Omega_{\mathfrak{T}_S(\beta, n)}[-1])
\end{align*}
is isomorphic to the moduli stack of pairs $(\mathcal{O}_X \stackrel{s}{\to} E)$, 
where $E \in \Coh_{\leq 1}(X)$ is compactly supported and $s$ 
has at most zero-dimensional cokernel. 
Moreover, the stack $\mathfrak{T}_S(\beta, n)$ is of finite type, as the 
pairs $(F, s)$ with zero-dimensional cokernel are bounded. 
Therefore, if $\beta$ is reduced, we can take 
\begin{align}\label{MdagT}
    \mathfrak{M}_S^{\dag}(\beta, n)_{\rm{fin}}=\mathfrak{T}_S(\beta, n). 
\end{align}
In what follows, we take $\beta$ to be reduced and $\mathfrak{M}_S^{\dag}(\beta, n)_{\rm{fin}}$
to be (\ref{MdagT}). 

Let $\mathcal{T}_S(\beta, n)$ be the classical truncation of 
$\mathfrak{T}_S(\beta, n)$. 
As discussed in~\cite[Section~4.2.1, Section~6.3.2]{T}, 
the stack $\mathcal{T}_S(\beta, n)$ admits a good moduli space
\begin{align}\label{mor:TS}
	\pi_S\colon \mathcal{T}_S(\beta, n) \to T_S(\beta, n)
	=\mathrm{Chow}^{\beta}(S) \times \mathrm{Sym}^d(S).
	\end{align}
In the above, $\mathrm{Chow}^{\beta}(S)$	
	is the moduli space of effective divisors
	in $S$ with class $\beta$ 
	(which is 
	a stratified projective bundle 
	over $\mathrm{Pic}^{\beta}(S)$), 
	and $d$ is determined by 
	\begin{align}\label{computed}
	    -\frac{1}{2}\beta^2-\frac{1}{2}K_S \beta+d=n. 
	\end{align}
	The first two terms correspond to 
	$\chi(\mathcal{O}_C)$ for $[C] \in \mathrm{Chow}^{\beta}(S)$. 
	The morphism (\ref{mor:TS}) is given by 
	\begin{align*}
	(\mathcal{O}_S \stackrel{s}{\to} F)
	\mapsto (\mathrm{Supp}(F), \mathrm{Supp}(F_{\rm{tor}})
	+\mathrm{Supp}(\mathrm{Cok}(s))),
	\end{align*}
	where $F_{\rm{tor}} \subset F$ is the maximal zero-dimensional subsheaf. 
	
	For $y=(C, p) \in T_S(\beta, n)$
	with $p=\sum_{j=1}^m d^{(j)}x^{(j)}$
	such that 
	$x^{(1)}, \ldots, x^{(m)} \in S$ are
	mutually distinct, 
the unique closed point 
in the fiber of (\ref{mor:TS}) at $y$ 
corresponds to a direct sum 
\begin{align}\label{rep:dsum}
	I=(\mathcal{O}_S \to F)=(\mathcal{O}_S \twoheadrightarrow \mathcal{O}_C) \oplus \bigoplus_{j=1}^m 
	V^{(j)} \otimes (0 \to \mathcal{O}_{x^{(j)}}),
	\end{align}
	where
$V^{(j)}$ is a finite dimensional vector space with $\dim V^{(j)}=d^{(j)}$.

\subsection{Window categories for DT/PT categories}

\subsubsection{Construction of window categories}
Let $\beta$ be a reduced class. 
Consider the universal pair 
\begin{align*}
\mathcal{I}=(\mathcal{O}_{S\times\mathcal{T}_S(\beta,n)}\to \mathcal{F}) \in D^b(S \times \mathcal{T}_S(\beta, n)). 
\end{align*}
Denote by $\pi_\mathcal{T}\colon S\times\mathcal{T}_S(\beta,n)\to \mathcal{T}_S(\beta,n)$ the projection onto the second factor. Let \[\mathscr{L}:=\det\left(Rp_{\mathcal{T}*}\mathcal{F}\right)\in\mathrm{Pic}(\mathcal{T}_S(\beta,n)).\]
By \cite[Proposition 5.5.2]{T}, we have that
\[I_X(\beta, n)=\mathcal{T}_X(\beta, n)^{\mathscr{L}\text{-ss}},\, P_X(\beta, n)=\mathcal{T}_X(\beta, n)^{\mathscr{L}^{-1}\text{-ss}}.\]
We further take 
\begin{align*} 
b:=2\ch_2(Rp_{\mathcal{T}*}\mathcal{F}) \in H^4(\mathcal{T}_S(\beta, n), \mathbb{Q}). 
\end{align*}
Its pull-back to $H^4(\mathcal{T}_X(\beta, n), \mathbb{Q})$ is also denoted by $b$
and it is positive definite, see the argument of~\cite[Lemma~7.25]{Totheta}. 
Then $(\mathcal{L}, b)$
determine $\Theta$-stratifications
\begin{align}\label{theta:strata}
    \mathcal{T}_X(\beta, n)&=\mathcal{S}_0^I \sqcup \cdots \sqcup \mathcal{S}_d^I \sqcup I_X(\beta, n) \\
 \notag &=\mathcal{S}_0^P \sqcup \cdots \sqcup \mathcal{S}_d^P \sqcup P_X(\beta, n),
  \end{align}
  where $\mathcal{S}_i^I$ and $\mathcal{S}_i^P$ are given by 
  \begin{align*}
      &\mathcal{S}_i^I=\{(\mathcal{O}_X \stackrel{s}{\to} F) : \chi(\mathrm{Cok}(s))=d-i\}, \\ 
      &\mathcal{S}_i^P=\{(\mathcal{O}_X \stackrel{s}{\to} F) : \chi(F_{\rm{tor}})=d-i\},
  \end{align*}
  where $d$ is given by (\ref{computed}).
  The same stratifications are constructed in~\cite[Subsection~6.2.1]{T} without using $b$, 
  and it is straightforward to show that they are $\Theta$-stratifications 
  with respect to $(\mathcal{L}, b)$. 

Since $\mathcal{T}_S(\beta, n)$ admits 
a good moduli space, 
we can apply Theorem~\ref{thm:window:M} to obtain the following
window theorem for DT/PT categories
(which is also obtained in~\cite[Theorem~1.4.6]{T}): 
\begin{thm}\emph{(\cite[Theorem~1.4.6]{T})}\label{thmtodawindow}
	For each maps $k_I, k_P \colon \mathbb{Z}_{\geq 1} \to \mathbb{R}$, there exist
	dg-subcategories 
	\begin{align}\label{subcat:W}
		\mathbb{W}_I \subset D^b(\mathfrak{T}_S(\beta, n)), 
		\mathbb{W}_P \subset D^b(\mathfrak{T}_S(\beta, n))
		\end{align}
	such that 
	the following composition functors are equivalences 
	\begin{align*}
		&\mathbb{W}_I \hookrightarrow D^b(\mathfrak{T}_S(\beta, n))
		\twoheadrightarrow \mathcal{DT}_X(\beta, n) \\
		&\mathbb{W}_P \hookrightarrow D^b(\mathfrak{T}_S(\beta, n))
		\twoheadrightarrow \mathcal{PT}_X(\beta, n). 
		\end{align*}
	\end{thm}

  \begin{remark}
  In~\cite[Section~6]{T}, the existence of stratifications (\ref{theta:strata}) 
  and Theorem~\ref{thmtodawindow} are proved 
  without using $b$, so in the arguments below 
  the use of $b$ is not relevant. 
  \end{remark}
	
	\subsubsection{The formal local descriptions of window categories}
The subcategories (\ref{subcat:W}) are characterized in terms of Koszul 
duality equivalences on each formal fiber 
along with the good moduli space $\mathcal{T}_S(\beta, n)\to T_S(\beta, n)$, see \eqref{wt:cond}, \eqref{Wasty}, and Proposition \ref{prop:local}. 
Below we make it more explicit. 
Let $y \in T_S(\beta, n)$ be a closed point 
represented by a direct sum (\ref{rep:dsum}). 
We also denote by $y \in \mathcal{T}_S(\beta, n)$
the corresponding closed point. 
Then 
\begin{align*}
	G_y := \mathrm{Aut}(y)=\prod_{j=1}^m \GL(V^{(j)}), \ 
	T_y=\prod_{j=1}^m T(d^{(j)}) \subset G_y
	\end{align*}
	where $T_y$ is the maximal torus of $G_y$. 
The derived stack $\mathfrak{T}_S(\beta, n)$ 
satisfies the formal neighborhood theorem, see~\cite[Lemma~7.4.3]{T}. 
Moreover, we have 
\begin{align}\label{htangent}
    \mathcal{H}^i(\mathbb{T}_{\mathfrak{T}_S(\beta, n)}|_{y})=\Hom^i_S(I, F)
\end{align} for every $i\in\mathbb{Z}$.
	Here, we regard
	$I$ as a two term complex 
	$(\mathcal{O}_S \to F)$
	with $\mathcal{O}_S$ located in degree zero, 
	and $\Hom$ denotes morphisms in the derived category $D^b(S)$. 
Then the formal fiber $\widehat{\mathfrak{T}}_S(\beta, n)_y$
of $\mathfrak{T}_S(\beta, n)$ at $y \in T_S(\beta, n)$
is a quotient stack by the $G_y$-action 
of the derived zero locus of a $G_y$-equivariant morphism 
\begin{align*}
	\kappa_y \colon \widehat{\Hom}_S(I, F) \to \Hom^1_S(I, F). 
	\end{align*}
We set
\begin{align*}
	\mathcal{X}_y:=\left(\widehat{\Hom}_S(I, F) \oplus  \Hom^1_S(I, F)^{\vee}\right)/G_y.
	\end{align*}
Consider the following regular function 
\begin{align*}
	f_y \colon \mathcal{X}_y \to \mathbb{C}, \ 
	(x, u) \mapsto \langle \kappa_y(x), u \rangle.
\end{align*}
The 
 Koszul duality in Theorem~\ref{thm:Kduality} gives the equivalence 
\begin{align}\label{Kos:y}
	\Phi_y \colon D^b(\widehat{\mathfrak{T}}_S(\beta, n)_y) \stackrel{\sim}{\to}
	\text{MF}^{\text{gr}}(\mathcal{X}_y, f_y). 	\end{align}
	

The following lemma shows that 
the stack $\mathcal{X}_y$ is a completion of 
the product of the stack of representations of DT/PT
quivers considered in~\cite{PT3}. 
\begin{lemma}\label{lem:Hom}
	As a $G_y$-representation, we have  
\begin{align*}
	&\Hom_S(I, F) \oplus \Hom^1_S(I, F)^{\vee} \\
	&=H^0(\mathcal{O}_C(C)) \oplus H^1(\mathcal{O}_C(C))^{\vee}
	\oplus \bigoplus_{j=1}^m\left( (V^{(j)})^{\oplus a^{(j)}+1} \oplus (V^{(j)\vee})^{\oplus a^{(j)}} 
	\oplus \mathrm{End}(V^{(j)})^{\oplus 3}\right)
	\end{align*}
where $a^{(j)}=1$ if $x^{(j)} \in C$ and $a^{(j)}=0$ if $x^{(j)} \notin C$, and where $G_y$ acts trivially on the first two summands and acts naturally on the other summands.
\end{lemma}
\begin{proof}
Note that from (\ref{rep:dsum}) we have 
\begin{align*}
    F=\mathcal{O}_C \oplus \bigoplus_{j=1}^m V^{(j)} \otimes \mathcal{O}_{x^{(j)}}. 
\end{align*}
By direct calculation, we have  
	\begin{align}\label{HomIF}
		\Hom_S(I, F)=&H^0(\mathcal{O}_C(C))\oplus \\
	\notag	&
	\bigoplus_{j=1}^m\left( V^{(j)} \oplus (V^{(j)\vee} 
		\otimes \Hom_S^1(\mathcal{O}_{x^{(j)}}, \mathcal{O}_C))
		\oplus \mathrm{End}(V^{(j)})^{\oplus 2}\right), \\
		\notag	
		\Hom^1_S(I, F)=&H^1(\mathcal{O}_C(C))\oplus \\
		\notag	&
		\bigoplus_{j=1}^m\left((V^{(j)\vee} 
		\otimes \Hom_S^2(\mathcal{O}_{x^{(j)}}, \mathcal{O}_C))
		\oplus \mathrm{End}(V^{(j)})\right). 
		\end{align}
	Since $C$ is a reduced divisor, we have 
\begin{align*}	
	\Hom_S^1(\mathcal{O}_{x^{(j)}}, \mathcal{O}_C)=
	\Hom_S^2(\mathcal{O}_{x^{(j)}}, \mathcal{O}_C)=\begin{cases}
		\mathbb{C}, & x^{(j)} \in C, \\
		0, & x^{(j)} \notin C. 
		\end{cases}
	\end{align*}
\end{proof}
We denote by $\chi_0^{(j)}$ and $\chi_0$ the determinant characters
\begin{align}\label{dchar}
	\chi_0^{(j)} \colon \GL(V^{(j)}) \to \mathbb{C}^{\ast}, \ 
	g^{(j)} \mapsto \det (g^{(j)}), \ 
	\chi_0=\prod_{j=1}^m \chi_0^{(j)} \colon G_y \to \mathbb{C}^{\ast}. 
	\end{align}
	The data
	$(\mathcal{L}, b)$
	pulled back to $\mathrm{Crit}(f_y)$ is induced from 
	$BG_y$. By writing it as $(\mathcal{L}_y, b_y)$, 
	we have $\mathcal{L}_y=\chi_0$ 
	and $b_y \in H^4(BG_y, \mathbb{Q})$ corresponds to the 
	Weyl-invariant norm on the set of cocharacters $\lambda=\prod_{j=1}^m (\lambda_i^{(j)})_{i=1}^{d^{(j)}}$
	of $\prod_{j=1}^m T(d^{(j)})$
given by 
\begin{align}\label{Wnorm}
    \lvert \lambda \rvert^2=\sum_{j=1}^m \sum_{i=1}^{d^{(j)}} (\lambda_i^{(j)})^2. 
\end{align}
We have the Kempf-Ness stratifications 
\begin{align}\label{KN:X}
	\mathcal{X}_y=\mathscr{S}_1^{\pm} \sqcup \cdots \sqcup \mathscr{S}_{N^{\pm}}^{\pm}
	\sqcup \mathcal{X}_y^{\chi_0^{\pm}\mathchar`-\rm{ss}}
	\end{align}
of $\mathcal{X}_y$ with respect to $\chi_0^{\pm 1}$ and a Weyl-invariant 
norm (\ref{Wnorm}), see~\cite[Section~5.1.5]{T}. 
The restrctions of the above stratifications to $\mathrm{Crit}(f_y)$
coincides with the restrictions of (\ref{theta:strata}). 
We note that the restrictions of (\ref{theta:strata}) to $\mathrm{Crit}(f_y)$
may be decomposed into several connected components, and these components 
are distinguished in 
(\ref{KN:X}), so $N^{\pm}$ may be larger than $d$. 
Let $\lambda_i^{\pm} \colon \mathbb{C}^{\ast} \to G_y$ be
the associated cocharacter and $\mathscr{Z}_i^{\pm} \subset \mathscr{S}_i^{\pm}$
the center of $\mathscr{S}_i^{\pm}$. 
By~\cite[Lemma~5.1.9]{T}, for each $1\leq i\leq N^{\pm}$ and $1\leq j\leq m$, there is a decomposition 
\begin{align}\label{decom:Vij}
	V^{(j)}=V_1^{(j)} \oplus V_2^{(j)}, \ 1\leq j\leq m,
	\end{align}
such that $\lambda_i^{\pm}$ is the cocharacter 
\begin{align*}
	\lambda_i^{\pm}(t)=\{(\id, t^{\mp 1} \id)\}_{1\leq j\leq m} \subset 
	\prod_{j=1}^m \GL(V_1^{(j)}) \times \GL(V_2^{(j)})
	\subset G_y. 
	\end{align*}
	Note that the decompositions (\ref{decom:Vij}) depend on $i$.
We set
\begin{align}
\label{decompose:I}
I_1=(\mathcal{O}_S \to F_1) &:=
	(\mathcal{O}_S \twoheadrightarrow \mathcal{O}_C) \oplus 
	\bigoplus_{j=1}^{m}V_1^{(j)} \otimes (0 \to \mathcal{O}_{x^{(j)}}), \\
	\notag F_2 &:=\bigoplus_{j=1}^m V_2^{(j)} \otimes \mathcal{O}_{x^{(j)}}. 
	\end{align}
		Then we have the decompositions 
	\begin{align}\label{decom:IF12}
	I=I_1 \oplus (0 \to F_2), \ F=F_1 \oplus F_2. 
	\end{align}
We also set $n_2 :=\chi(F_2)=\sum_{j=1}^m \dim V_2^{(j)}$ and let
\begin{align*}
	O_{12}:= \Hom^1_S(I_1, F_2), \ 
	O_{21} := \Hom^2_S(F_2, F_1). 
	\end{align*}
	Define the widths $\eta_i^{\pm}$ of window categories by 
\begin{align*}
    \eta_i^{\pm}:=\langle \lambda_i^{\pm},  N^{\vee}_{\mathscr{S}_i^{\pm}/\mathcal{X}_y}|_{\mathscr{Z}_i^{\pm}}\rangle
    \in \mathbb{Z}. 
\end{align*}
We define the subcategories 
\begin{align*}
	\mathbb{V}_{I, y} \subset \text{MF}^{\text{gr}}(\mathcal{X}_y, f_y), \ 
	\mathbb{V}_{P, y} \subset \text{MF}^{\text{gr}}(\mathcal{X}_y, f_y)
	\end{align*}
to be consisting of objects $\mathcal{F}$ such that for each center $\mathscr{Z}_i^{\pm}$ we have 
\begin{align}\label{wt:cond}
	&\mathrm{wt}_{\lambda_i^+}(\mathcal{F}|_{\mathscr{Z}_i^+})
	\subset [ -k_I(n_2)-\dim O_{21}, -k_I(n_2)-\dim O_{21}+\eta_i^+), \\
	\notag &\mathrm{wt}_{\lambda_i^-}(\mathcal{F}|_{\mathscr{Z}_i^-})
	\subset [ k_P(n_2)-\dim O_{12}, k_P(n_2)-\dim O_{12}+\eta_i^-). 
	\end{align}
The widths $\eta^{\pm}_i$ are computed as follows: 
\begin{lemma}\label{lem:eta}
We set $n_1^{(j)}:= \dim V_1^{(j)}$ and $n_2^{(j)}:= \dim V_2^{(j)}$.  
We have the equalities
\begin{align}\label{eta+}
	\eta_i^+ &= \hom(I_1, F_2)+\hom^2(F_2, F_1)-\hom^{-1}(I_1, F_2) \\
\notag
&= \sum_{j=1}^m n_2^{(j)}(a^{(j)}+1+2n_1^{(j)}) \\
\notag	\eta_i^- &= \hom^1(F_2, F_1)+\hom^1(I_1, F_2)-\hom(F_2, F_1) \\
\notag &= \sum_{j=1}^m n_2^{(j)}(a^{(j)}+2n_1^{(j)}).
\end{align}
\end{lemma}
\begin{proof}
By the definition of Kempf-Ness stratifications, we have 
\begin{align*}
    \eta_i^{\pm}&=\langle \lambda_i^{\pm},  N^{\vee}_{\mathscr{S}_i^{\pm}/\mathcal{X}_y}|_{\mathscr{Z}_i^{\pm}}\rangle=\langle \lambda_i^{\pm}, \mathbb{L}_{\mathcal{X}_y}^{\lambda_i^{\pm}>0}\rangle \\
    &=\langle \lambda_i^{\pm}, (\Hom_S(I, F)^{\vee}+\Hom^1_S(I, F)-\Hom^{-1}_S(I, F)^{\vee})^{\lambda_i^{\pm}>0} \rangle, 
\end{align*}
where the third identity holds from the definition of $\mathcal{X}_y$ as a quotient stack. 
By substituting the decompositions (\ref{decom:IF12})
and taking account of the $\lambda_i^{\pm}$-weights, we obtain the first identities for $\eta^{\pm}_i$. 
The second identities follow by 
substituting the decompositions (\ref{decompose:I}) 
and straightforward calculations. 
\end{proof}

Recall the Koszul duality equivalence (\ref{Kos:y}). 
For $\ast \in \{I, P\}$, we define 
\begin{align}\label{Wasty}
	\mathbb{W}_{\ast, y} :=\Phi_y^{-1}(\mathbb{V}_{\ast, y}) \subset 
D^b(\widehat{\mathfrak{T}}_S(\beta, n)_y). 
	\end{align}
We have the following characterizations of the subcategories (\ref{subcat:W}) in terms 
of formal fibers:
\begin{prop}\label{prop:local}
	An object $\mathcal{E} \in D^b(\mathfrak{T}_S(\beta, n))$
	lies in $\mathbb{W}_I$ (resp. $\mathbb{W}_P)$ if and only if for any 
	closed point $y \in T_S(\beta, n)$, we have 
	\begin{align*}
		\mathcal{E}|_{\widehat{\mathfrak{T}}_S(\beta, n)_y}
		\in \mathbb{W}_{I, y}, \quad 
		(\mbox{resp.~}	\mathcal{E}|_{\widehat{\mathfrak{T}}_S(\beta, n)_y}
		\in \mathbb{W}_{P, y}). 
		\end{align*}
	\end{prop}
	\begin{proof}
	The proposition is proved in~\cite[Theorem~6.3.13]{T}. 
	It is also a 
	direct application of (\ref{PhiEy}). 
	From (\ref{htangent}) and (\ref{decom:IF12}), we have 
	\begin{align*}
	    \mathcal{H}^1(\mathbb{T}_{\mathfrak{T}_S(\beta, n)}|_{y})=
	    \Hom^1_S(I_1, F_1) \oplus \Hom^1_S(I_1, F_2) \oplus \Hom^2_S(F_2, F_1) \oplus \Hom^2_S(F_2, F_2). 
	\end{align*}
	Therefore we have 
	\begin{align*}
	  \mathcal{H}^1(\mathbb{T}_{\mathfrak{T}_S(\beta, n)}|_{y})^{\lambda_i^+>0}=O_{21}, \
	   \mathcal{H}^1(\mathbb{T}_{\mathfrak{T}_S(\beta, n)}|_{y})^{\lambda_i^->0}=O_{12}. 
	\end{align*}
	Then the proposition follows from (\ref{PhiEy}) and (\ref{wt:cond}). 
	\end{proof}

\subsubsection{Generators of window categories}
Below, let $0<\varepsilon \ll 1$ and let $k_I$ and $k_P$ be the constant functions
\[k_I(-)=0,\, k_P(-)=\varepsilon.\]
Recall that $\chi_0$ is the determinant character (\ref{dchar}), so $\chi_0=\sum_{j=1}^{m}\sum_{i=1}^{d^{(j)}}\beta_i^{(j)}$
in the notation of Subsection~\ref{subsub:const}. In this case, we have another description of 
$\mathbb{V}_{I, y}$, $\mathbb{V}_{P, y}$ as follows (see \cite[Proposition 2.6]{KoTo}, \cite[Proposition 3.13]{PT0}, and especially the ``magic window theorem" in \cite{hls} for similar statements where window categories are described by generating vector bundles): 
\begin{lemma}\label{lem:WIP}
	The subcategories $\mathbb{V}_{I, y}$, $\mathbb{V}_{P, y}$ 
	are generated by matrix factorizations whose factors are 
	generated by $\Gamma_{G_y}(\chi)\otimes \mathcal{O}_{\mathcal{X}_y}$ for a
	$G_y$-dominant weight $\chi$ of the maximal torus $T_y \subset G_y$
	satisfying 
	\begin{align}
	\label{wt:cond1}	&\chi+\rho_{G_y}+
	\delta \in \frac{3}{2}\mathrm{sum}\left[0, \beta_i^{(j)}-\beta_{i'}^{(j)}\right]
		+\mathrm{sum}\left[ -\frac{a^{(j)}}{2}\beta_i^{(j)}, \frac{a^{(j)}}{2}\beta_i^{(j)} \right]
		+\mathrm{sum}\left[-\beta_i^{(j)}, 0\right], \\
	\notag	&	\chi+\rho_{G_y}+\delta \in \frac{3}{2}\mathrm{sum}\left[0, \beta_i^{(j)}-\beta_{i'}^{(j)}\right]
		+\mathrm{sum}\left[ -\frac{a^{(j)}}{2}\beta_i^{(j)}, \frac{a^{(j)}}{2}\beta_i^{(j)} \right]
				\end{align}
	respectively, where the Minkowski sums above are after $1\leq j\leq m$, $1\leq i, i'\leq d^{(j)}$, 
	and
	the weight 
	$\delta$ is defined by 
	\begin{align*}
		\delta=-\sum_{j=1}^m \left(\frac{a^{(j)}}{2}+\varepsilon  \right)\chi_0^{(j)}, \ 
		0<\varepsilon \ll 1. 
		\end{align*}
	\end{lemma}
\begin{proof}
	We only prove the case of $\mathbb{V}_{I, y}$. 
	Let $\delta_0$ be the following $G_y$-character
	\begin{align*}
	    \delta_0:=\det \Hom^1_S(I, F)=-\sum_{j=1}^m a^{(j)}\chi_0^{(j)},
	\end{align*}
	where the second equality follows from (\ref{HomIF}).
	In particular, we have $\delta=\delta_0/2-\varepsilon \chi_0$. 
		From the computation in~\cite[Theorem~6.3.18, (6.3.37)]{T},
	the condition (\ref{wt:cond}) 
	is equivalent to 
	\begin{align}\label{equiv:wt}
		\mathrm{wt}_{\lambda_i^+}(\mathcal{F}|_{\mathscr{Z}_i^+})
	\subset \left[ -\frac{1}{2}\eta_i^+ -\frac{1}{2}\langle \lambda_i^+, \delta_0 \rangle
	+\frac{1}{2}n_2, \frac{1}{2}\eta_i^+ -\frac{1}{2}\langle \lambda_i^+, \delta_0 \rangle
	+\frac{1}{2}n_2 \right). 
	\end{align}
	In fact, by Lemma~\ref{lem:eta} and 
	the decompositions (\ref{decom:IF12}) at each $\mathscr{Z}_i^+$, 
	we have 
	\begin{align*}
	    &\frac{1}{2}\eta_i^++\frac{1}{2}\langle \delta_0, \lambda_i^+ \rangle-\frac{1}{2}n_2 \\
	    &=\frac{1}{2}(\hom(I_1, F_2)+\hom^2(F_2, F_1)-\hom^{-1}(I_1, F_2)) \\
	    &\qquad \qquad \qquad \qquad +\frac{1}{2}(\hom^2(F_2, F_1)-\hom^1(I_1, F_2))-\frac{1}{2}n_2 \\
	    &=\hom^2(F_2, F_1)+\frac{1}{2}\chi(I_1, F_2)-\frac{1}{2}n_2 =\dim O_{21}. 
	\end{align*}
	Therefore (\ref{wt:cond}) is equivalent to (\ref{equiv:wt}). 
	Since $n_2=-\langle \lambda_i^+, \chi_0 \rangle$, 
the condition (\ref{equiv:wt}) is also equivalent to 
\begin{align}\label{wt:0}
		\mathrm{wt}_{\lambda_i^+}(\mathcal{F}|_{\mathscr{Z}_i^+})
	\subset \left[ -\frac{1}{2}\eta_i^+ -\big\langle \lambda_i^+, \frac{1}{2}\delta_0+\frac{1}{2}\chi_0-\varepsilon \chi_0 \big\rangle,  \frac{1}{2}\eta_i^+-\big\langle \lambda_i^+, \frac{1}{2}\delta_0+\frac{1}{2}\chi_0-\varepsilon \chi_0 \big\rangle \right].
	\end{align}

Let $U_y$ be the following symmetric $G_y$-representation 
\begin{align*}
	U_y=H^0(\mathcal{O}_C(C)) \oplus H^1(\mathcal{O}_C(C))^{\vee} \oplus 
	\bigoplus_{j=1}^m 
	(V^{(j)})^{\oplus a^{(j)}+1} \oplus (V^{(j)\vee})^{\oplus a^{(j)}+1} \oplus 
	\mathrm{End}(V^{(j)})^{\oplus 3},
	\end{align*}
	where $G_y$ acts trivially on the first two summands and acts naturally on the other summands.
	Denote by $\mathfrak{g}_y=\bigoplus_{j=1}^m \mathrm{End}(V^{(j)})$ the Lie algebra of $G_y$. 
	For each cocharacter $\lambda \colon \mathbb{C}^{\ast} \to T_y$,
	we set
	\begin{align*}
	    \eta_{\lambda}:=\langle \lambda, U_y^{\lambda>0}-\mathfrak{g}_y^{\lambda>0}\rangle.
	\end{align*}
A straightforward computation shows that 
\begin{align}\label{lambda:L}
\eta_{\lambda_i^+}=
	\sum_{j=1}^{m}n_2^{(j)}(a^{(j)}+1+2n_1^{(j)})=\eta_i^+. 
	\end{align}
		Let $\mathbb{V}_{I, y}'$ be the triangulated subcategory of 
$\text{MF}^{\text{gr}}(\mathcal{X}_y, f_y)$ of
matrix factorizations whose factors are generated by $\Gamma_{G_y}(\chi) \otimes \mathcal{O}_{\mathcal{X}_y}$
satisfying the condition (\ref{wt:cond1}), or equivalently 
\begin{align}\label{equiv:cond}
	&\chi+\rho_{G_y}+\frac{1}{2}\delta_0+\frac{1}{2}\chi_0-\varepsilon \chi_0
	 \\
	 \notag 
	 &\in \frac{3}{2}\mathrm{sum}\left[0, \beta_i^{(j)}-\beta_{i'}^{(j)}\right]
	+\mathrm{sum}\left[-\frac{a^{(j)}+1}{2}\beta_i^{(j)}, \frac{a^{(j)}+1}{2}\beta_i^{(j)} \right].
	\end{align}
	Note that the right hand side is half of the convex hull of weights in 
	$\wedge^{\bullet}(U_y)$. 
	By~\cite[Lemma~2.9]{hls},
	for a weight $\chi$ satisfying the above condition (\ref{equiv:cond}),
	any weight $\chi'$ in $\Gamma_{G_y}(\chi)$ satisfies 
	\begin{align*}
	    \langle \lambda, \chi' \rangle \in 
	    \left[ -\frac{1}{2}\eta_{\lambda} -\big\langle \lambda, \frac{1}{2}\delta_0+\frac{1}{2}\chi_0-\varepsilon \chi_0 \big\rangle,  \frac{1}{2}\eta_{\lambda}-\big\langle \lambda, \frac{1}{2}\delta_0+\frac{1}{2}\chi_0-\varepsilon \chi_0 \big\rangle \right]
	\end{align*}
	for any cocharacter $\lambda \colon \mathbb{C}^{\ast} \to T_y$. 
	Therefore, by (\ref{lambda:L}), 
the vector bundle $\Gamma_{G_y}(\chi) \otimes \mathcal{O}_{\mathcal{X}_y}$ for a weight $\chi$ satisfying (\ref{wt:cond1}) 
satisfies the condition (\ref{wt:0}), i.e. 
$\mathbb{V}_{I, y}' \subset \mathbb{V}_{I, y}$ holds. 

It remains to show that $\mathbb{V}_{I, y}' \subset \mathbb{V}_{I, y}$
is essentially surjective. 
From the construction of $\mathbb{V}_{I, y}$, the 
``window theorem"~\cite[Theorem 2.10]{halp} (see Subsection \ref{subsection:window}) shows that the composition 
\begin{align*}
	\mathbb{V}_{I, y} \hookrightarrow 
	\text{MF}^{\text{gr}}(\mathcal{X}_y, f_y)
	\twoheadrightarrow \text{MF}^{\text{gr}}(\mathcal{X}_y^{\chi_0\mathchar`-\rm{ss}}, f_y)
	\end{align*} 
is an equivalence. 
On the other hand, by the ``magic window theorem"~\cite{hls},
the category $\text{MF}^{\text{gr}}\left(U^{\chi_0\mathchar`-\rm{ss}}_y/G_y, f_y\right)$
is generated by matrix factorizations whose factors are generated by $\Gamma_{G_y}(\chi) \otimes \mathcal{O}_{\mathcal{X}_y}$
for a weight $\chi$ satisfying (\ref{equiv:cond}).
Consider an embedding $\iota\colon \mathcal{X}_y \hookrightarrow U_y/G_y$
given by linear inclusions $(V^{(j)\vee})^{\oplus a^{(j)}} \hookrightarrow (V^{(j)\vee})^{\oplus a^{(j)}+1}$
into the first $a^{(j)}$-th factors
for each $j$. 
Then $\mathcal{X}_y^{\chi_0\mathchar`-\rm{ss}}$
is the pull-back of
$U^{\chi_0\mathchar`-\rm{ss}}_y/G_y\subset U_y/G_y$
under $\iota$, see~\cite[Remark~3.1]{PT3} or \cite[Lemma~7.10]{T}.  
In particular, $\mathrm{MF}^{\rm{gr}}(\mathcal{X}_y^{\chi_0\text{-ss}}, f_y)$
is generated by restrictions from 
$\mathrm{MF}^{\rm{gr}}(U_y^{\chi_0\text{-ss}}/G_y, f_y)$. 
So 
the composition 
\begin{align*}
	\mathbb{V}_{I, y}' \hookrightarrow 
	\text{MF}^{\text{gr}}(\mathcal{X}_y, f_y)
	\twoheadrightarrow \text{MF}^{\text{gr}}(\mathcal{X}_y^{\chi_0\mathchar`-\rm{ss}}, f_y)
	\end{align*}
is essentially surjective, hence $\mathbb{V}_{I, y}' \subset \mathbb{V}_{I, y}$ is essentially 
surjective. 
	\end{proof}

If we write $\chi=\sum_{j=1}^m\chi^{(j)}$
as in (\ref{decom:chij}), 
the conditions (\ref{wt:cond1}) are equivalent to 
	\begin{align}
\notag	&\chi^{(j)}+\rho^{(j)}+\delta^{(j)} \in \frac{3}{2}\mathrm{sum}\left[0, \beta_i^{(j)}-\beta_{i'}^{(j)}\right]
	+\mathrm{sum}\left[ -\frac{a^{(j)}}{2}\beta_i^{(j)}, \frac{a^{(j)}}{2}\beta_i^{(j)} \right]
	+\mathrm{sum}\left[-\beta_i^{(j)}, 0\right], \\
	\notag	&	\chi^{(j)}+\rho^{(j)}+\delta^{(j)} \in \frac{3}{2}\mathrm{sum}\left[0, \beta_i^{(j)}-\beta_{i'}^{(j)}\right]
	+\mathrm{sum}\left[ -\frac{a^{(j)}}{2}\beta_i^{(j)}, \frac{a^{(j)}}{2}\beta_i^{(j)} \right]
\end{align}
for each $1\leq j\leq m$ 
respectively.
Let $y^{(j)} \in T_{S}(\beta, \chi(\mathcal{O}_C)+d^{(j)})$ be the point 
corresponding to 
\begin{align}\label{y(j)}
	(\mathcal{O}_S \twoheadrightarrow \mathcal{O}_C) \oplus (V^{(j)} \otimes (0 \to \mathcal{O}_{x^{(j)}})). 
	\end{align}
Then $\mathcal{X}_y$ is written as
\begin{align*}
	\mathcal{X}_{y}=\prod_{j=1}^m \mathcal{X}_{y^{(j)}}
	\end{align*}
where the product is the fiber product 
over $H^0(\mathcal{O}_C(C)) \oplus H^1(\mathcal{O}_C(C))^{\vee}$.
Under the above product, by Lemma~\ref{lem:WIP}
we have decompositions of window categories
\begin{align}\label{prop:IP}
	\mathbb{V}_{I, y}=\boxtimes_{j=1}^m \mathbb{V}_{I, y^{(j)}}, \ 
	\mathbb{V}_{P, y}=\boxtimes_{j=1}^m \mathbb{V}_{P, y^{(j)}}. 
	\end{align}

\subsection{Semiorthogonal decomposition of DT categories}\label{sub:SODcatDT}
Recall the categorical Hall product \eqref{PS} on $D^b(\mathfrak{M}_{S}(d))$. 
There is a left action of the categorical Hall algebra on $D^b(\mathfrak{T}_S(\beta, n))$
via the stack of short exact sequences of framed sheaves, 
see~\cite[Section~7.2]{T2}. 
For $\widetilde{d}+n'=n$, 
let $\mathfrak{T}_S(\widetilde{d}, (\beta, n'))$ be the derived 
stack of objects $(\mathcal{O}_S \to F)$
in $\mathfrak{T}_S(\beta, n)$
together with $Q \subset F$
such that $Q$ is zero-dimensional with length $\overline{d}$. 
We have evaluation morphisms 
\begin{align*}
    \mathfrak{M}_S(\widetilde{d}) \times \mathfrak{T}_S(\beta, n')
    \stackrel{q}{\leftarrow} \mathfrak{T}_S(\widetilde{d}, (\beta, n'))
    \stackrel{p}{\to} \mathfrak{T}_S(\beta, n),
\end{align*}
where $q$ sends $(\mathcal{O}_S \to F)$ together with $Q\subset F$
to $(Q, (\mathcal{O}_S \to F/Q))$. 
The morphism $q$ is quasi-smooth and $p$ is proper, see the 
argument of~\cite[Section~7.1]{T2}. 	
The above mentioned left action is 
	given by the functor 
\begin{align}\label{actionMNOP}
	\ast=p_{\ast}q^{\ast} \colon 
	D^b(\mathfrak{M}_{S}(\widetilde{d})) \otimes D^b(\mathfrak{T}_S(\beta, n')) \to 
	D^b(\mathfrak{T}_{S}(\beta, n)). 
	\end{align}
Its composition gives a functor
\begin{align}\label{catH}
	\ast \colon 
		D^b(\mathfrak{M}_{S}(d_1))\otimes \cdots \otimes
		D^b(\mathfrak{M}_{S}(d_k)) \otimes D^b(\mathfrak{T}_{S}(\beta, n'))
		\to D^b(\mathfrak{T}_S(\beta, n)) 
		\end{align}
		where $d_1+\cdots+d_k+n'=n$. 
The main result of this paper,
Theorem \ref{thm:main}, follows from the following result via the window theorem: 
\begin{thm}\label{thm:windowdec}
	The categorical Hall product (\ref{catH}) induces
	a fully-faithful functor 
	\begin{align}\label{catH2}
		\ast \colon \mathbb{T}_X(d_1)_{v_1} \otimes \cdots \otimes 
		\mathbb{T}_X(d_k)_{v_k} \otimes \mathbb{W}_P \to 
		\mathbb{W}_I
		\end{align}
	for $-1<v_1/d_1<\cdots<v_k/d_k \leq 0$ such that there is a semiorthogonal 
	decomposition 
	\begin{align*}
		\mathbb{W}_I=
		\left\langle  \mathbb{T}_X(d_1)_{v_1} \otimes \cdots \otimes 
		\mathbb{T}_X(d_k)_{v_k} \otimes \mathbb{W}_P \,\Big| 
		-1<\frac{v_1}{d_1}<\cdots<\frac{v_k}{d_k} \leq 0  \right\rangle. 
		\end{align*}
		The order of the semiorthogonal summands is 
		that of $(d_i, v_i)_{i=1}^k \in R$
		given in 
		Definition~\ref{def:order}, see Subsection \ref{notation}. 
	\end{thm}
	
\begin{step}
It suffices to prove the statement formally locally on $T_S(\beta, n)$.\end{step}
	\begin{proof}
Reducing the proof of global semiorthogonal decompositions to the proof of formal local semiorthogonal decompositions is standard (in the presence of right adjoints), see for example \cite[Lemma 6.3.12]{T}, \cite[Theorem 4.5]{TodDK}, \cite[Theorem 5.16]{T3}. 

We first note that 
the functor (\ref{catH}) is given by the correspondence 
\begin{align*}
    \prod_{i=1}^k \mathfrak{M}_S(d_i) \times 
    \mathfrak{T}_S(\beta, n') \leftarrow \mathfrak{T}_S(d_{\bullet}, (\beta, n')) \to 
    \mathfrak{T}_S(\beta, n)
\end{align*}
where the middle term is the derived stack of 
objects $(\mathcal{O}_S \to F)$ in $\mathfrak{T}_S(\beta, n)$
together with a filtration 
$0=Q_0 \subset Q_1 \subset \cdots \subset Q_k \subset F$
such that each $Q_i/Q_{i-1}$ is zero-dimensional of length $d_i$. 
Consider a closed point $y \in T_S(\beta, n)$ and recall that it corresponds to a complex \eqref{rep:dsum}:
\[I=(\mathcal{O}_S \to F)=(\mathcal{O}_S \twoheadrightarrow \mathcal{O}_C) \oplus \bigoplus_{j=1}^m 
	V^{(j)} \otimes (0 \to \mathcal{O}_{x^{(j)}}).\]
There are Cartesian squares, see~\cite[Lemma~6.4]{Totheta}
\begin{align*}
    \xymatrix{
    \coprod_{(p_1, \ldots, p_k, y')}\left(\prod_{i=1}^k \widehat{\mathfrak{M}}_S(d_i)_{p_i}\right)
    \times \widehat{\mathfrak{T}}_S(\beta, n')_{y'} \ar[d]\ar@{}[rd]|\square & 
    \widehat{\mathfrak{T}}_S(d_{\bullet}, (\beta, n'))_{y} \ar[r] \ar[d] \ar[l]\ar@{}[rd]|\square & 
    \widehat{\mathfrak{T}}_S(\beta, n)_{y} \ar[d] \\
     \prod_{i=1}^k \mathfrak{M}_S(d_i) \times 
    \mathfrak{T}_S(\beta, n') & \ar[l] \ar[r] \mathfrak{T}_S(d_{\bullet}, (\beta, n')) & 
    \mathfrak{T}_S(\beta, n)
    }
\end{align*}
In the above, the union after $(p_1, \ldots, p_k, y')$ consists of the union after points 
in the fiber of the following morphism at $y$:
\begin{align}\label{dsummap}
		\oplus \colon 
		\mathrm{Sym}^{d_1}(S) \times \cdots \times 
		\mathrm{Sym}^{d_k}(S) \times 
		T_{S}(\beta, n') \to T_S(\beta, n). 
		\end{align}
Note that each point $p_i$ is written as $p_i=\sum_{j=1}^m d_i^{(j)}x^{(j)}$
for a decomposition $d_{\bullet}=d_{\bullet}^{(1)}+\cdots+d_{\bullet}^{(m)}$
and $y'$ is uniquely determined by the above decomposition 
together with the condition $p_1 \oplus \cdots \oplus p_k \oplus y'=y$. 
 	Therefore the functor (\ref{catH}) pulls back to the functor 
	\begin{align}\label{func:formal}
		\bigoplus_{\begin{subarray}{c}
				d_{\bullet}=d_{\bullet}^{(1)}+\cdots +d_{\bullet}^{(m)} \\
				p_i=\sum_{j} d_i^{(j)}x^{(j)}
		\end{subarray}}D^b(\widehat{\mathfrak{M}}_{S}(d_1)_{p_1}) \otimes \cdots \otimes 
	D^b(\widehat{\mathfrak{M}}_{S}(d_k)_{p_k}) &\otimes D^b(\widehat{\mathfrak{T}}_{S}(\beta, n')_{y'}) \\
	&\notag \to D^b(\widehat{\mathfrak{T}}_S(\beta, n)_{y})
		\end{align}
which 
commutes with (\ref{catH}) under 
		restriction functors to formal fibers. 

		In the next steps, we show that the
		functor (\ref{func:formal}) restricts to the 
fully-faithful functor 	
	\begin{align}\label{ast:NW}
		\ast
		\colon \bigoplus_{\begin{subarray}{c}
				d_{\bullet}=d_{\bullet}^{(1)}+\cdots +d_{\bullet}^{(m)} \\
				p_i=\sum_{j} d_i^{(j)}x^{(j)}
			\end{subarray}}
		\mathbb{T}_X(d_1)_{v_1, p_1}\otimes \cdots \otimes 
		\mathbb{T}_X(d_k)_{v_k, p_k} \otimes \mathbb{W}_{P, y'} 
		\to \mathbb{W}_{I, y}
		\end{align}
		and that the images of these functors form a semiorthogonal decomposition for 
	$-1<v_1/d_1<\cdots<v_k/d_k \leq 0$, with the order of $R$ 
	from Definition~\ref{def:order}. 
	
	We now briefly explain that the
	above formal local statement proves
	the theorem, following the latter half of the 
		proof of~\cite[Theorem~5.16]{T3}. 
	First, if the above statement is true, then by Proposition~\ref{prop:local}
the functor (\ref{catH}) restricts to the functor (\ref{catH2}). 
Let $R' \subset R$ be the set of tuplets $(d_i, v_i)_{i=1}^k \in R$
which appear in the statement of the theorem. 
Note that $R'$ is a finite set. For $A=(d_i, v_i)_{i=1}^k\in R'$,
we denote the functor \eqref{catH2} by $\Phi_A$.
We next explain
that $\Phi_A$ has a right adjoint $\Psi_A$.
For a fixed $v$, the functor
(\ref{actionMNOP}) restricted to 
$D^b(\mathfrak{M}_{S}(\widetilde{d}))_v \otimes 
D^b(\mathfrak{T}_S(\beta, n'))$
admits a right adjoint, see the 
proof of~\cite[Proposition~6.7]{Totheta}.
Therefore,
by Lemma~\ref{lemma:adjoint} and
Theorem~\ref{SODsurface}, 
the functor $\Phi_A$
		has a right adjoint $\Psi_A$
		compatible with base change over 
		the good moduli space $T_S(\beta, n)$. 
	
	To show that $\Phi_A$ is fully faithful, let $E$ be an object in the source of $\Phi_A$ and let $C$ be the cone of the natural map 
		\[E\to \Psi_A\Phi_A(E)\to C\to E[1].\] The cone $C$ is trivial formally locally on $T_S(\beta, n)$
		by the assumption that (\ref{ast:NW}) is fully-faithful. 
		Thus it is trivial by \cite[Lemma 6.5]{T3} and so $\Phi_A$ is fully faithful. The argument for orthogonality is similar. 
		
		To show generation, let $E$ be in $\mathbb{W}_I$. Let $\mathscr{C}$ be the subcategory of $\mathbb{W}_I$ generated by the images of the functors $\Phi_A$ for $A\in R'$ and let $\mathscr{D}$ be the right complement of $\mathscr{C}$ in $\mathbb{W}_I$.
		Let $B\in R'$ be maximal and let 
		\[\Phi_B\Psi_B(E)\to E\to E'\to \Phi_B\Psi_B(E)[1].\] We repeat this procedure for a maximal $B'\in R'\setminus \{B\}$ and so on. 
		In the end, we obtain a distinguished triangle 
		\[\overline{E}\to E\to \widetilde{E}\to \overline{E}[1]\] with $\overline{E}\in \mathscr{C}$ and $\widetilde{E}\in\mathscr{D}$. By the formal local assumption, the object $\widetilde{E}$ is trivial formally locally on $T_S(\beta, n)$, thus it is zero by \cite[Lemma 6.5]{T3}, and this implies that $\mathscr{C}=\mathbb{W}_I$.
		\end{proof}
\begin{step}
    Let $y\in T_S(\beta, n)$. 
    We show the functors \eqref{ast:NW} are fully faithful and that their images form a semiorthogonal decomposition for 
	\[-1<\frac{v_1}{d_1}<\cdots<\frac{v_k}{d_k}
	\leq 0.\]
\end{step}
\begin{proof}
Recall that $y$ corresponds to a complex \eqref{rep:dsum}:
\[I=(\mathcal{O}_S \to F)=(\mathcal{O}_S \twoheadrightarrow \mathcal{O}_C) \oplus \bigoplus_{j=1}^m 
	V^{(j)} \otimes (0 \to \mathcal{O}_{x^{(j)}}).\]
	Let us fix $(d_1, \ldots, d_k)$ and decompositions 
	$d_{\bullet}=d_{\bullet}^{(1)}+ \cdots + d_{\bullet}^{(m)}$
	such that $\dim V^{(j)} \geq d_1^{(j)}+\cdots+d_k^{(j)}$ for all $j$. 
	Let 
	\begin{align}\label{decom:Vj}
		V^{(j)}=V_1^{(j)} \oplus \cdots \oplus V_k^{(j)} \oplus V_{k+1}^{(j)}
		\end{align}
	be decompositions such that $\dim V_{i}^{(j)}=d_i^{(j)}$. Let $d^{(j)}_{k+1}:=\dim V^{(j)}_{k+1}$.
	Consider the cocharacter $\lambda \colon 
	\mathbb{C}^{\ast} \to G_y$ which
	acts on $V_i^{(j)}$ with weight $k+1-i$. 
	Then the functor (\ref{func:formal}) on the component
	corresponding to $d_{\bullet}=d_{\bullet}^{(1)}+\cdots+d_{\bullet}^{(m)}$
	is induced by the diagram of attracting loci with respect to $\lambda$. 
	From the compatibility of Hall products and Koszul duality, 
	we have the following commutative diagram (see~\cite[Proposition~3.1]{P2})
	\begin{align}\label{com:MF}
		\xymatrix{
\bigotimes_{i=1}^k D^b(\widehat{\mathfrak{M}}_{S}(d_i)_{p_i}) \otimes 
D^b(\widehat{\mathfrak{T}}_{S}(\beta, n')_{y'}) \ar[r]^-{\ast} \ar[d] & 	D^b(\widehat{\mathfrak{T}}_{S}(\beta, n)_{y}) \ar[d] \\
\bigotimes_{i=1}^k \text{MF}^{\text{gr}}(\mathcal{X}(d_i)_{p_i}, w_{p_i})
\otimes \text{MF}^{\text{gr}}(\mathcal{X}_{y'}, w_{y'})
\ar[r]^-{\ast} & 	\text{MF}^{\text{gr}}(\mathcal{X}_{y}, w_{y}). 	
}		
		\end{align}
	In the diagram above, the upper horizontal map is \eqref{catH}, the lower horizontal map is induced as in \eqref{Hallproductquiver}, \eqref{cat:Hall:M},
	the right vertical arrow is a Koszul duality equivalence (\ref{Kos:y}), and the 
	left vertical arrow is the composition of Koszul duality equivalences (\ref{Kos:y}), (\ref{Kos:p})
	with the tensor product of the line bundle 
	\begin{align}\label{line:IF}
		\det(\Hom^1_S(I, F)^{\lambda >0})^{\vee}
		\in \mathrm{Pic}(BG_y^{\lambda}). 
		\end{align}
		Note that we have 
		\begin{align*}
		    G_y^{\lambda}=\prod_{j=1}^m \prod_{i=1}^{k+1} \GL(V_i^{(j)}). 
		\end{align*}
	By (\ref{HomIF}) and (\ref{decom:Vj}),
	we have 
	\begin{align*}
	    \left(\Hom^1(I, F)^{\lambda>0}\right)^{\vee}=
	    \bigoplus_{j=1}^m \bigoplus_{i>i'}\Hom(V_{i'}^{(j)}, V_i^{(j)}). 
	\end{align*}
	Then 
	the $G_y^{\lambda}$-character (\ref{line:IF}) is calculated as 
	\begin{align*}
	\bigotimes_{j=1}^m 
			\bigotimes_{i=1}^k \det(V_i^{(j)})^{-\sum_{i<i'\leq k}d_{i'}^{(j)}+\sum_{i>i'}d_{i'}^{(j)}-d_{k+1}}
			\otimes \det (V_{k+1}^{(j)})^{\sum_{1\leq i\leq k}d_i^{(j)}}. 
		\end{align*}
	Therefore the left vertical arrow in (\ref{com:MF}) induces the equivalence 
	\begin{align}\notag
		\bigotimes_{i=1}^k 	\mathbb{T}_X(d_i)_{v_i, p_i}
		\otimes \mathbb{W}_{P, y'} &=\boxtimes_{j=1}^m 
		\bigotimes_{i=1}^k \mathbb{T}_X(d_i^{(j)})_{v_i^{(j)},p_i^{(j)}}
		\otimes \mathbb{W}_{P, y^{'(j)}} \\
		&\stackrel{\sim}{\to} \label{bproduct}
		\boxtimes_{j=1}^m 
		\bigotimes_{i=1}^k \mathbb{S}^{\rm{gr}}(d_i^{(j)})_{w_i^{(j)}, p_i^{(j)}}
		\otimes (\mathbb{V}_{P, y^{'(j)}}\otimes \chi_0^{\sum_{i=1}^k d_i^{(j)}}). 
		\end{align}
		Here, the point $y^{'(j)}$ is defined as in (\ref{y(j)}) for $y'$, 
		$v_i^{(j)}$ is determined by 
		$v_i/d_i=v_i^{(j)}/d_i^{(j)}$, $w_{i}^{(j)}$ is given by 
	\begin{align}\label{wij}
		w_i^{(j)}:=
		v_i^{(j)}+d_i^{(j)}\left(-\sum_{i<i'\leq k}d_{i'}^{(j)}+\sum_{i>i'}d_{i'}^{(j)}-d_{k+1}^{(j)}\right),
		\end{align}
		and $\chi_0$ is the determinant character of $V^{j}_{k+1}$. Observe that the relation \eqref{wij} is the same as the relation \eqref{vwtransform}. 
		
	Then by setting $\mu=-(a^{(j)}/2+\varepsilon)$ 
and $0<\varepsilon \ll 1$ for each $j$ in Theorem~\ref{thm:locDTPT}, 
we obtain the semiorthogonal decomposition 
	\begin{align*}
	\mathbb{V}_{I, y^{(j)}}=\left\langle 	\bigotimes_{i=1}^k \mathbb{S}^{\rm{gr}}(d_i^{(j)})_{w_i^{(j)},p_i^{(j)}}
	\otimes (\mathbb{V}_{P, y^{'(j)}}\otimes \chi_0^{\sum_{i=1}^k d_i^{(j)}})  \,\Big|
	-1<\frac{v_1^{(j)}}{d_1^{(j)}}<\cdots<\frac{v_k^{(j)}}{d_k^{(j)}} \leq 0\right\rangle. 
	\end{align*}
	The twist by $\chi_0^{\sum_{i=1}^k d_i^{(j)}}$ of $\mathbb{V}_{P, y^{'(j)}}$ is due to the definition of $\delta'$ from Theorem~\ref{thm:locDTPT}.
By taking the box-product over $1\leq j\leq m$ and using the diagram (\ref{com:MF}), 
we obtain the semiorthogonal decomposition 
\begin{align}\label{SOD:WI}
	\mathbb{W}_{I, y}=\left\langle 	\bigotimes_{i=1}^k \mathbb{T}_X(d_i)_{v_i, p_i}
\otimes \mathbb{W}_{P, y'}\,\Big|
-1<\frac{v_1}{d_1}<\cdots<\frac{v_k}{d_k} \leq 0, d_{\bullet}=d_{\bullet}^{(1)}+\cdots+d_{\bullet}^{(m)}\right\rangle. 		\end{align}
	
	In the next step, we show that the 
	order of the above semiorthogonal decomposition 
	coincides with the order of $R$ induced by $O\subset R\times R$. 
		Moreover, we 
	show that
	for a fixed $d_{\bullet}$, 
the two summands for the decompositions 
\begin{align}\label{decom:dbu}
d_{\bullet}=d_{\bullet}^{(1)}+\cdots+d_{\bullet}^{(m)}
=d_{\bullet}^{'(1)}+\cdots+d_{\bullet}^{'(m)}
\end{align}
are mutually 
orthogonal. 
Therefore we conclude that 
the functors (\ref{ast:NW}) are fully-faithful and 
their essential images form a semiorthogonal decomposition.
\end{proof}
\begin{step}
The order of the semiorthogonal 
decomposition (\ref{SOD:WI})
is compatible with the order of the set $R$. 
Moreover, 
the two summands in (\ref{SOD:WI})
for the decompositions (\ref{decom:dbu}) are
mutually orthogonal. 
\end{step}
\begin{proof}
Suppose that $A=(d_i, v_i)_{i=1}^k$
and $A'=(d_i', v_i')_{i=1}^{k'}$
are elements of $R$. 
First assume that 
 either 
$\sum_{i=1}^k v_i>\sum_{i=1}^{k'} v_i'$, 
or $\sum_{i=1}^k v_i=\sum_{i=1}^{k'}v_i'$
and $\sum_{i=1}^k d_i<\sum_{i=1}^{k'}d_i'$.
Let us take decompositions 
\begin{align*}
	(d_{\bullet}, v_{\bullet})
	&=(d_{\bullet}^{(1)}, v_{\bullet}^{(1)})+\cdots+(d_{\bullet}^{(m)}, v_{\bullet}^{(m)}), \\
		(d_{\bullet}', v_{\bullet}')
	&=(d_{\bullet}^{'(1)}, v_{\bullet}^{'(1)})+\cdots +
	(d_{\bullet}^{'(m)}, v_{\bullet}^{'(m)})
	\end{align*}
such that $v_i/d_i=v_i^{(j)}/d_i^{(j)}$ and $v_i'/d_i'=v_i^{'(j)}/d_i^{'(j)}$.
We set 
\begin{align*}
	(d^{(j)}, v^{(j)})&=(d_1^{(j)}, v_1^{(j)})+\cdots+(d_k^{(j)}, v_k^{(j)}), \\
	(d^{'(j)}, v^{'(j)})&=(d_1^{'(j)}, v_1^{'(j)})+\cdots+(d_k^{'(j)}, v_k^{'(j)}). 
	\end{align*}
Then $\sum_{j=1}^m (d^{(j)}, v^{(j)})=\sum_{i=1}^k (d_i, v_i)$, and 
similarly for $A'$. 
Therefore 
there exists $1\leq j\leq m$
such that $v^{(j)}>v^{'(j)}$ 
or $v^{(j)}=v^{'(j)}$ and 
$d^{(j)}<d^{'(j)}$. 
From the order of the semiorthogonal decomposition in Theorem~\ref{thm:locDTPT}
and the box-product description (\ref{bproduct}), we obtain the semiorthogonality of the corresponding 
semiorthogonal summands in (\ref{SOD:WI}).

Next assume that 
$\sum_{i=1}^k (d_i, v_i)=\sum_{i=1}^{k'}(d_i', v_i')$. 
Then we have 
$\sum_{j=1}^m (d^{(j)}, v^{(j)})=\sum_{j=1}^m(d^{'(j)}, v^{'(j)})$. 
If $(d^{(j)}, v^{(j)})_{j=1}^m \neq 
(d^{'(j)}, v^{'(j)})_{j=1}^m$, then there exist 
$j_1, j_2$ such that 
$v^{(j_1)}>v^{'(j_1)}$, $v^{(j_2)}<v^{'(j_2)}$, or 
$v^{(j)}=v^{'(j)}$ for all $j$ and 
$d^{(j_1)}>d^{'(j_1)}$, $d^{(j_2)}<d^{'(j_2)}$. 
In either case, we have the mutual orthogonality of the 
corresponding summands in (\ref{SOD:WI})
from the order of the semiorthogonal decomposition in Theorem~\ref{thm:locDTPT}
and the box-product description (\ref{bproduct}). 

We finally assume that $(d^{(j)}, v^{(j)})_{j=1}^m = 
(d^{'(j)}, v^{'(j)})_{j=1}^m$. 
Let $\widetilde{d}=d_1+\cdots+d_k$
and $\widetilde{p}=\sum_{j=1}^m d^{(j)}x^{(j)} \in \mathrm{Sym}^{\widetilde{d}}(S)$. 
The semiorthogonal decomposition \eqref{eqSODsurface}
induces the semiorthogonal decomposition over the formal 
fibers at $\widetilde{p}$: 
\begin{align}\label{eqSODsurfacep}
	D^b(\widehat{\mathfrak{M}}_{S} (\widetilde{d})_{\widetilde{p}})=\left\langle \bigoplus_{d_{\bullet}^{(\ast)}}
				\bigotimes_{i=1}^k 
	\mathbb{T}_X(d_i)_{w_i, p_i} \,\Big|\, 
	\frac{w_1}{d_1}<\cdots<\frac{w_k}{d_k}, d_1+\cdots+d_k=\widetilde{d}\right\rangle. 		\end{align}
In the above, for a fixed decompositions $d_1+\cdots+d_k=\widetilde{d}$, 
the direct sum consists of decompositions 
$d_{\bullet}=d_{\bullet}^{(1)}+\cdots+d_{\bullet}^{(m)}$
satisfying $d^{(\ast)}=d_1^{(\ast)}+\cdots+d_k^{(\ast)}$
and $p_i=\sum_{j=1}^m d_i^{(j)}x^{(j)}$. 
Then the corresponding summands to $A$ and $A'$ in \eqref{SOD:WI} can be also obtained by taking the product with $\mathbb{W}_{P, y'}$ of the corresponding summands in \eqref{eqSODsurfacep}, and thus the order in \eqref{SOD:WI} is compatible with the set $R$ in the third case in Definition~\ref{def:order}. 
The above argument for $A=A'$ also implies that the 
two summands in (\ref{SOD:WI})
for the decompositions (\ref{decom:dbu}) are
mutually orthogonal. 
		\end{proof}



Now we can prove Theorem~\ref{thm:main}: 
\begin{proof}[Proof of Theorem \ref{thm:main}]
The claim follows from Theorems \ref{thmtodawindow} and \ref{thm:windowdec}. 
\end{proof}

\begin{remark}\label{remark:reducedcurve}
It is possible to obtain a version of Theorem \ref{thm:main} for an arbitrary class $\beta\in H_2(S,\mathbb{Z})$ if we restrict to the locus of reduced curves.
The definitions of the DT/PT categories for reduced curves are similar to Definition~\ref{def:DTPTcat}.
We consider the (finite type) open substack $\mathfrak{T}^{\mathrm{red}}_S(\beta, n)\subset \mathfrak{T}_S(\beta, n)$ of pairs $(F,s)$ such that $F$ has reduced support. 
The classical truncation $\mathcal{T}^{\mathrm{red}}_S(\beta, n)$ of 
$\mathfrak{T}^{\mathrm{red}}_S(\beta, n)$ admits a good moduli space $\mathrm{Chow}^{\beta}(S)^{\mathrm{red}} \times \mathrm{Sym}^d(S)$ as in \eqref{mor:TS}, where $\mathrm{Chow}^{\beta}(S)^{\mathrm{red}}$ is the locus of 
reduced divisors in $S$ of class $\beta$.

By the argument in \cite[Lemma 5.5.4]{T}, we have that 
\begin{align*}
\mathcal{T}^{\mathrm{red}}_X(\beta, n):=t_0(\Omega_{\mathfrak{T}^{\mathrm{red}}_S(\beta, n)}[-1])
\end{align*}
is the moduli stack of pairs $(\mathcal{O}_X \stackrel{s}{\to} E)$, 
where $E \in \Coh_{\leq 1}(X)$ has compact support, $s$ 
has at most zero-dimensional cokernel
and the support of $\tau_{\ast}E$ is reduced. 
Then there are open immersions of the DT/PT spaces 
with reduced supports over $S$:
\[ I^{\mathrm{red}}_X(\beta, n) \subset t_0(\Omega_{\mathfrak{T}_S^\mathrm{red}(\beta, n)}[-1])
    \subset P^{\mathrm{red}}_X(\beta, n). \]
Let $\mathcal{Z}_I^{\mathrm{red}}, \mathcal{Z}_P^{\mathrm{red}}$ be the complements of the above open immersions. Define 
\begin{align*}
		&\mathcal{DT}^{\mathrm{red}}_X(\beta, n):=
		D^b(\mathfrak{T}^{\mathrm{red}}_S(\beta, n))/\mathcal{C}_{\mathcal{Z}_I^{\mathrm{red}}}, \\
		&\mathcal{PT}^{\mathrm{red}}_X(\beta, n):=
		D^b(\mathfrak{T}^{\mathrm{red}}_S(\beta, n))/\mathcal{C}_{\mathcal{Z}_P^{\mathrm{red}}}. 
		\end{align*}
		By the same argument used to prove Theorem \ref{thm:main}, we obtain that there is a semiorthogonal decomposition
			\begin{align*}
			\mathcal{DT}^{\mathrm{red}}_X(\beta, n)
			=
			\left\langle  \bigotimes_{i=1}^k\mathbb{T}_X(d_i)_{v_i}\otimes \mathcal{PT}^{\mathrm{red}}_{X}(\beta, n') \,\Big|
			-1<\frac{v_1}{d_1}<\cdots<\frac{v_k}{d_k} \leq 0  \right\rangle,
		\end{align*}
where the right hand side is after all partitions $d_1+\cdots+d_k+n'=n$. 
\end{remark}

Below we give proofs of corollaries of Theorem~\ref{thm:main}:

\begin{proof}[Proof of Corollary \ref{MacMahonsurface}]
The claim follows from Theorem \ref{thm:main} for $\beta=0$.
\end{proof}

\begin{proof}[Proof of Corollary \ref{cor410}]

By taking the Grothendieck groups of the categories in Theorem \ref{thm:main}, we have that
\begin{equation}\label{decompKK1}
K\left(\mathcal{DT}_X(\beta, n)\right)
			=\bigoplus 
			K\left(\bigotimes_{i=1}^k\mathbb{T}_X(d_i)_{v_i}\otimes \mathcal{PT}_{X}(\beta, n')\right),
			\end{equation}
			where the sum on the right hand side is after all partitions $d_1+\cdots+d_k+n'=n$ and all weights
$-1<v_1/d_1<\cdots<v_k/d_k \leq 0 $. 

Fix $n'\leq n$.
Consider the semiorthogonal decomposition obtained by taking the product of the categories in \eqref{SODMacMahonsurface} with $\mathcal{PT}_X(\beta, n')$. By taking the Grothendieck group of the categories appearing in this semiorthogonal decomposition, we obtain that
\begin{equation}\label{decompKK2}
K\left(\mathcal{DT}_{X}(0, n-n')\otimes \mathcal{PT}_{X}(\beta, n')\right)=\bigoplus
K\left(\bigotimes_{i=1}^k\mathbb{T}_X(d_i)_{v_i}\otimes \mathcal{PT}_{X}(\beta, n')\right)
,
\end{equation}
where the sum on right hand side is after all partitions $d_1+\cdots+d_k=n-n'$ and all weights
$-1<v_1/d_1<\cdots<v_k/d_k \leq 0 $. 
The claim follows from \eqref{decompKK1} and \eqref{decompKK2}.
\end{proof}

\begin{remark}\label{sub45:DTPT}
The Künneth isomorphism does not hold for the right hand side of \eqref{eq:cor410} for a general surface $S$. 
To obtain more refined versions of the isomorphism \eqref{eq:cor410} in particular cases, we may proceed as in \cite{PT3} and consider localized equivariant K-theory for DT and PT spaces of a toric surface $S$.
Recall the notation  $\mathbb{K}:=K_0(BT)$, $\mathbb{F}:=\mathrm{Frac}\,\mathbb{K}$, and $V_\mathbb{F}:=V\otimes_\mathbb{K}\mathbb{F}$ for $V$ a $\mathbb{K}$-module.
In this case, there is an isomorphism of 
$\mathbb{F}$-vector spaces
\[K_T(\mathcal{DT}_X(\beta, n))_\mathbb{F}\cong \bigoplus_{n'\geq 0}K_T(\mathcal{DT}_{X}(0, n-n'))_\mathbb{F}\otimes_\mathbb{F} K_T(\mathcal{PT}_{X}(\beta, n'))_\mathbb{F}.\]
The same proof as for \cite[Corollary 4.7]{PT3}
applies here. Observe that the dimensions of $K_T(\mathcal{DT}_{X}(0, n-n'))_\mathbb{F}$ are known from the formula \eqref{dtseriestoric}. 
\end{remark}

\begin{remark}\label{sub:catMM}
We revisit the computation of DT invariants of points from Subsection \ref{sub:KDT}. 
Assume that $S$ is a toric surface.
By Theorem \ref{SODsurface}, Proposition \ref{prop38}, and Corollary \ref{cor39}, 
there is a Künneth isomorphism
\[K_T\left(\bigotimes_{i=1}^k\mathbb{T}_X(d_i)_{v_i}\right)_{\mathbb{F}}\cong \bigotimes_{i=1}^kK_T(\mathbb{T}_X(d_i)_{v_i})_\mathbb{F},\] where the tensor product on the right hand side is over $\mathbb{F}$. 
Using the same argument used to prove \cite[Corollary 4.13]{PT0} (based on Theorem \ref{MacMahonsurface} and Theorem \ref{thmtoric}), we obtain the following equality
\begin{equation}\label{dtseriestoric}
\sum_{d\geq 0}\dim_\mathbb{F} K_T(\mathcal{DT}_X(0, d))_\mathbb{F}\cdot q^d=\prod_{d\geq 1}(1-q^d)^{-d\chi_c(S)},
\end{equation}
compare with \eqref{wallcrossingDT} and \eqref{wallcrossingDT2}.
\end{remark}

	\bibliographystyle{amsalpha}
\bibliography{math}

\providecommand{\bysame}{\leavevmode\hbox to3em{\hrulefill}\thinspace}
\providecommand{\MR}{\relax\ifhmode\unskip\space\fi MR }
\providecommand{\MRhref}[2]{%
  \href{http://www.ams.org/mathscinet-getitem?mr=#1}{#2}
}
\providecommand{\href}[2]{#2}
\begin{thebibliography}{BBBBJ15}

\bibitem[AG15]{AG}
D.~Arinkin and D.~Gaitsgory, \emph{Singular support of coherent sheaves and the
  geometric {L}anglands conjecture}, Selecta Math. (N.S.) \textbf{21} (2015),
  no.~1, 1--199.

\bibitem[AHD20]{AHR}
J.~Alper, J.~Hall, and D.~David, \emph{A {L}una \'{e}tale slice theorem for
  algebraic stacks}, Ann. of Math. (2) \textbf{191} (2020), no.~3, 675--738.

\bibitem[Alp13]{MR3237451}
J.~Alper, \emph{Good moduli spaces for {A}rtin stacks}, Ann. Inst. Fourier
  (Grenoble) \textbf{63} (2013), no.~6, 2349--2402.

\bibitem[AN10]{MR2573635}
M.~Aprodu and J.~Nagel, \emph{Koszul cohomology and algebraic geometry},
  University Lecture Series, vol.~52, American Mathematical Society,
  Providence, RI, 2010.

\bibitem[BBBBJ15]{BBBJ}
O.~Ben-Bassat, C.~Brav, V.~Bussi, and D.~Joyce, \emph{A '{D}arboux {T}heorem'
  for shifted symplectic structures on derived {A}rtin stacks, with
  applications}, Geom.~Topol.~ \textbf{19} (2015), 1287--1359.

\bibitem[BF08]{BF}
K.~Behrend and B.~Fantechi, \emph{Symmetric obstruction theories and {H}ilbert
  schemes of points on threefolds}, Algebra Number Theory \textbf{2} (2008),
  313--345.

\bibitem[BFK19]{MR3895631}
M.~Ballard, D.~Favero, and L.~Katzarkov, \emph{Variation of geometric invariant
  theory quotients and derived categories}, J. Reine Angew. Math. \textbf{746}
  (2019), 235--303.

\bibitem[Bla16]{Blanc}
A.~Blanc, \emph{Topological {K}-theory of complex noncommutative spaces},
  Compos. Math. \textbf{152} (2016), no.~3, 489--555.

\bibitem[Bri11]{Br}
T.~Bridgeland, \emph{Hall algebras and curve-counting invariants}, J. Amer.
  Math. Soc. \textbf{24} (2011), no.~4, 969--998.

\bibitem[Dav]{Dav}
B.~Davison, \emph{The integrality conjecture and the cohomology of
  preprojective stacks}, arXiv:1602.02110.

\bibitem[DM20]{DM}
B.~Davison and S.~Meinhardt, \emph{Cohomological {D}onaldson-{T}homas theory of
  a quiver with potential and quantum enveloping algebras}, Invent. Math.
  \textbf{221} (2020), no.~3, 777--871.

\bibitem[Efi18]{Eff}
A.~I. Efimov, \emph{Cyclic homology of categories of matrix factorizations},
  Int. Math. Res. Not. IMRN (2018), no.~12, 3834--3869.

\bibitem[Gö90]{MR1032930}
L.~Göttsche, \emph{The {B}etti numbers of the {H}ilbert scheme of points on a
  smooth projective surface}, Math. Ann. \textbf{286} (1990), no.~1-3,
  193--207.

\bibitem[Hir17]{Hirano}
Y.~Hirano, \emph{Derived {K}n\"orrer periodicity and {O}rlov's theorem for
  gauged {L}andau-{G}inzburg models}, Compos. Math. \textbf{153} (2017), no.~5,
  973--1007.

\bibitem[HLa]{HalpK32}
D.~Halpern-Leistner, \emph{Derived {$\Theta$}-stratifications and the
  {$D$}-equivalence conjecture}, arXiv:2010.01127.

\bibitem[HLb]{Halpinstab}
\bysame, \emph{On the structure of instability in moduli theory},
  arXiv:1411.0627.

\bibitem[HL15]{halp}
\bysame, \emph{The derived category of a {GIT} quotient}, J. Amer. Math. Soc.
  \textbf{28} (2015), no.~3, 871--912.

\bibitem[HLS20]{hls}
D.~Halpern-Leistner and S.~V. Sam, \emph{Combinatorial constructions of derived
  equivalences}, J. Amer. Math. Soc. \textbf{33} (2020), no.~3, 735--773.

\bibitem[Isi13]{I}
M.~U. Isik, \emph{Equivalence of the derived category of a variety with a
  singularity category}, Int. Math. Res. Not. IMRN (2013), no.~12, 2787--2808.

\bibitem[Joy06]{Joy1}
D.~Joyce, \emph{Configurations in abelian categories {I}. {B}asic properties
  and moduli stack}, Advances in Math \textbf{203} (2006), 194--255.

\bibitem[Joy07a]{Joy2}
\bysame, \emph{Configurations in abelian categories {I}\hspace{-.1em}{I}.
  {R}ingel-{H}all algebras}, Advances in Math \textbf{210} (2007), 635--706.

\bibitem[Joy07b]{Joy3}
D.~Joyce, \emph{Configurations in abelian categories
  {I}\hspace{-.1em}{I}\hspace{-.1em}{I}. {S}tability conditions and
  identities}, Advances in Math \textbf{215} (2007), 153--219.

\bibitem[Joy08]{Joy4}
\bysame, \emph{Configurations in abelian categories {I}\hspace{-.1em}{V}.
  {I}nvariants and changing stability conditions}, Advances in Math
  \textbf{217} (2008), 125--204.

\bibitem[JS12]{JS}
D.~Joyce and Y.~Song, \emph{A theory of generalized {D}onaldson-{T}homas
  invariants}, Mem. Amer. Math. Soc. \textbf{217} (2012), no.~1020, iv+199.

\bibitem[KS]{K-S}
M.~Kontsevich and Y.~Soibelman, \emph{Stability structures, motivic
  {D}onaldson-{T}homas invariants and cluster transformations},
  arXiv:0811.2435.

\bibitem[KT21]{KoTo}
N.~Koseki and Y.~Toda, \emph{Derived categories of {T}haddeus pair moduli
  spaces via d-critical flips}, Adv. Math. \textbf{391} (2021), Paper No.
  107965, 55.

\bibitem[MNOP06]{MNOP}
D.~Maulik, N.~Nekrasov, A.~Okounkov, and R.~Pandharipande,
  \emph{Gromov-{W}itten theory and {D}onaldson-{T}homas theory. {I}},
  Compositio.~Math \textbf{142} (2006), 1263--1285.

\bibitem[P{\u{a}}da]{P0}
T.~P{\u{a}}durariu, \emph{Categorical and {K}-theoretic {H}all algebras for
  quivers with potential}, arXiv:2107.13642, to appear in
  J.~Inst.~Math.~Jussieu.

\bibitem[P{\u{a}}db]{P2}
\bysame, \emph{Generators for categorical {H}all algebras of surfaces},
  arXiv:2106.05176, to appear in Math.~Z.

\bibitem[PS]{PoSa}
M.~Porta and F.~Sala, \emph{Two dimensional categorified {H}all algebras},
  arXiv:1903.07253, to appear in J. Eur. Math. Soc.

\bibitem[PTa]{PT0}
T.~P\u{a}durariu and Y.~Toda, \emph{Categorical and {K}-theoretic
  {D}onaldson-{T}homas theory of $\mathbb{C}^3$ (part {I})}, arXiv:2207.01899.

\bibitem[PTb]{PT1}
\bysame, \emph{Categorical and {K}-theoretic {D}onaldson-{T}homas theory of
  $\mathbb{C}^3$ (part {II})}, arXiv:2209.05920.

\bibitem[PTc]{PT3}
\bysame, \emph{The local categorical {DT}/{PT} correspondence}, to appear.

\bibitem[PT09]{MR2545686}
R.~Pandharipande and R.~P. Thomas, \emph{Curve counting via stable pairs in the
  derived category}, Invent. Math. \textbf{178} (2009), no.~2, 407--447.

\bibitem[PT14]{MR3221298}
\bysame, \emph{13/2 ways of counting curves}, Moduli spaces, London Math. Soc.
  Lecture Note Ser., vol. 411, Cambridge Univ. Press, Cambridge, 2014,
  pp.~282--333. \MR{3221298}

\bibitem[PV11]{PoVa3}
A.~Polishchuk and A.~Vaintrob, \emph{Matrix factorizations and singularity
  categories for stacks}, Annales de l'Institut Fourier \textbf{61} (2011),
  no.~7, 2609--2642.

\bibitem[Seg11]{MR2795327}
E.~Segal, \emph{Equivalence between {GIT} quotients of {L}andau-{G}inzburg
  {B}-models}, Comm. Math. Phys. \textbf{304} (2011), no.~2, 411--432.

\bibitem[ST11]{MR2869309}
J.~Stoppa and R.~P. Thomas, \emph{Hilbert schemes and stable pairs: {GIT} and
  derived category wall crossings}, Bull. Soc. Math. France \textbf{139}
  (2011), no.~3, 297--339.

\bibitem[{\v S}VdB17]{SVdB}
{\v S}.~{\v S}penko and M.~{V}an~den Bergh, \emph{Non-commutative resolutions
  of quotient singularities for reductive groups}, Invent. Math. \textbf{210}
  (2017), no.~1, 3--67.

\bibitem[Tak94]{MR1297442}
Y.~Takeda, \emph{Localization theorem in equivariant algebraic {$K$}-theory},
  J. Pure Appl. Algebra \textbf{96} (1994), no.~1, 73--80.

\bibitem[Toda]{T}
Y.~Toda, \emph{Categorical {D}onaldson-{T}homas theory for local surfaces},
  arXiv:1907.09076.

\bibitem[Todb]{T4}
\bysame, \emph{Categorical {D}onaldson-{T}homas theory for local surfaces:
  $\mathbb{Z}/2$-periodic version}, arXiv:2106.05493, to appear in Int. Math.
  Res. Not. IMRN.

\bibitem[Todc]{T3}
\bysame, \emph{Categorical wall-crossing formula for {D}onaldson-{T}homas
  theory on the resolved conifold}, arXiv:2109.07064, to appear in Geometry and
  Topology.

\bibitem[Todd]{Totheta}
\bysame, \emph{Semiorthogonal decompositions for categorical
  {D}onaldson-{T}homas theory via {$\Theta$}-stratifications},
  arXiv:2106.05496.

\bibitem[Tod10]{MR2669709}
\bysame, \emph{Curve counting theories via stable objects {I}. {DT}/{PT}
  correspondence}, J. Amer. Math. Soc. \textbf{23} (2010), no.~4, 1119--1157.

\bibitem[Tod18]{Todstack}
\bysame, \emph{Moduli stacks of semistable sheaves and representations of
  {E}xt-quivers}, Geom. Topol. \textbf{22} (2018), no.~5, 3083--3144.

\bibitem[Tod20a]{Thall}
\bysame, \emph{Hall algebras in the derived category and higher-rank {DT}
  invariants}, Algebr. Geom. \textbf{7} (2020), no.~3, 240--262.

\bibitem[Tod20b]{T2}
\bysame, \emph{Hall-type algebras for categorical {D}onaldson-{T}homas theories
  on local surfaces}, Selecta Math. (N.S.) \textbf{26} (2020), no.~4, 64.

\bibitem[Tod21]{TodDK}
\bysame, \emph{Semiorthogonal decompositions of stable pair moduli spaces via
  d-critical flips}, J. Eur. Math. Soc. (JEMS) \textbf{23} (2021), no.~5,
  1675--1725.

\end{thebibliography}

\textsc{\small Tudor P\u adurariu: Columbia University, 
Mathematics Hall, 2990 Broadway, New York, NY 10027, USA.}\\
\textit{\small E-mail address:} \texttt{\small tgp2109@columbia.edu}\\

\textsc{\small Yukinobu Toda: Kavli Institute for the Physics and Mathematics of the Universe (WPI), University of Tokyo, 5-1-5 Kashiwanoha, Kashiwa, 277-8583, Japan.}\\
\textit{\small E-mail address:} \texttt{\small yukinobu.toda@ipmu.jp}\\

\end{document}